\documentclass[11pt, oneside]{article}   	
\usepackage{geometry}                		
\usepackage{graphicx}				
\usepackage{amssymb}
\usepackage{amsmath}
\usepackage{amsthm}

\usepackage{enumerate} 

\newcommand{\diag}{\mathop{\mathrm{diag}}}
\newcommand{\sgn}{\mathop{\mathrm{sgn}}} 
\newcommand{\ac}{\mathrm{ac}} 
\newcommand{\rad}{\mathrm{rad}}
\newcommand{\re}{\mathop{\mathrm{Re}}} 
\newcommand{\im}{\mathop{\mathrm{Im}}} 
\newcommand{\tr}{\mathop{\mathrm{Tr}}}
\newcommand{\supp}{\mathop{\mathrm{supp}}}
\newcommand{\intt}{\mathop{\mathrm{int}}}
\DeclareMathOperator{\slim}{s-lim}
\DeclareMathOperator{\resp}{resp}
\newcommand{\N}{\mathbb{N}} 
\newcommand{\Z}{\mathbb{Z}}
\newcommand{\R}{\mathbb{R}} 
\newcommand{\C}{\mathbb{C}} 

\newcommand{\hess}{\mathop{\mathrm{Hess}}}

\numberwithin{equation}{section}

\theoremstyle{plain}
\newtheorem{thm}{Theorem}[section]
\newtheorem{proposition}[thm]{Proposition}
\newtheorem{lemma}[thm]{Lemma} 
\newtheorem{corollary}[thm]{Corollary}

\theoremstyle{definition} 

\newtheorem{defn}[thm]{Definition} 
\newtheorem{example}[thm]{Example}

\newtheorem{assump}[thm]{Assumption}

 \newtheorem{remark}[thm]{Remark}

 \newtheorem*{remarks*}{Remarks}
\newtheorem*{remark*}{Remark}

\title{Orthonormal Strichartz estimate for dispersive equations with potentials}
\author{Akitoshi Hoshiya\thanks{Graduate School of Mathematical Sciences, The University of Tokyo, 3-8-1 Komaba, Meguro-ku, Tokyo 153-8914, Japan \\
 Email address: hoshiya@ms.u-tokyo.ac.jp}}

\begin{document}
\maketitle

\begin{abstract}
In this paper we prove the orthonormal Strichartz estimates for the higher order and fractional Schr\"odinger, wave, Klein-Gordon and Dirac equations with potentials. As in the case of the Schr\"odinger operator, the proofs are based on the smooth perturbation theory by T. Kato. However, for the Klein-Gordon and Dirac equations, we also use a method of the microlocal analysis in order to prove the estimates for wider range of admissible pairs. As applications we prove the global existence of a solution to the higher order or fractional Hartree equation with potentials which describes the dynamics of infinitely many particles. We also give a local existence result for the semi-relativistic Hartree equation with electromagnetic potentials. As another application, the refined Strichartz estimates are proved for higher order and fractional Schr\"odinger, wave and Klein-Gordon equations.
\end{abstract}

\section{Introduction}\label{23110221}
In this paper we give some continuation results of the author's previous research \cite{Ho}. We prove the orthonormal Strichartz estimates for the higher order and fractional Schr\"odinger, wave, Klein-Gordon and Dirac equations with potentials. The orthonormal Strichartz estimates were first proved by Frank-Lewin-Lieb-Seiringer \cite{FLLS} for the free Schr\"odinger equation and they were generalized by \cite{FS}, \cite{BHLNS}, \cite{BKS2} and \cite{BLN2}. They are the following inequalities: 
\begin{align}
\left \| \sum_{n=0}^ \infty{\nu_n|e^{it\Delta}f_n|^2} \right\|_{L^p_t L^q_x} \lesssim \| \nu_n\|_{\ell^\alpha}.
\label{231115}
\end{align}
Here  $\{f_n\}$ is an orthonormal system in $L^2(\mathbb{R}^d)$ and $p, q, \alpha \in [1, \infty]$ satisfy some conditions (see \cite{BHLNS}). In the case of the higher order or fractional free Schr\"odinger equation, the orthonormal Strichartz estimates were proved in \cite{BLN} and \cite{BKS2}.  Concerning the wave or Klein-Gordon equations, see \cite{FS}, \cite{BLN} and \cite{BKS1}. On the other hand, the orthonormal Strichartz estimates for PDEs with potentials are less known. As far as the author knows, there exist a few results on the special operators including the special Hermite operator, $(k, a)$-generalized Laguerre operator and Dunkl operator (see \cite{Ms} and \cite{MS}) but general potentials are treated only in \cite{Ho}. In \cite{Ho}, the Schr\"odinger operator with very short range or inverse square type potentials are treated (also including the magnetic Schr\"odinger operator). We extend the method used in \cite{Ho} and apply it to various equations. Our first result concerns the higher order and fractional Schr\"odinger equations. To state the result, we collect some notations: $O=(0, 0), A_{d/2} =(\frac{d-1}{2(d+1)}, \frac{d}{2(d+1)}), C= (1/2, 0), D=(0, 1/2), E_{d/2} =(\frac{d-2}{2d}, 1/2)$.
The first result is including the case where $V$ is a very short range potential. If $m=1$, similar results are proved in \cite{Ho}. However, if $m>1$, assumptions on the spectrum of $H$ are stronger since $H$ may have positive eigenvalues which do not appear if $m=1$.
\begin{thm} \label{1251816}
Assume $m < d/2$, $m \in \N$, $|V(x)| \lesssim \langle x \rangle^{-2m-\epsilon}$, $H= (-\Delta)^m +V$ does not have positive eigenvalues and zero is neither an eigenvalue nor a resonance of $H$. 
\begin{enumerate}
\item \label{1251556}
If $(1/r, 1/q) \in \intt (OA_{d/2} C)$ and $d/2 = d/r + 2m/q$, we have
\begin{align} \label{1251819}
\left\| \sum_{n=0}^ \infty{\nu_n|e^{-itH}P_{ac}(H)f_n|^2} \right\|_{L^{q/2}_t L^{r/2}_x} \lesssim \| \nu_n\|_{\ell^\beta}.
\end{align}
for $\beta$ satisfying $d/{2\beta} = 1/q +d/r$, $\{ f_n\}$: all orthonormal systems in $L^2$.
\item \label{1251812}
If $(1/r, 1/q) \in \intt (ODE_{d/2} A_{d/2})$, $d/2 = d/r + 2m/q$ and $d \ge3$, we have (\ref{1251819}) for $\beta < q/2$.
\end{enumerate}
\end{thm}
The above theorem is a generalization of the ordinary Strichartz estimates for $d/{2m}$ admissible pairs. We may also prove the estimates for $d/2$ admissible pairs as follows.
\begin{corollary} \label{1271749}
Let $H$ be as in Theorem \ref{1251816}. We assume $m< \frac{d+2}{4}$, $\langle H_0 u, u \rangle \lesssim \langle Hu, u \rangle$, $\|H^{1/2} u\|_{\frac{2d}{d-2}} \lesssim \|H^{1/2} _0 u\|_{\frac{2d}{d-2}}$ and $H$ does not have positive eigenvalues. Then, for $(q, r)$: $d/2$ admissible pair satisfying $d/r + 2m/q <d$, we have
\begin{enumerate}
\item[(1)] For $r$ satisfying $2 \le r < \frac{2(d+1)}{d-1}$, $\{f_n\}$: all orthonormal systems in $\dot{H}^s$, $s= d(1-m)(1/2 - 1/r)$, $\beta = \frac{2r}{r+2}$,
\[ \left\| \sum_{n=0}^ \infty{\nu_n|e^{-itH} f_n|^2} \right\|_{L^{q/2} _t L^{r/2} _x} \lesssim \| \nu_n\|_{\ell^\beta}\]
holds. 
\item[(2)] For $d \ge 3$, $r$ satisfying $\frac{2(d+1)}{d-1} \le r \le \frac{2d}{d-2}$, $\{f_n\}$: all orthonormal systems in $\dot{H}^s$, $s= d(1-m)(1/2 - 1/r)$, the same estimate as (1) holds if $1 \le \beta <q/2$.
\end{enumerate}
\end{corollary}
See also Example \ref{1272012} for a sufficient condition of $H$. We also prove the estimates for the higher order or fractional operators with Hardy type potentials (including the critical Hardy potential) in Section 2. As an application of the orthonormal Strichartz estimates, we prove the global existence of a solution to the higher order or fractional Hartree equation with potentials for infinitely many particles. Our target is
\begin{align}
\left\{
\begin{array}{l}
i\partial_t \gamma=[H+w*\rho_{\gamma},\gamma] \\
\gamma(0)=\gamma_0
\end{array}
\right.
\tag{H}\label{2311152204}
\end{align}
Here, $\gamma$ is an operator-valued function and $\rho_{\gamma}$ is its density function. For functions $f$ and $g$, $f*g$ denotes the convolution with respect to the space variable. In the following theorem, see Section 2 for the definition of the Schatten spaces $\mathfrak{S}^{\alpha}$.
\begin{thm} \label{1281958}
Assume $1< \sigma< d/2 $, $2\sigma/q =d(1/2-1/r)$, $r \in [2, \frac{2d}{d-2\sigma})$. Define $\beta^*$ by the relation $d/2\beta^* =1/q +d/r$. We also assume $\beta \in [1, \beta^*)$ if $r< \frac{d+2\sigma-1}{d-1}$ and $\beta \in [1, q/2)$ if $r> \frac{d+2\sigma-1}{d-1}$. For the potential and initial data, we assume $w \in L^{(r/2)'}$ and $\gamma_0 \in \mathfrak{S}^{\beta}$ is self-adjoint. Let $H$ be as in Theorem \ref{1251816} or Theorem \ref{1251821} and $\sigma (H) = \sigma _{\ac} (H)$ (if $H$ is as in Theorem \ref{1251816}, we assume the above condition substituting $m$ for $\sigma$). Then there exists a unique global solution $\gamma \in C(\R: \mathfrak{S}^{\beta})$ to (\ref{2311152204}) such that $\rho_{\gamma} \in L^{q/2} _{{loc}, t} L^{r/2} _x$.
\end{thm}
\begin{remark} \label{12131551}
The equation (\ref{2311152204}) is an infinitely many particle version of the corresponding $N$-particle system:
\[
\left\{
\begin{array}{l}
i\partial_t u_j=Hu_j+w*\left(\displaystyle \sum_{k=1}^{k=N} |u_k|^2\right)u_j \\
u_j(0)=\ u_{0,j}
\end{array}
\right.
\]     
for $j=1,2,\dots,N$. See \cite{LS} for the physical background on infinitely many particle systems.
\end{remark}
As another application of the orthonormal Strichartz estimates, we prove the refined Strichartz estimates for higher order operators. We assume the following conditions.
\begin{assump} \label{129026}
In addition to $m< d/2$ and $m \in \N$, 
\begin{itemize}

\item $|V(x)| \lesssim \langle x \rangle^{-\beta}$ for some $\beta >d+3$ if $d$ is odd and for some $\beta > d+4$ if $d$ is even.

\item $\|\langle x \rangle^{1+\epsilon} V\|_{H^{\epsilon}} < \infty$ for some $\epsilon >0$ if $d=4m-1$.

\item If $d>4m-1$, $\|\mathcal{F} (\langle x \rangle^{\sigma} V)\|_{\frac{d-1-\delta}{d-2m-\delta}} < \infty$ holds for some $\sigma > \frac{2d-4m}{d-1-\delta}$ and sufficiently small $\delta>0$.

\item $H=(-\Delta)^m +V$ has neither positive eigenvalues nor zero resonances.
\end{itemize}
\end{assump}
Our result is an improvement of the ordinary Strichartz estimates (see e.g. \cite{MY2}) in terms of the Besov spaces.
\begin{thm} \label{1291939}
Assume $H$ is as in Assumption \ref{129026}, $\langle Hu, u \rangle \approx \langle H_0 u, u\rangle$, $\|H^{1/2} u\|_{\frac{2d}{d-2}} \lesssim \|H^{1/2} _0 u\|_{\frac{2d}{d-2}}$. If $r, q, \beta, s$ satisfy the conditions in Corollary \ref{1271749}, we have
\[\|e^{-itH} u\|_{L^q _t L^r _x} \lesssim \|u\|_{\dot{B}^s _{2,2\beta}}.\]
\end{thm}
For the sufficient condition of $\|H^{1/2} u\|_{\frac{2d}{d-2}} \lesssim \|H^{1/2} _0 u\|_{\frac{2d}{d-2}}$, see Example \ref{1272012}. We also give a result for the case of $d/{2m}$ admissible pairs in Section 2. Our next results concern the wave and Klein-Gordon equations with inverse square type potentials. In the following, $M^{p, q}$ denotes the Morrey-Campanato spaces. 
\begin{corollary}\label{12131724}
Let $H= H_0 +V$, $H_0 = -\Delta$, $V:\mathbb{R}^d \rightarrow \mathbb{R}$, $d \ge3$.
We assume $V\in X^\sigma_d :=\{ V:\mathbb{R}^d \rightarrow \mathbb{R} \mid |x|V \in M^{d,2\sigma} ,x\cdot \nabla V \in M^{d/2,\sigma}  \} $ for some $ \sigma \in (\frac{d-1}{2},\frac{d}{2})$,
and 
\begin{align*}
\langle Hf,f\rangle  \approx \| \nabla f \|^2_2,
\quad 
\langle (-\Delta-V-x\cdot \nabla V)f,f \rangle \gtrsim \| \nabla f \|^2_2 
\end{align*}
hold. Then 
\begin{align}
\left\| \sum_{n=0}^ \infty{\nu_n|e^{-it\sqrt{H}} f_n|^2} \right\|_{L^{p}_t L^{q}_x} \lesssim \| \nu_n\|_{\ell^\alpha} \label{12131753}
\end{align}
holds for $(2p, 2q)$: $\frac{d-1}{2}$ admissible pair satisfying $1\le q \le \frac{d+1}{d-1}$, $\{f_n\}$: orthonormal system in $\dot{H}^s$, $s= \frac{d+1}{2} (\frac{1}{2} -\frac{1}{2q})$, $\alpha = \frac{2q}{q+1}$.
\end{corollary}
\begin{corollary}\label{12131859}
Let $H= H_0 +V$, $H_0 = -\Delta$, $V:\mathbb{R}^d \rightarrow \mathbb{R}$, $d \ge3$.
We assume $V\in X^\sigma_d :=\{ V:\mathbb{R}^d \rightarrow \mathbb{R} \mid |x|V \in M^{d,2\sigma} ,x\cdot \nabla V \in M^{d/2,\sigma}  \} $ for some $ \sigma \in (\frac{d-1}{2},\frac{d}{2})$,
and 
\begin{align*}
\langle Hf,f\rangle  \gtrsim \| \nabla f \|^2_2,
\quad 
\langle (-\Delta-V-x\cdot \nabla V)f,f \rangle \gtrsim \| \nabla f \|^2_2 ,
\quad
\langle Vu, u\rangle \lesssim \|u\|^2 _{H^1}
\end{align*}
hold. Then we have
\begin{align}
\left\| \sum_{n=0}^ \infty{\nu_n|e^{-it\sqrt{H+1}} f_n|^2} \right\|_{L^{p}_t L^{q}_x} \lesssim \| \nu_n\|_{\ell^\alpha} \label{12131926}
\end{align}
if either of the following holds.
\begin{enumerate}
\item $(2p, 2q)$: $\frac{d-1}{2}$ admissible pair satisfying $1\le q \le \frac{d+1}{d-1}$, $\{f_n\}$: orthonormal system in $H^s$, $s= \frac{d+1}{2} (\frac{1}{2} -\frac{1}{2q})$, $\alpha = \frac{2q}{q+1}$.
\item $(2p, 2q)$: $\frac{d}{2}$ admissible pair satisfying $1\le q \le \frac{d+2}{d}$, $\{f_n\}$: orthonormal system in $H^s$, $s= \frac{d+2}{2} (\frac{1}{2} -\frac{1}{2q})$, $\alpha = \frac{2q}{q+1}$.
\end{enumerate}
\end{corollary}
We also give some results on the magnetic Schr\"odinger operator and the Schr\"odinger operator with critical inverse square potentials in Section 3.
However, the above estimates are possible only for restricted pairs $(p, q)$. For the Klein-Gordon equation, we give a result for more general admissible pairs but its proof is based on the microlocal analysis and we assume that electromagnetic potentials are smooth functions.
\begin{corollary}\label{1216058}
Assume $d \ge 3$, $H=(D+A)^2 +V$, $A \in C^{\infty} (\R^d ; \R^d)$, $V \in C^{\infty} (\R^d ; \R)$ and they satisfy
\[|\partial^{\alpha} _x A(x)| \lesssim \langle x \rangle^{-2-\epsilon -|\alpha|}, \quad |\partial^{\alpha} _x V(x)| \lesssim \langle x \rangle^{-2-\epsilon -|\alpha|} \]
for some $\epsilon >0$. Furthermore we assume $H \ge \gamma >-1$ for some $\gamma \in \R$, zero is a regular point of $H$, $\sigma (H)=\sigma_{\ac} (H)$ and $\sigma^{W} (H) +1 \gtrsim \langle \xi \rangle^{2}$. Here $\sigma^{W} (H)$ denotes the symbol of $H$ when we think of it as a Weyl quantization. Then
\[\left\| \sum_{n=0}^ \infty{\nu_n|e^{-it\sqrt{H +1}} f_n|^2} \right\|_{L^{q/2}_t L^{r/2}_x} \lesssim \| \nu_n\|_{\ell^\beta}\]
holds for $q, r, \beta, \{f_n\}$ satisfying either of the following:
\begin{enumerate}
\item $(q, r)$ is a $d/2$-admissible pair satisfying $2 \le r < \frac{2(d+1)}{d-1}$. $\{f_n\}$ is an orthonormal system in $H^s$ with $s=\frac{d+2}{2} (1/2 -1/r)$. $\beta = \frac{2r}{r+2}$.

\item $(q, r)$ is a $d/2$-admissible pair satisfying $\frac{2(d+1)}{d-1} \le r \le \frac{2d}{d-2}$. $\{f_n\}$ is an orthonormal system in $H^s$ with $s=\frac{d+2}{2} (1/2 -1/r)$. $1 \le \beta < q/2$.

\item $(q, r)$ is a $\frac{d-1}{2}$-admissible pair satisfying $2 \le r < \frac{2d}{d-2}$. $\{f_n\}$ is an orthonormal system in $H^s$ with $s=\frac{d+1}{2} (1/2 -1/r)$. $\beta = \frac{2r}{r+2}$.

\item $(q, r)$ is a $\frac{d-1}{2}$-admissible pair. We assume $r \in [6, \infty)$ if $d=3$ and $\frac{2d}{d-2} \le r \le \frac{2(d-1)}{d-3}$ if $d \ge 4$. $\{f_n\}$ is an orthonormal system in $H^s$ with $s=\frac{d+1}{2} (1/2 -1/r)$. $1 \le \beta < q/2$.
\end{enumerate}
\end{corollary}
We note that though the above results do not contain the endpoint estimates, they are reduced to the ordinary endpoint Strichartz estimates. This is because the orthonormal Strichartz estimates with $\beta =1$ are proved by applying the triangle inequality to the ordinary endpoint Strichartz estimates. In Appendix A and B, we give some examples of the endpoint Strichartz estimates for the wave and Klein-Gordon equations which can be deduced by an abstract perturbation theorem.

As an application of the above results, we also prove the local existence of a solution to the semi-relativistic Hartree equation with electromagnetic potentials for infinitely many particles. Similarly to the case of the higher order or fractional Schr\"odinger operator, we also prove the refined Strichartz estimates for both the wave and Klein-Gordon equations. See Section 3 for more details about this topics. As our final result on the orthonormal Strichartz estimates, we prove the estimates for the massive Dirac operator with potentials. See Section 3 for notations used in the following theorem.
\begin{thm} \label{12212051}
Assume $V(x)$ is a $\C$-valued $4 \times 4$ matrix satisfying $V(x) = V^* (x)$. Furthermore we assume $V \in C^{\infty} (\R^3; \C^4)$, $|\partial^{\alpha} _x V(x)| \lesssim \langle x \rangle^{-2-|\alpha|}$ for all $\alpha \in \N^{3} _{0}$ and $|\partial^{\beta} _x V(x)| \lesssim \epsilon \langle x \rangle^{-2-|\beta|}$ for all $|\beta| \le 1$ and sufficiently small $0<\epsilon$. Then $\mathcal{D} + \beta +V$ is a self-adjoint operator on $\mathcal{H}$ and it satisfies $\sigma (\mathcal{D} + \beta +V) =\sigma_{\ac} (\mathcal{D} + \beta +V) = \R \setminus (-1, 1)$ and $(\mathcal{D} + \beta +V)^2 \approx (\mathcal{D} + \beta)^2 = (1-\Delta)I_4$. Furthermore we have
\[\left\| \sum_{n=0}^{\infty} \nu_n |e^{-it(\mathcal{D} + \beta +V)} g_n|^2 \right\|_{L^{q/2} _t L^{r/2} _x} \lesssim \|\{\nu_n\}\|_{\ell^{\beta}}\]
for all orthonormal systems $\{g_n\}$ in $H^s (\R^3 ; \C^4)$. Here $q, r, \beta, s$ are as in Theorem \ref{12211346}.
\end{thm}
Since there seems to be no result on the orthonormal Strichartz estimates for the Dirac operator, even if we assume $V=0$, we first prove the estimates for the free Dirac operator including massless and massive cases in Section 3.

Though there are few results on the orthonormal Strichartz estimates for PDEs with potentials, there are many results on the ordinary Strichartz estimates. For example, see Introduction in \cite{Ho} and references therein for the Schr\"odinger operator. For the wave or Klein-Gordon equation, see Appendix A and B in this paper. There are still many references to be mentioned not only for the ordinary Strichartz estimates but also for other topics e.g. infinitely many particle systems and the refined Strichartz estimates. They are given in each sections below. 

This paper is organized as follows. In Section 2, we prove the orthonormal Strichartz estimates for the higher order or fractional Schr\"odinger operators based on an abstract perturbation method and apply them to the Hartree equations and the refined Strichartz estimates. In Section 3, we prove the estimates for the wave, Klein-Gordon and Dirac equations by using the perturbation method and the microlocal analysis and apply them to the semi-relativistic Hartree equation and the refined Strichartz estimates. In Appendix A and B, we prove the abstract perturbation theorems for the ordinary Strichartz estimates for the wave and Klein-Gordon equations respectively including the endpoint case. We also give their applications to some operators including the magnetic Schr\"odinger operator, higher order or fractional Schr\"odinger operator.
\subsubsection*{\textbf{Notations}}

\begin{itemize}
\item
For a measure space $(X, d\mu)$, $ L^p(X)$ denotes the ordinary Lebesgue space and its norm is denoted by $\| \cdot \|_p$. Similarly, $L^{p, q} (X)$ denotes the ordinary Lorentz space and its norm is denoted by $\|\cdot\|_{p, q}$.

\item
 $\mathcal{F}$ denotes the Fourier transform on $\mathcal{S}'$. Here $\mathcal{S}'$ denotes the set of all the tempered distributions.

\item
 For a Banach space $X$, $L^p_t X$ denotes the set of all the measurable functions $f : \mathbb{R} \rightarrow X$ such that $\|f\|_{L^p_t X} :=(\int_{\mathbb{R}} \|f(t)\|^p_X dt)^{1/p} < \infty$.
 
\item For $p, q \in [1, \infty]$ and $s \in \R$, $B^s _{p, q}$ ($\dot{B}^s _{p, q}$) denotes the ordinary inhomogeneous (homogeneous) Besov space. $H^{s, p}$ ($\dot{H}^{s, p}$) denotes the ordinary $L^p$-inhomogeneous (homogeneous) Sobolev space. 

\item For a self-adjoint operator $H$ and a Borel measurable function $f$, $f(H)$ is defined as $f(H)= \int_{\mathbb{R}} f(\lambda) dE(\lambda)$. Here $E(\lambda)$ is the spectral measure associated to $H$.

\item For a normed space $X$, $\mathcal{B}(X)$ denotes the set of all the bounded operators on $X$.
\end{itemize}

\section{Higher order and fractional Schr\"odinger equation}\label{2311152141}
In this section, we prove the orthonormal Strichartz estimates for the higher order Schr\"odinger operator: $H = (-\Delta)^m +V$, $m \in \N$ and the fractional Schr\"odinger opeartor: $H =(-\Delta)^{\sigma} +V$, $\sigma \in \R$. In these cases, concerning the ordinary Strichartz estimates, the following two types are known if $V$ is a very short range or Hardy type potential (\cite{MY1}, \cite{MY2}):
\begin{align}
& \|e^{-itH} P_{\ac}u\|_{L^p _t L^q _x} \lesssim \|u\|_2 \label{1251455}\\
& \||D|^{2(m -1)/p} e^{-itH} P_{\ac}u\|_{L^p _t L^q _x} \lesssim \|u\|_2 \label{1251456}.
\end{align}
Here $P_{\ac}$ denotes the projection onto the absolutely continuous subspace of $H$. (\ref{1251455}) was proved for the sharp $d/{2m}$ or $d/{2\sigma}$ admissible pairs and (\ref{1251456}) was proved for the sharp $d/2$ admissible pairs including the endpoint. Actually (\ref{1251456}) is stronger than (\ref{1251455}) by the Sobolev embedding. The proof of these estimates are based on the perturbation method by Rodnianski and Schlag \cite{RS} (see also \cite{BM}) and their assumptions on $V$ are the same in (\ref{1251455}) and (\ref{1251456}). However, in the orthonormal Strichartz estimates, our conditions on $V$ are different according to admissible pairs. In the first subsection, we consider the orthonormal Strichartz estimates for the $d/{2m}$ or $d/{2\sigma}$ admissible pair under less assumptions. In the next subsection, we consider the $d/2$ admissible case under more assumptions. We also give two applications of the orthonormal Strichartz estimates. In the third subsection, we prove the global existence of a solution to the higher order infinitely many particle system with potentials. Concerning the Schr\"odinger operator, i.e. $m=1$ , this result was proved in the author's previous research \cite{Ho} including magnetic potentials. However there seems to be no result if $m>1$ even for the local existence of a solution except for $V=0$. In the last subsection, we prove the refined Strichartz estimates in terms of the Besov spaces by using the $L^p$-boundedness of wave operators. 
\subsection{\textbf{$d/{2m}$ or $d/{2\sigma}$ admissible case}}
In addition to very short range potentials mentioned in Section 1, we can treat the Hardy type potentials including $V(x) = a |x|^{-2\sigma}$, $a > -C_{\sigma, d},   C_{\sigma, d} := \left \{\frac{2^{\sigma} \Gamma (\frac{d+2\sigma}{4})}{\Gamma (\frac{d-2\sigma}{4})}\right \}^2$. If $\sigma =1$, the same results are proved in \cite{Ho}.
\begin{thm}\label{1251821}
Assume $H= (-\Delta)^{\sigma} +V$, $1<\sigma <d/2$, $(x \cdot \nabla)^l V \in L^1 _{loc}$ for $l=0, 1, 2$, $(x \cdot \nabla)^l V$ is $(-\Delta)^{\sigma}$-form bounded, $|x|^{2\sigma} V \in L^{\infty}$ and
\begin{align}
&\langle Hu, u \rangle \gtrsim \langle (-\Delta)^{\sigma} u, u \rangle \label{1251831}\\
&\langle [H, iA]u, u \rangle \gtrsim \langle (-\Delta)^{\sigma} u, u \rangle \label{1251832} \\
& |\langle [ [H, iA], iA]u, u \rangle | \lesssim \langle [H, iA]u, u \rangle. \label{1251833}
\end{align}
Here $A = \frac{1}{2} (xD + Dx)$ denotes the generator of the dilation group.
\begin{enumerate}
\item
If $(1/r, 1/q) \in \intt (OA_{d/2} C)$, $d/2 = d/r + {2\sigma}/q$, we have
\begin{align}
\left\| \sum_{n=0}^ \infty{\nu_n|e^{-itH} f_n|^2} \right\|_{L^{q/2}_t L^{r/2}_x} \lesssim \| \nu_n\|_{\ell^\beta} \label{1251840}
\end{align}
for $\beta$ satisfying $d/{2\beta} = 1/q +d/r$, $\{f_n\}$: all orthonormal systems in $L^2$.
\item
If $(1/r, 1/q) \in \intt (ODE_{d/2} A_{d/2})$, $d/2 = d/r + {2\sigma}/q$ and $d\ge3$, we have (\ref{1251840}) for all $\beta < q/2$.
\end{enumerate}
\end{thm}
Theorem \ref{1251816} and \ref{1251821} are easy consequences of the following abstract perturbation theorem, which was also used to prove the orthonormal Strichartz estimates for the Schr\"odinger operator.
\begin{thm} [\cite{Ho}] \label{1251906}
Let $(X ,d\mu)$ be a $\sigma$-finite measure space and $\mathcal{H}:=L^2(X)$. Assume that for self-adjoint operators $H , H_0$ and densely defined closed operators $Y, Z $ on $\mathcal{H}$, $H = H_0 + V ,V=Y^*Z$ holds in the form sense, i.e.
\[D(H) \cup D(H_0) \subset D(Y) \cap D(Z) \quad \text{and}\quad  \langle Hu,v \rangle=\langle H_0 u,v\rangle+\langle Zu,Yv\rangle \] 
for all $u, v \in D(H) \cap D(H_0)$.
We also assume $Y$ is $H_0$-smooth and $ZP_{ac}(H)$ is $H$-smooth.
If we have 
\[ \left\| \sum_{n=0}^ \infty{\nu_n|e^{-itH_0}f_n|^2} \right\|_{L^p_t L^q_x} \lesssim \| \nu_n\|_{\ell^\alpha}\] for some $p, q \in [1,\infty] $, $ \alpha \in(1,\infty)$ and all orthonormal systems $\{f_n\}$ in $\mathcal{H}$,
then we also have 
\[ \left\| \sum_{n=0}^ \infty{\nu_n|e^{-itH}P_{ac}(H)f_n|^2} \right\|_{L^p_t L^q_x} \lesssim \| \nu_n\|_{\ell^\alpha}.\]
\end{thm}

\begin{proof}[Proof of Theorem \ref{1251816} and \ref{1251821}]
The orthonormal Strichartz estimates for the free Hamiltonian $H_0 =(-\Delta)^m$ or $H_0 =(-\Delta)^{\sigma}$ are proved in \cite{BLN} (Theorem 14) for the same $q, r, \beta$ as in the statement in Theorem \ref{1251816} and \ref{1251821}. By Theorem \ref{1251906}, it suffices to show that $|V|^{1/2} \sgn V$ is $H_0$-smooth and $|V|^{1/2} P_{\ac} (H)$ is $H$-smooth. However, they are proved in \cite{MY1}(Theorem 1.3) and \cite{MY2}(Corollary 3.5). Note that in Theorem \ref{1251821}, $P_{\ac} (H) = I$ holds.
\end{proof}
Next we consider the critical Hardy potential: $V(x) = -C_{\sigma, d} |x|^{-2\sigma}$. In the following statement, $P_{\rad}$ denotes the orthogonal projection onto the subspace of radial functions: $P_{\rad}: L^2(\R^d) \rightarrow L^2_{\rad}(\R^d), f \mapsto \frac{1}{|S^{d-1}|} \int_{S^{d-1}} f(|x|\theta) d\sigma(\theta)$.
\begin{thm} \label{1252218}
Assume $1<\sigma<d/2$, $H_{crit, \sigma} = (-\Delta)^{\sigma} -C_{\sigma, d} |x|^{-2\sigma}$, $q, r, \beta$ as in Theorem \ref{1251821} 1 or 2. Then
\begin{align}
 \left\| \sum_{n=0}^ \infty{\nu_n|e^{-itH_{crit, \sigma}}P_{\rad}^{\perp} f_n|^2} \right\|_{L^{q/2}_t L^{r/2}_x} \lesssim \| \nu_n\|_{\ell^\beta}
\end{align}
holds. Here $P_{\rad}^{\perp} = I - P_{\rad}$
\end{thm}
If $\sigma =1$, this result is proved in \cite{Ho}. Furthermore, the orthonormal Strichartz estimates for the radial part of $H_{crit, 1}$ are also proved if $d=3$. The proof for the radial part is based on the ground state representation. However, if $\sigma >1$, we do not know the ground state representation for $H_{crit, \sigma}$. So in these cases we have no results on the orthonormal Strichartz estimates (even the usual Strichartz estimates are not proved (see also Section 6 in \cite{MY1})).
\begin{proof}
By repeating the proof of Theorem \ref{1251906}, substituting $P_{\rad}$ for $P_{\ac} (H)$, it suffices to show that $|x|^{-\sigma} P_{\rad}$ is $H_{crit, \sigma}$-smooth. Note that $P_{\rad}$ commutes with $(-\Delta)^{\sigma}, H_{crit, \sigma}$, hence with $f((-\Delta)^{\sigma}), f(H_{crit, \sigma})$ for $f \in L^2 _{loc}$. However, this is proved in Theorem 6.1 in \cite{MY1}.
\end{proof}
Instead of the estimate for the radial part of $H_{crit, \sigma}$, we consider
\[\tilde{H}_{crit, \sigma} = \Lambda_{\sigma} (-\Delta_{\R^{2\sigma}})^{\sigma} \Lambda^{-1} _\sigma\]
on $L^2 _{\rad} (\R^d)$.
Here $\sigma \in [1, d/2) \cap \N$, $\Lambda_{\sigma} f =r^{-\frac{d-2\sigma}{2}} f$. By direct computations,  
\[\tilde{H}_{crit, \sigma} P_{\rad} = \left (-\Delta_{\R^d} - \frac{(d-2\sigma)(d+2\sigma -4)}{4|x|^2} \right)^{\sigma} P_{\rad}\] 
holds. Furthermore, for all $f \in L^2 _{loc}$, we obtain
\[f(\tilde{H}_{crit, \sigma})P_{\rad} = \Lambda_{\sigma} f((-\Delta_{\R^{2\sigma}})^{\sigma})\Lambda^{-1} _\sigma P_{\rad}.  \]
This operator was first considered in \cite{MY1} and the weak type endpoint Strichartz estimates were proved. We prove the orthonormal Strichartz estimates for $\tilde{H}_{crit, \sigma}$ under some restrictions on $d, q, r$. Note that if $\sigma =1$, $\tilde{H}_{crit, 1} P_{\rad} = H_{crit, 1} P_{\rad}$ holds, the Strichatz estimates were proved in \cite{M2} except for the endpoint and the orthonormal Strichartz estimates were proved in \cite{Ho} if $d=3$.
\begin{thm} \label{1252217}
Assume $2\sigma<d< \frac{4\sigma^2}{2\sigma -1}$, $\frac{1}{r} >\frac{\sigma}{d} (\frac{d}{2\sigma} -\frac{2\sigma}{4\sigma-1})$, $d/2 = d/r + 2\sigma/q$ holds. Then
\begin{align*}
\left\| \sum_{n=0}^ \infty{\nu_n|e^{-it\tilde{H}_{crit, \sigma}}P_{\rad} f_n|^2} \right\|_{L^{q/2}_t L^{r/2}_x} \lesssim \| \nu_n\|_{\ell^\beta}
\end{align*}
holds for $\beta$ satisfying $d/2\beta = 1/q + d/r$, $\{f_n\}$: all orthonormal systems in $L^2 (\R^d)$.
\end{thm}
To prove this theorem, we use the following lemma. If $\sigma =1$, this is proved in \cite{Ho}.
\begin{lemma} \label{1252337}
Let $d=2\sigma$, $2\sigma/2 = 2\sigma/r +2\sigma/q$, $(1/r, 1/q) \in \intt (OA_{\sigma} C)$. Here $A_{\sigma} = (\frac{2\sigma-1}{2(2\sigma+1)}, \frac{\sigma}{2\sigma+1})$. Then
\[\left\| \sum_{n=0}^ \infty{\nu_n|e^{-it(-\Delta_{2\sigma})^{\sigma}} \Lambda^{-1} _{\sigma}P_{\rad} f_n|^2} \right\|_{L^{q/2}_t L^{r/2, q/2}_x} \lesssim \| \nu_n\|_{\ell^\beta}\]
holds for $\beta$ satisfying $2\sigma/2\beta =1/q +2\sigma/r$.
\end{lemma}
\begin{proof}[Outline of the proof]
The proof of this lemma is essentially the same as Lemma 4.6 in \cite{Ho}. So we only give the outline.  By assumptions on $\sigma, r, q$, we have
\[\left\| \sum_{n=0}^ \infty{\nu_n|e^{-it(-\Delta_{2\sigma})^{\sigma}} g_n|^2} \right\|_{L^{q/2}_t L^{r/2}_x} \lesssim \| \nu_n\|_{\ell^\beta}\]
for all orthonormal systems $\{g_n\}$ in $L^2 (\R^{2\sigma})$. By setting $A(t) = e^{-it(-\Delta_{\R^{2\sigma}})^{\sigma}} \Lambda^{-1} _{\sigma}P_{\rad} : L^2 (\R^d) \rightarrow L^2 (\R^{2\sigma})$, $A_1: L^2 (\R^d) \rightarrow L^{\infty} L^2 (\R^{2\sigma}), u \mapsto A(t)u(x)$, $A_2: L^1 _t L^2 (\R^{2\sigma}) \rightarrow L^2 (\R^d), f \mapsto \int_{\R} A^*(s) f(s) ds$ and repeating the argument as Lemma 4.6 in \cite{Ho}, we obtain
\[\left\| \sum_{n=0}^ \infty{\nu_n|e^{-it(-\Delta_{2\sigma})^{\sigma}} \Lambda^{-1} _{\sigma}P_{\rad} f_n|^2} \right\|_{L^{q/2}_t L^{r/2}_x} \lesssim \| \nu_n\|_{\ell^\beta}.\]
Then by real interpolations: 
\begin{itemize}

\item $(\ell^{p_0}, \ell^{p_1})_{\theta, q} = \ell^{p, q}$. Here $1/p = (1-\theta)/p_0 + \theta/p_1$

\item $(L^{p_0}_t L^{q_0}_x , L^{p_1}_t L^{q_1}_x)_{\theta, p} = L^p _t L^{q, p}_x$. Here $1/p = (1-
\theta)/p_0 + \theta/p_1$ and $1/q = (1-\theta)/q_0 + \theta/q_1$

\end{itemize}
we have 
\[\left\| \sum_{n=0}^ \infty{\nu_n|e^{-it(-\Delta_{2\sigma})^{\sigma}} \Lambda^{-1} _{\sigma}P_{\rad} f_n|^2} \right\|_{L^{q/2}_t L^{r/2, q/2}_x} \lesssim \| \nu_n\|_{\ell^{\beta, q/2}}.\]
Since $(1/r, 1/q) \in \intt (OA_{\sigma} C)$, $1/q < \frac{\sigma}{4\sigma-1}$ holds and this implies $\beta \le q/2$. Hence
\[\| \nu_n\|_{\ell^{\beta, q/2}} \lesssim \| \nu_n\|_{\ell^{\beta}}\]
and we have the desired estimates.
\end{proof}

\begin{proof}[Proof of Theorem \ref{1252217}]
Defining $\tilde{r} :=|x|$ and changing into polar coordinates,
\begin{align}
\left\| \sum_{n=0}^ \infty{\nu_n|e^{-it\tilde{H}_{crit, \sigma}}P_{\rad} f_n|^2} \right\|_{L^{q/2}_t L^{r/2}_x} &\lesssim \left\| \tilde{r}^{\frac{2(d-2\sigma)}{r}} \sum_{n=0}^ \infty{\nu_n|e^{-it\tilde{H}_{crit, \sigma}}P_{\rad} f_n|^2} \right\|_{L^{q/2}_t L^{r/2} (\R^{2\sigma})} \notag \\
& \lesssim \left\| \tilde{r}^{\frac{2(d-2\sigma)}{r} -(d-2 \sigma)} \sum_{n=0}^ \infty{\nu_n|e^{-it(-\Delta_{\R^{2\sigma}})^{\sigma}} \Lambda^{-1} _{\sigma} P_{\rad} f_n|^2} \right\|_{L^{q/2}_t L^{r/2} (\R^{2\sigma})} \label{1261033}.
\end{align}
Set $1/\tilde{q} =1- 2/q$. Then we obtain
\begin{align}
&\left\| \tilde{r}^{\frac{2(d-2\sigma)}{r} -(d-2 \sigma)} \sum_{n=0}^ \infty{\nu_n|e^{-it(-\Delta_{\R^{2\sigma}})^{\sigma}} \Lambda^{-1} _{\sigma} P_{\rad} f_n|^2} \right\|_{L^{r/2} (\R^{2\sigma})} \notag \\
&\lesssim \|\tilde{r}^{\frac{2(d-2\sigma)}{r} -(d-2 \sigma)}\|_{\frac{2\sigma}{(d-2\sigma)(1-2/r)}, \infty} \cdot \left\| \sum_{n=0}^ \infty{\nu_n|e^{-it(-\Delta_{\R^{2\sigma}})^{\sigma}} \Lambda^{-1} _{\sigma} P_{\rad} f_n|^2} \right\|_{L^{\tilde{q}, r/2} (\R^{2\sigma})} \notag \\
& \lesssim \left\| \sum_{n=0}^ \infty{\nu_n|e^{-it(-\Delta_{\R^{2\sigma}})^{\sigma}} \Lambda^{-1} _{\sigma} P_{\rad} f_n|^2} \right\|_{L^{\tilde{q}, r/2} (\R^{2\sigma})} \label{1261031}.
\end{align}
Now we assume $r/2 \ge q/2 \Leftrightarrow r \ge \frac{2}{d} (d+2\sigma)$. Then 
\[(\ref{1261031}) \lesssim \left\| \sum_{n=0}^ \infty{\nu_n|e^{-it(-\Delta_{\R^{2\sigma}})^{\sigma}} \Lambda^{-1} _{\sigma} P_{\rad} f_n|^2} \right\|_{L^{\tilde{q}, q/2} (\R^{2\sigma})}.\]
Combining with (\ref{1261033}) and using Lemma \ref{1252337}, we have
\begin{align}
\left\| \sum_{n=0}^ \infty{\nu_n|e^{-it\tilde{H}_{crit, \sigma}}P_{\rad} f_n|^2} \right\|_{L^{q/2}_t L^{r/2}_x} &\lesssim \left\| \sum_{n=0}^ \infty{\nu_n|e^{-it(-\Delta_{\R^{2\sigma}})^{\sigma}} \Lambda^{-1} _{\sigma} P_{\rad} f_n|^2} \right\|_{L^{q/2} _t L^{\tilde{q}, q/2} (\R^{2\sigma})} \notag \\
& \lesssim \|\nu_n\|_{\ell^{\tilde{\beta}}} \label{1261101}
\end{align}
for $\tilde{\beta}$ satisfying $2\sigma/2\tilde{\beta} = 1/q +2\sigma/2\tilde{q}$.
Precisely speaking, in (\ref{1261101}), we need to assume $1/{\tilde{q}} > \frac{2\sigma-1}{4\sigma-1} \Leftrightarrow 1/r > \frac{\sigma}{d} (\frac{d}{2\sigma} -\frac{2\sigma}{4\sigma-1})$ in order to use Lemma \ref{1252337}. However $d< \frac{4\sigma^2}{2\sigma-1}$ ensures $\frac{d}{2} \frac{1}{d+2\sigma} > \frac{\sigma}{d} (\frac{d}{2\sigma} -\frac{2\sigma}{4\sigma-1})$ and if we take $1/r$ sufficiently close to $\frac{\sigma}{d} (\frac{d}{2\sigma} -\frac{2\sigma}{4\sigma-1})$, we have (\ref{1261101}). Now by direct computations, we obtain
\begin{align*}
\tilde{\beta} \ge \beta \Leftrightarrow 2(1-1/2\sigma -d/2\sigma +d/4\sigma^2)\cdot 1/r \ge (1-1/2\sigma -d/2\sigma +d/4\sigma^2) . 
\end{align*}
Since we assume $2\sigma <d< \frac{4\sigma^2}{2\sigma-1}$ and this is equivalent to $r\ge2$, we have $(1-1/2\sigma -d/2\sigma +d/4\sigma^2) <0$. Therefore $\tilde{\beta} \ge \beta$ and from (\ref{1261101}),
\begin{align*}
\left\| \sum_{n=0}^ \infty{\nu_n|e^{-it\tilde{H}_{crit, \sigma}}P_{\rad} f_n|^2} \right\|_{L^{q/2}_t L^{r/2}_x}\lesssim \|\nu_n\|_{\ell^{\tilde{\beta}}} \lesssim \|\nu_n\|_{\beta}
\end{align*} 
holds.
\end{proof}

\subsection{\textbf{$d/2$ admissible case}}
In this subsection we consider the orthonormal Strichartz estimates for the $d/2$ admissible pair, i.e. for $(q, r)$ such that $2/q = d(1/2-1/r)$. In order to explain our results, we recall the smooth perturbation theory by Kato \cite{KY}. The next theorem is one of the key lemma in the smooth perturbation theory.
\begin{thm}[\cite{KY}]\label{2311152150}

Suppose $\mathcal{H}$ is an arbitrary Hilbert space. Let $A$ be a densely defined closed operator and $H$ be a self-adjoint operator on $\mathcal{H}$. Then the following are equivalent:
\begin{enumerate}[(1)]
\item
$\|Ae^{-itH}u\|_{L^2_t \mathcal{H}} \lesssim \|u\|_\mathcal{H}$ holds for any $u \in \mathcal{H}$.
\item
$ \sup_{z\in \mathbb{C} \backslash \mathbb{R}}|\langle \im(H-z)^{-1} A^* u, A^* u \rangle| \lesssim \|u\|^2_{\mathcal{H}} $ holds.
\end{enumerate}
In particular, if $ \sup_{z\in \mathbb{C} \backslash \mathbb{R}}\| A(H-z)^{-1} A^*\|_{\mathcal{B}(\mathcal{H})} \lesssim 1$, we have (1) and (2).
\end{thm}
The above smoothing estimate is called the Kato smoothing estimate or the Kato-Yajima estimate.

\begin{defn}
If $A$ satisfies the condition in Theorem~\ref{2311152150}, we say that $A$ is $H$-smooth.
\end{defn}
We first state our abstract perturbation theorem where we need to consider derivatives.
\begin{thm} \label{1271257}
Let $\mathcal{H}$ be $L^2 (X)$ where $X$ is a $\sigma$-finite measure space. Assume that for a nonnegative self-adjoint operator $H_0$, self-adjoint operator $H$ and densely defined closed operators $Y, Z $ on $\mathcal{H}$, $H = H_0 + V ,V=Y^*Z$ holds in the form sense, i.e.
\[D(H) \cup D(H_0) \subset D(Y) \cap D(Z) \quad \text{and}\quad  \langle Hu,v \rangle=\langle H_0 u,v\rangle+\langle Zu,Yv\rangle \] 
for all $u, v \in D(H) \cap D(H_0)$. We also assume $YH^{\theta} _0$ is $H_0$-smooth and 
\[\|Ze^{-itH} H^{-\theta} _0 u\|_{L^2 _t L^2 _x} \lesssim \|u\|_2\]
holds for some $\theta \in \R$. If we have
\[ \left\| \sum_{n=0}^ \infty{\nu_n|e^{-itH_0}f_n|^2} \right\|_{L^{q/2}_t L^{r/2}_x} \lesssim \| \nu_n\|_{\ell^\alpha}\] for some $q, r \in [2,\infty] $, $ \alpha \in(1,\infty)$ and all orthonormal systems $\{f_n\}$ in $D(H^{\theta} _0)$, we also obtain
\[ \left\| \sum_{n=0}^ \infty{\nu_n|e^{-itH}f_n|^2} \right\|_{L^{q/2}_t L^{r/2}_x} \lesssim \| \nu_n\|_{\ell^\alpha}\]
for the same $q, r, \alpha$.
\end{thm}
Before proving Theorem \ref{1271257}, we recall some definitions and lemmas.
\begin{defn}[Schatten space]
Let $\mathcal{H}_1$ and $\mathcal{H}_2$ be a Hilbert space. For a compact operator $A : \mathcal{H}_1 \rightarrow \mathcal{H}_2$, $\{ \mu_n \}$: the singular value of $A$  is defined as the set of all the eigenvalues of $(A^*A)^{1/2}$. The Schatten space $\mathfrak{S}^{\alpha} (\mathcal{H}_1 \rightarrow \mathcal{H}_2)$ for $\alpha \in [1, \infty]$ is the set of all compact operators: $\mathcal{H}_1 \rightarrow \mathcal{H}_2$  such that its singular value belongs to $\ell^{\alpha}$. Its norm is defined by the $\ell^\alpha$ norm of the singular value. We also use the following notation for simplicity: $\|S\|_{\mathfrak{S}^{\alpha} (\mathcal{H}_0)} := \|S\|_{\mathfrak{S}^{\alpha} (\mathcal{H}_0 \rightarrow \mathcal{H}_0)}$.
\end{defn}

\noindent
For $A \in \mathcal{B}(L^2 (X))$, where $X$ is a $\sigma$-finite measure space, $\rho_A (x) : = k_A (x, x)$ denotes the density function of $A$. Here $k_A (x, y)$ is the integral kernel of $A$. We sometimes write $\rho (A)$.
For the proof of properties associated to the Schatten spaces, see \cite{S}, \cite{Ha}. 
We often use Hölder's inequality for the Schatten space:
\[\|ST\|_{\mathfrak{S}^{\alpha} (\mathcal{H}_0 \rightarrow \mathcal{H}_2)} \lesssim \|S\|_{\mathfrak{S}^{\alpha_1} (\mathcal{H}_1 \rightarrow \mathcal{H}_2)} \cdot \|T\|_{\mathfrak{S}^{\alpha_2} (\mathcal{H}_0 \rightarrow \mathcal{H}_1)}\]
Here $1/\alpha = 1/\alpha_1 + 1/\alpha_2$.
We use the following notations for the Schatten spaces: 
\begin{gather*}
\ \mathfrak{S}^{\alpha}_{t, x}= \mathfrak{S}^{\alpha} (L^2_{t, x}) , \mathfrak{S}^{\alpha}_{x \rightarrow {t, x}}= \mathfrak{S}^{\alpha} (L^2_x \rightarrow L^2_{t, x}) , \mathfrak{S}^{\alpha}_{t, x \rightarrow x}= \mathfrak{S}^{\alpha} (L^2_{t, x} \rightarrow L^2_x)
\end{gather*}
Here $L^2_{t,x} = L^2_t X$ for some Hilbert space $X$ and $x$ denotes the variable in $X$.
The next lemma is called the duality principle. This was first proved in \cite{FS} and changed into the following form by \cite{Ha}.
\begin{lemma}[ \cite{FS}, \cite{Ha}]\label{1113719}
Let $p, q, \alpha \in[1, \infty]$, $A :\mathbb{R} \rightarrow \mathcal{B}(L^2(X))$ be a strongly continuous function. Then the following are equivalent:

(1) For any $\gamma \in \mathfrak{S}^{\alpha} (L^2(X))$, 
\[\| \rho(A(t) \gamma A(t)^*)\|_{L^p_t L^q_x} \lesssim \| \gamma \|_{\mathfrak{S}^\alpha} .\]

(2) For any $f\in L^{2p'}_t L^{2q'}_x$, 
\[\|fA(t)\|_{\mathfrak{S}^{2\alpha'} _{x\rightarrow (t,x)}} \lesssim \|f\|_{L^{2p'}_t  L^{2q'}_x}\] 
\end{lemma}
\begin{remark}

In [FS] and [Ha], the above lemma is proved for $L^2(\mathbb{R}^d)$. However, the modification of the proof in the case of $L^2(X)$ is straightforward, thus we omit the proof here.
\end{remark}
If $(1)$ or $(2)$ in Lemma \ref{1113719} holds, by defining $\gamma := \displaystyle \sum_{n=0}^{\infty} \nu_n |f_n \rangle \langle f_n |$ for an orthonormal system $\{f_n\}$ in $L^2 (X)$, we obtain
\[\left\| \sum_{n=0}^{\infty} \nu_n |A(t)f_n|^2 \right\|_{L^p _t L^q _x} \lesssim \|\nu_n\|_{\ell^{\alpha}}.\]
Thus this lemma is useful to our setting.

We also use the following Christ-Kiselev type lemma.
\begin{lemma}[\cite{GK}, \cite{BS}, \cite{Ha}]\label{2311152210}

Let $-\infty \le a < b \le\infty$, $\alpha \in(1, \infty)$, $\mathcal{H}$ be an arbitrary Hilbert space. Let $K :(t,\tau) \mapsto K(t,\tau) \in \mathcal{B}(\mathcal{H})$ be a strongly continuous function. Assume that 
\[ \widetilde{T} g(t)=\int_a^b K(t, \tau) g(\tau) d\tau\]
defines a bounded operator on $L^2_t \mathcal{H}$ and $\widetilde{T} \in \mathfrak{S}^{\alpha}(L^2_t \mathcal{H})$.
Then 
\[Tg(t)=\int_a^t K(t, \tau) g(\tau) d\tau\]
 also satisfies $T \in \mathfrak{S}^{\alpha}(L^2_t \mathcal{H})$ and $\|T\|_{\mathfrak{S}^{\alpha} }  \lesssim \| \widetilde{T}\|_{\mathfrak{S}^{\alpha}}.$
\end{lemma}
\begin{remark}
In Theorem 3.1 of \cite{Ha}, it is assumed that $\mathcal{H}=L^2(\mathbb{R}^d)$. However, the proof in the case of $\mathcal{H}$ is just the same as the comments in Section 3.1 of \cite{Ha}. Hence we omit the proof here.
\end{remark}
\begin{proof}[Proof of Theorem \ref{1271257}]
By the Duhamel formula, we have
\begin{gather*} 
\ U_{H} = U_{H_0} -i  \Gamma_{H_0} V U_{H} \\
\ U_{H}=e^{-itH} ,U_{H_0}=e^{-itH_0}, \Gamma_{H_0} g(t) = \int_0^t e^{i(t-s)H_0} g(s) ds .
\end{gather*}
Hence
\begin{align}
fU_{H} H^{-\theta} _0 = fU_{H_0} H^{-\theta} _0  -i f \Gamma_{H_0} V U_{H}  H^{-\theta} _0 \label{1271326}
\end{align}
holds for all $f: \R \times X \rightarrow \C$. By the assumption and Lemma \ref{1113719}, we have
\begin{align}
\|fU_{H_0} H^{-\theta} _0 \|_{\mathfrak{S}^{2\alpha'}} \lesssim \|f\|_{L^{2(q/2)'} _t L^{2(r/2)'} _x}. \label{1271634}
\end{align}
For the second term in (\ref{1271326}), we have
\begin{align}
\| -i f \Gamma_{H_0} V U_{H}  H^{-\theta} _0 \|_{\mathfrak{S}^{2\alpha'}} &\lesssim \left\|-ifU_{H_0} H^{-\theta} _0 \int_{0}^{\infty} H^{\theta} _0 e^{isH_0} Y^* Z e^{-isH} H^{-\theta} _0 ds \right\|_{\mathfrak{S}^{2\alpha'}} \label{1271631}\\
& \lesssim \|f\|_{L^{2(q/2)'} _t L^{2(r/2)'} _x} \cdot \|YH^{\theta} _0 e^{-isH_0}\|_{L^2 _x \rightarrow L^2 _t L^2 _x} \cdot \|Ze^{-isH} H^{-\theta} _0\|_{L^2 _x \rightarrow L^2 _t L^2 _x} \label{1271633}\\
& \lesssim  \|f\|_{L^{2(q/2)'} _t L^{2(r/2)'} _x}. \label{1271636}
\end{align}
Here, in (\ref{1271631}), we have used Lemma \ref{2311152210}, in (\ref{1271633}), we have used (\ref{1271634}), and in (\ref{1271636}), $H_0$-smoothness of $YH^{\theta} _0$ and $\|Ze^{-itH} H^{-\theta} _0 u\|_{L^2 _t L^2 _x} \lesssim \|u\|_2$ are used. Hence we obtain
\[\|fU_{H} H^{-\theta} _0 \|_{\mathfrak{S}^{2\alpha'}} \lesssim \|f\|_{L^{2(q/2)'} _t L^{2(r/2)'} _x}. \]
By Lemma \ref{1113719}, we have the desired estimates.
\end{proof}
As an application of Theorem \ref{1271257}, first, we prove the estimates for very short range potentials.
\begin{proof}[Proof of Corollary \ref{1271749}]
Set $-\theta = \frac{d}{2m} (m-1) (1/2 - 1/r)$. Then $-\theta \in (0, 1/2]$ holds since $(q, r)$ is $d/2$ admissible. Decompose $V=v^* _1 v_2$ such that $|v_1 (x)| \lesssim \langle x \rangle^{-m+d(1-m)(1/2 -1/r)-\epsilon}$, $|v_2 (x)| \lesssim \langle x \rangle^{-m-d(1-m)(1/2 -1/r)-\epsilon}$.

\noindent (step1) In this step we prove $\|v_1 H^{\theta} _0 e^{-isH_0}\|_{L^2 _x \rightarrow L^2 _t L^2 _x} \lesssim 1$. By the complex interpolation with respect to $r$, we may assume $r=2$ or $r= \frac{2d}{d-2}$. If $r=2$, then $\theta =0$ and
\begin{align*}
\|\langle x \rangle^{-m-\epsilon} e^{-itH_0} u\|_{L^2 _t L^2 _x} \lesssim \|e^{-itH_0} u\|_{L^2 _t L^{\frac{2d}{d-2m}}} \lesssim \|u\|_2
\end{align*}
holds. Here, in the first inequality, we have used H\"older's inequality and in the second inequality, we have used the following endpoint Strichartz estimates, which were proved in \cite{MY1}:
\[\|e^{-itH_0}u\|_{L^2 _t L^{\frac{2d}{d-2m}} _x} \lesssim \|u\|_2.\]
If $r= \frac{2d}{d-2}$, then $-\theta = \frac{m-1}{2m}$ and 
\begin{align*}
\|\langle x \rangle^{-2m+1-\epsilon} |D|^{-(m-1) u}\|_2 \lesssim \||D|^{-(m-1)} u\|_{\frac{2d}{d-4m+2}} \lesssim \|u\|_{\frac{2d}{d-2m}}
\end{align*}
holds by the H\"older inequality and the Sobolev embedding. Then
\begin{align*}
\|\langle x \rangle^{-2m+1-\epsilon} |D|^{-(m-1)}  e^{-itH_0} u\|_{L^2 _t L^2 _x} \lesssim \|e^{-itH_0}u\|_{L^2 _t L^{\frac{2d}{d-2m}} _x} \lesssim \|u\|_2
\end{align*}
holds by the endpoint Strichartz estimates. Hence we obtain $\|v_1 H^{\theta} _0 e^{-isH_0}\|_{L^2 _x \rightarrow L^2 _t L^2 _x} \lesssim 1$.

\noindent (step2) In this step we prove $\|v_2 e^{-itH} H^{-\theta} _0\|_{L^2 _x \rightarrow L^2 _t L^2 _x} \lesssim 1$. By using the complex interpolation again, we may assume $r=2$ or $r=\frac{2d}{d-2}$. If $r=2$, then $\theta=0$ and
\[\|v_2 e^{-itH} u\|_{L^2 _t L^2 _x} \lesssim \|e^{-itH} u\|_{L^2 _t L^{\frac{2d}{d-2m}}} \lesssim \|u\|_2\]
holds by the H\"older inequality and the Strichartz estimates. If $r=\frac{2d}{d-2}$, then $-\theta = \frac{m-1}{2m}$ and 
\begin{align*}
\|v_2 e^{-itH} H^{-\theta} _0 u\|_{L^2 _t L^2 _x} &\lesssim \|e^{-itH} H^{-\theta} _0 u\|_{L^2 _t L^{\frac{2d}{d-2}} _x} \\
& = \|H^{-\theta} H^{\theta} _0 H^{-\theta} _0 e^{-itH} H^{\theta} H^{-\theta} _0 u\|_{L^2 _t L^{\frac{2d}{d-2}} _x}  \\
& \lesssim \|H^{-\theta} _0 e^{-itH} H^{\theta} H^{-\theta} _0 u\|_{L^2 _t L^{\frac{2d}{d-2}} _x}
\end{align*}
holds. Here, in the first line, we have used H\"older's inequality and in the last line, we have used
$\|H^{-\theta} u\|_{\frac{2d}{d-2}} \lesssim \|H^{-\theta} _0 u\|_{\frac{2d}{d-2}}$, which follows from $\|H^{1/2} u\|_{\frac{2d}{d-2}} \lesssim \|H^{1/2} _0 u\|_{\frac{2d}{d-2}}$, $-\theta \in (0, 1/2]$ and the complex interpolation. Then by the endpoint Strichartz estimates:
\[ \||D|^{m-1} e^{-itH} u\|_{L^2 _t L^{\frac{2d}{d-2}} _x} \lesssim \|u\|_2,\]
we have
\begin{align*}
\|H^{-\theta} _0 e^{-itH} H^{\theta} H^{-\theta} _0 u\|_{L^2 _t L^{\frac{2d}{d-2}} _x} \lesssim \|H^{\theta} H^{-\theta} _0 u\|_2 \lesssim \|u\|_2.
\end{align*}
In the last inequality, we have used $\langle H_0 u, u \rangle \lesssim \langle Hu, u \rangle$ and the complex interpolation. Hence we obtain $\|v_2 e^{-itH} H^{-\theta} _0\|_{L^2 _x \rightarrow L^2 _t L^2 _x} \lesssim 1$.

Finally, by step1, step2, Theorem \ref{1271257}, and the free orthonormal Strichartz estimates:
\[ \left\| \sum_{n=0}^ \infty{\nu_n|e^{-itH_0} f_n|^2} \right\|_{L^{q/2} _t L^{r/2} _x} \lesssim \| \nu_n\|_{\ell^\beta}\]
for the same $q, r, \beta, s$, which were proved in \cite{BLN}, we obtain the desired estimates.
\end{proof}
\begin{example} \label{1272012}
Assume $|V(x)| \lesssim \langle x \rangle^{-2m-\epsilon}$, $V(x) \ge 0$, $-x \cdot \nabla V(x) \ge 0$, $m< \frac{d-2}{4}$ holds. If $H$ does not have zero resonances,  then assumptions in Corollary \ref{1271749} hold.
\end{example}
\begin{proof}
Since $V$ is a repulsive potential, $\sigma (H) = \sigma _{\ac} (H)$ holds (This is an easy consequence of the virial identity. See \cite{FSWY} for its proof.). The assumption $V(x) \ge 0$ implies $\langle H_0 u, u \rangle \lesssim \langle Hu, u \rangle$. Then it suffices to show $\|H u\|_{\frac{2d}{d-2}} \lesssim \|H _0 u\|_{\frac{2d}{d-2}}$ since by the complex interpolation we obtain $\|H^{1/2} u\|_{\frac{2d}{d-2}} \lesssim \|H^{1/2} _0 u\|_{\frac{2d}{d-2}}$. By the H\"older inequality and the Sobolev embedding, we have the desired estimates:
\[\|Vu\|_{\frac{2d}{d-2}} \lesssim \|V\|_{d/2m} \|u\|_{\frac{2d}{d-4m-2}} \lesssim \||D|^{2m} u\|_{\frac{2d}{d-2}} = \|H_0 u\|_{\frac{2d}{d-2}}.\]
\end{proof}
Next we consider the Hardy type potentials.
\begin{corollary} \label{1272125}
Let $H$ be as in Theorem \ref{1251821} and assume $\sigma< \frac{d+2}{4}$, $\|H^{1/2} u\|_{\frac{2d}{d-2}, 2} \lesssim \|H^{1/2} _0 u\|_{\frac{2d}{d-2}, 2}$. Here $\|\cdot\|_{{p, q}}:= \|\cdot\|_{L^{p, q}}$ denotes the norm of the Lorentz spaces. Then, for $(q, r)$: $d/2$ admissible pair satisfying $d/r + 2\sigma/q <d$, we have
\begin{enumerate}
\item[(1)] For $r$ satisfying $2 \le r < \frac{2(d+1)}{d-1}$, $\{f_n\}$: all orthonormal systems in $\dot{H}^s$, $s= d(1-\sigma)(1/2 - 1/r)$, $\beta = \frac{2r}{r+2}$,
\[ \left\| \sum_{n=0}^ \infty{\nu_n|e^{-itH} f_n|^2} \right\|_{L^{q/2} _t L^{r/2} _x} \lesssim \| \nu_n\|_{\ell^\beta}\]
holds. 
\item[(2)] For $d \ge 3$, $r$ satisfying $\frac{2(d+1)}{d-1} \le r \le \frac{2d}{d-2}$, $\{f_n\}$: all orthonormal systems in $\dot{H}^s$, $s= d(1-\sigma)(1/2 - 1/r)$, the same estimate as (1) holds if $1 \le \beta <q/2$.
\end{enumerate}
\end{corollary}
For the proof of Corollary \ref{1272125}, we use the following Strichartz estimates.
\begin{lemma} [\cite{MY1}] \label{1272211}
Let $H$ be as in Theorem \ref{1251821}. Then, for all $d/2$ admissible pair $(p, q)$, we have
\begin{align}
\||D|^{2(\sigma -1)/p} e^{-itH} u\|_{L^p _t L^{q, 2} _x} \lesssim \|u\|_2. \label{1272158}
\end{align}
Furthermore, if $(\tilde{p}, \tilde{q})$ is a $d/2\sigma$ admissible pair, we have
\begin{align}
\|e^{-itH} u\|_{L^{\tilde{p}} _t L^{\tilde{q}, 2} _x} \lesssim \|u\|_2. \label{128932}
\end{align}
\end{lemma}
\begin{remark}
In the statement of Theorem 1.8 in \cite{MY1}, their results are described as $\||D|^{2(\sigma -1)/p} e^{-itH} u\|_{L^p _t L^{q} _x} \lesssim \|u\|_2$. However, in their proof, the stronger estimates (\ref{1272158}) are proved. The same comment holds for (\ref{128932}).
\end{remark}
\begin{proof}[Proof of Corollary \ref{1272125} ]
We decompose $V$ in the similar way as in the proof of Corollary \ref{1271749}: $V=v^* _1 v_2 = |x|^{-\sigma+d(1-\sigma)(1/2 -1/r)} \cdot (|x|^{\sigma+d(\sigma-1)(1/2 -1/r)} V)$. Set $-\theta = \frac{d}{2\sigma} (\sigma-1)(1/2-1/r) \in (0, 1/2]$.

\noindent (step1) In this step we prove $\|v_1 H^{\theta} _0 e^{-isH_0}\|_{L^2 _x \rightarrow L^2 _t L^2 _x} \lesssim 1$. As before we assume $r=2$ or $r= \frac{2d}{d-2}$. If $r=2$, then $\theta =0$ and
\begin{align*}
\| |x|^{-\sigma} e^{-itH_0} u\|_{L^2 _t L^2 _x} \lesssim \|e^{-itH_0} u\|_{L^2 _t L^{\frac{2d}{d-2\sigma}, 2}} \lesssim \|u\|_2
\end{align*}
holds. Here, in the first inequality, we have used H\"older's inequality for the Lorentz space and in the second inequality, we have used the endpoint Strichartz estimates (\ref{128932}). If $r=\frac{2d}{d-2}$, then $-\theta = \frac{\sigma-1}{2\sigma}$ and
\begin{align*}
\| |x|^{-2\sigma+1} |D|^{-(\sigma-1)} e^{-itH_0} u\|_{L^2 _t L^2 _x} &\lesssim \||D|^{-(\sigma-1)} e^{-itH_0} u\|_{L^2 _t L^{\frac{2d}{d-4\sigma+2}, 2} _x} \\
&\lesssim \|e^{-itH_0} u\|_{L^2 _t L^{\frac{2d}{d-2\sigma}, 2}} \lesssim \|u\|_2
\end{align*}
holds. Here, in the first line, we have used the H\"older inequality, in the second inequality, we have used the Sobolev inequality: $|D|^{-(\sigma-1)}: L^{\frac{2d}{d-2\sigma}, 2} \rightarrow L^{\frac{2d}{d-4\sigma+2}, 2}$ and in the last inequality, we have used (\ref{128932}). Hence $\|v_1 H^{\theta} _0 e^{-isH_0}\|_{L^2 _x \rightarrow L^2 _t L^2 _x} \lesssim 1$ holds.

\noindent (step2) In this step, we prove $\|v_2 e^{-itH} H^{-\theta} _0\|_{L^2 _x \rightarrow L^2 _t L^2 _x} \lesssim 1$. As before we may assume $r=2$ or $r= \frac{2d}{d-2}$. If $r=2$, then $\theta =0$ and
\begin{align*}
\|v_2 e^{-itH} u\|_{L^2 _t L^2 _x} &\lesssim \||x|^{\sigma} V e^{-itH} u\|_{L^2 _t L^2 _x} \lesssim \||x|^{-\sigma} e^{-itH} u\|_{L^2 _t L^2 _x} \\
& \lesssim \|e^{-itH} u\|_{L^2 _t L^{\frac{2d}{d-2\sigma}, 2}} \lesssim \|u\|_2
\end{align*}
holds. Here we have used $|x|^{2\sigma} V \in L^{\infty}$ and the endpoint Strichartz estimates (\ref{128932}). If $r=\frac{2d}{d-2}$, then $-\theta = \frac{\sigma-1}{2\sigma}$ and
\begin{align*}
\||x|^{2\sigma-1} V e^{-itH} H^{-\theta} _0 u\|_{L^2 _t L^2 _x} &\lesssim \||x|^{-1} e^{-itH} H^{-\theta} _0 u\|_{L^2 _t L^2 _x} \lesssim \|e^{-itH} H^{-\theta} _0 u\|_{L^2 _t L^{\frac{2d}{d-2}, 2} _x} \\
& = \|H^{-\theta} H^{\theta} _0  H^{-\theta} _0 e^{-itH} H^{\theta} H^{-\theta} _0 u\|_{L^2 _t L^{\frac{2d}{d-2}, 2} _x} \\
& \lesssim \|H^{-\theta} _0 e^{-itH} H^{\theta} H^{-\theta} _0 u\|_{L^2 _t L^{\frac{2d}{d-2}, 2} _x} \\
& \lesssim \|H^{\theta} H^{-\theta} _0 u\|_2 \lesssim \|u\|_2
\end{align*}
holds. Here in the third line, we have used $\|H^{-\theta} H^{\theta} _0\|_{\mathcal{B} (L^{\frac{2d}{d-2}, 2} )} \lesssim 1$, which follows from $\|H^{1/2} u\|_{\frac{2d}{d-2}, 2} \lesssim \|H^{1/2} _0 u\|_{\frac{2d}{d-2}, 2}$ and the complex interpolation and in the last line, we have used $\|H^{\theta} H^{-\theta} _0\|_{\mathcal{B} (L^2)} \lesssim 1$, which follows from the assumption $\langle Hu, u \rangle \gtrsim \langle (-\Delta)^{\sigma} u, u \rangle$ and the complex interpolation. Hence we obtain $\|v_2 e^{-itH} H^{-\theta} _0\|_{L^2 _x \rightarrow L^2 _t L^2 _x} \lesssim 1$.

By step1, step2, Theorem \ref{1271257} and the free orthonormal Strichartz estimates, we have the desired estimates.
\end{proof}
\begin{example} \label{1281026}
Assume $H$ as in Theorem \ref{1251821} and $1< \sigma < \frac{d-2}{4}$. Then all the assumptions in Corollary \ref{1272125} holds.
\end{example}
\begin{proof}
It suffices to show $\|H u\|_{\frac{2d}{d-2}, 2} \lesssim \|H _0 u\|_{\frac{2d}{d-2}, 2}$. This follows from the H\"older inequality and the Sobolev embedding as follows:
\begin{align*}
\|Vu\|_{\frac{2d}{d-2}, 2} \lesssim \||x|^{-2\sigma} u\|_{\frac{2d}{d-2}, 2} \lesssim \|u\|_{\frac{2d}{d-4\sigma-2}, 2} \lesssim  \||D|^{2\sigma} u\|_{\frac{2d}{d-2}, 2} =  \|H _0 u\|_{\frac{2d}{d-2}, 2}.
\end{align*}
\end{proof}
Finally we consider the critical Hardy potential case.
\begin{thm}[Critical Hardy potential]
Let $H_{crit, \sigma} = (-\Delta)^{\sigma} -C_{\sigma, d} |x|^{-2\sigma}$ and $1< \sigma < \frac{d-2}{4}$. Then, for $(q, r)$: $d/2$ admissible pair satisfying $d/r + 2\sigma/q <d$, we have
\begin{enumerate}
\item[(1)] For $r$ satisfying $2 \le r < \frac{2(d+1)}{d-1}$, $\{f_n\}$: all orthonormal systems in $\dot{H}^s$, $s= d(1-\sigma)(1/2 - 1/r)$, $\beta = \frac{2r}{r+2}$,
\[ \left\| \sum_{n=0}^ \infty{\nu_n|e^{-itH_{crit, \sigma}} P^{\perp} _{\rad}f_n|^2} \right\|_{L^{q/2} _t L^{r/2} _x} \lesssim \| \nu_n\|_{\ell^\beta}\]
holds. 
\item[(2)] For $d \ge 3$, $r$ satisfying $\frac{2(d+1)}{d-1} \le r \le \frac{2d}{d-2}$, $\{f_n\}$: all orthonormal systems in $\dot{H}^s$, $s= d(1-\sigma)(1/2 - 1/r)$, the same estimate as (1) holds if $1 \le \beta <q/2$.
\end{enumerate}
\end{thm}
\begin{proof}
First, note that $P_{\rad} (P^{\perp} _{\rad})$ commutes with $|D|^{\alpha}$ and $H_{crit, \sigma}$. Hence they commute with $f(|D|^{\alpha})$ and $f(H_{crit, \sigma})$ for $f \in L^2 _{loc}$. By repeating the proof of Theorem \ref{1271257}, it suffices to show that $v_1 H^{\theta} _0$ is $H_0$-smooth and 
\begin{align}
\|v_2 e^{-itH_{crit, \sigma}} H^{-\theta} _0 P^{\perp} _{\rad} u\|_{L^2 _t L^2 _x} \lesssim \|u\|_2 \label{1281227}
\end{align}
holds for $\theta =-\frac{d}{2\sigma} (\sigma-1)(1/2-1/r)$, $v_1 = |x|^{-\sigma+d(1-\sigma)(1/2 -1/r)}$, $v_2 = -C_{\sigma, d} |x|^{-\sigma- d(1-\sigma)(1/2 -1/r)}$. However, $H_0$-smoothness of $v_1 H^{\theta} _0$ follows from the proof of Corollary \ref{1272125}. Then we consider (\ref{1281227}). As before we assume $r=2$ or $r= \frac{2d}{d-2}$. If $r=2$, 
\begin{align*}
\|v_2 e^{-itH_{crit, \sigma}} H^{-\theta} _0 P^{\perp} _{\rad} u\|_{L^2 _t L^2 _x} = \||x|^{-\sigma} e^{-itH_{crit, \sigma}} P^{\perp} _{\rad} u\|_{L^2 _t L^2 _x} \lesssim \|u\|_2
\end{align*}
holds. In the last inequality, we have used the Kato-Yajima estimates: 
\[ \sup_{z \in \C \setminus \R} \||x|^{-\sigma} P^{\perp} _{\rad} (H_{crit, \sigma} -z)^{-1} P^{\perp} _{\rad} |x|^{-\sigma}\|_{L^2 \rightarrow L^2} < \infty ,\]
which was proved in \cite{MY1} Theorem 6.1. If $r= \frac{2d}{d-2}$, then
\begin{align*}
\|v_2 e^{-itH_{crit, \sigma}} H^{-\theta} _0 P^{\perp} _{\rad} u\|_{L^2 _t L^2 _x} &\lesssim \|e^{-itH_{crit, \sigma}} H^{-\theta} _0 P^{\perp} _{\rad} u\|_{L^2 _t L^{\frac{2d}{d-2}, 2} _x} \\
& = \|H^{-\theta} H^{\theta} _0 H^{-\theta} _0 P^{\perp} _{\rad} e^{-itH_{crit, \sigma}}  P^{\perp} _{\rad} H^{\theta} H^{-\theta} _0u\|_{L^2 _t L^{\frac{2d}{d-2}, 2} _x} \\
& \lesssim  \| H^{-\theta} _0 P^{\perp} _{\rad} e^{-itH_{crit, \sigma}}  P^{\perp} _{\rad} H^{\theta} H^{-\theta} _0u\|_{L^2 _t L^{\frac{2d}{d-2}, 2} _x} 
\end{align*}
holds. Here we have used $\|H^{-\theta} _{crit, \sigma} u\|_{\frac{2d}{d-2}, 2} \lesssim \|H^{-\theta} _0 u\|_{\frac{2d}{d-2}, 2}$, which can be deduced by the same way as in Example \ref{1281026}. By using the usual Strichartz estimates (\cite{MY1} Theorem 6.1):
\[\||D|^{2(\sigma -1)/p} P^{\perp} _{\rad} e^{-itH_{crit, \sigma}} u\|_{L^p _t L^{q, 2} _x} \lesssim \|u\|_2\]
we obtain
\begin{align*}
\|v_2 e^{-itH_{crit, \sigma}} H^{-\theta} _0 P^{\perp} _{\rad} u\|_{L^2 _t L^2 _x} &\lesssim \|P^{\perp} _{\rad} H^{\theta} _{crit, \sigma} H^{-\theta} _0 u\|_2.
\end{align*}
Now it suffices to show 
\begin{align}
\| H^{-\theta} _0 P^{\perp} _{\rad} u\|_2 \lesssim \| H^{-\theta} _{crit, \sigma} P^{\perp} _{\rad} u\|_2 . \label{1281317}
\end{align}
Since $-\theta \in (0, 1/2]$, by the complex interpolation, we may assume $-\theta =1/2$. Then (\ref{1281317}) is equivalent to
\[\langle H_0 P^{\perp} _{\rad} u, P^{\perp} _{\rad} u \rangle \lesssim \langle H_{crit, \sigma} P^{\perp} _{\rad} u, P^{\perp} _{\rad} u \rangle.\]
This estimate follows from the following sharp Hardy inequality:
\[\tilde{C}_{\sigma, d} \int |x|^{-2\sigma} |u(x)|^2 dx \le \||D|^{\sigma} u\|_2\]
for all $u \in P^{\perp} _{\rad} (C^{\infty} _0)$, which was also proved in \cite{MY1}. Here the most important point is that $\tilde{C} _{\sigma, d} := \left\{\frac{2^{\sigma} \Gamma (\frac{d+2\sigma+2}{4})}{\Gamma (\frac{d-2\sigma-2}{4})}\right\}^2$ is strictly larger than $C_{\sigma, d}$, which is the best constant of the usual Hardy inequality.
\end{proof}

\subsection{\textbf{Applications to infinitely many particle system}}
In this subsection, we consider one application of the orthonormal Strichartz estimates to the nonlinear Schr\"odinger equation (H). Here, we consider the higher order or fractional type Hamiltonian: $H= (-\Delta)^{\sigma} +V$. Concerning previous researches, if $\sigma =1$ and $V=0$, \cite{LS} and \cite{CHP1} proved the local existence and global existence results which have physical phenomena in the background. In \cite{FS}, they proved the global existence of the solution by using the improved set of the exponent  in the orthonormal Strichartz estimates if $\sigma =1$ and $V=0$. In \cite{Ho}, the author proved the global existence of a solution for the magnetic Schr\"odinger operator: $H=(D+A)^2 +V$. If $\sigma=1$ and $V=0$, there are some results on the scattering or the stability of stationary solutions. For example, see \cite{LS2}, \cite{CHP2} and \cite{Ha}. Concerning the case $\sigma >1$, \cite{BLN} proved the local existence of the solution if $V=0$. As far as the author knows there seems to be no result on the global existence of the solution even if we assume $V=0$. In this subsection, we prove the global existence of a solution for (\ref{2311152204}) with potentials.
For the proof of Theorem \ref{1281958}, we need one lemma, which gives the existence of a solution to the higher order or fractional linear Schr\"odinger equations with time-dependent potentials.
\begin{lemma} \label{1281959}
Assume $H$ is as in Theorem \ref{1251816} or Theorem \ref{1251821} and $\sigma (H) = \sigma _{\ac} (H)$.  Let $V(t, x) = V_1 (t, x) +V_2 (t, x)$ satisfy $V_1 \in L^{\alpha}_T L^p _x$ and $V_2 \in L^{\beta} _T L^{\infty} _x$ for some $p \ge 1, \alpha \ge 1, \beta >1, 0 \le 1/\alpha < 1- \frac{d}{2mp}$. Here $L^q_T L^r _x$ denotes $L^q ([-T, T]; L^r _x)$ for $T \in (0, \infty]$. Then the integral equation:
\[ u(t) = e^{-i(t-s)H} u_0 -i \int_{s}^{t} e^{-i(t-s')H} V(s')u(s') ds' , \quad  t \in [-T, T]\]
has a unique solution $u \in C_T L^2 _x \cap L^{\theta} _{loc, T} L^q _x$ for $q= \frac{2p}{p-1}, \theta = \frac{4mp}{d}$. Furthermore the mass conservation law; $\|u(t)\|_2 = \|u_0\|_2$ holds. As before, if $H$ is as in Theorem \ref{1251821}, we assume the above condition substituting $\sigma$ for $m$.
\end{lemma}
Lemma \ref{1281959} for $H=-\Delta$ was first proved in \cite{Y}. If $H=(D+A)^2 +V$, it was proved in \cite{Ho} and was applied to the global existence of a unique solution to the infinitely many particle systems.
\begin{corollary} \label{1282056}
Under assumptions in Lemma \ref{1281959}, There exists a family of unitary operators $\{U(t, s)\}$ such that
\begin{itemize}

\item $U(t, s)$ is strongly continuous with respect to $(t, s)$.

\item $U(t, s)U(s, r) = U(t, r)$ holds.

\item $U(t, s)u_0$ is a unique solution to the integral equation in Lemma \ref{1281959}.

\end{itemize}
\end{corollary}
\begin{proof}[Proof of Lemma \ref{1281959} and Corollary \ref{1282056}]
In the proof, we follow the same line as in \cite{Y} and \cite{Ho}. Let $q, \theta$ be as in the statement. Set
\begin{align*}
&\mathcal{X}(a) = \mathcal{X}(a, q) = C(I; L^2 _x) \cap L^{\theta} (I; L^q _x) \\
&\mathcal{X}^* (a) = \mathcal{X}^* (a, q) = L^1 (I; L^2 _x) + L^{\theta'} (I; L^{q'} _x) \\
&\mathcal{M} (a) = L^{\alpha} (I; L^p _x) + L^{\beta} (I; L^{\infty} _x)
\end{align*}
for $0< a < 1/2$ and $I = [-a, a]$. First, we prove
\[\|Vu\|_{\mathcal{X}^*(a)} \lesssim (2a)^{\gamma} \|V\|_{\mathcal{M}(a)} \|u\|_{\mathcal{X}(a)}\]
for  $\gamma = min (1-\frac{1}{\beta}, 1- \frac{d}{2mp} -\frac{1}{\alpha}) >0$. By the assumptions, we have $1/\alpha < 1-\frac{d}{2mp} = 1/{\theta'} -1/\theta$. Therefore
\begin{align*}
&\|V_1 u\|_{L^{\theta'} (I; L^{q'} _x)} \lesssim \|V_1\|_{L^{\frac{2mp}{2mp-d}} (I; L^{p} _x)} \|u\|_{L^{\theta} (I; L^q _x)} \lesssim |I|^{1-\frac{d}{2mp} -1/\alpha} \|V_1\|_{L^{\alpha} (I; L^p _x)} \|u\|_{\mathcal{X}(a)} \\
& \|V_2 u\|_{L^1 (I; L^2 _x)} \lesssim \|V_2\|_{L^{1} (I; L^{\infty} _x)} \|u\|_{C(I; L^2 _x)} \lesssim |I|^{1-1/\beta} \|V_2\|_{L^{\beta} (I; L^{\infty} _x)} \|u\|_{\mathcal{X}(a)}
\end{align*}
holds. Then we obtain
\begin{align*}
\|Vu\|_{\mathcal{X}^*(a)} &\le \inf_{
V=V_1 +V_2, 
V_1 \in L^{\alpha} (I; L^p _x), V_2 \in L^{\beta} (I; L^{\infty} _x)} (\|V_1 u\|_{L^{\theta'} (I; L^{q'} _x)} + \|V_2 u\|_{L^1 (I; L^2 _x)}) \\
& \lesssim (2a)^{\gamma} \inf_{V=V_1 +V_2, 
V_1 \in L^{\alpha} (I; L^p _x), V_2 \in L^{\beta} (I; L^{\infty} _x)} (\|V_1\|_{L^{\alpha} (I; L^p _x)} + \|V_2\|_{L^{\beta} (I; L^{\infty} _x)}) \|u\|_{\mathcal{X}(a)} \\
& = (2a)^{\gamma} \|V\|_{\mathcal{M}(a)} \|u\|_{\mathcal{X}(a)}
\end{align*}
By this estimate and the Strichartz estimates for $H$, we have
\begin{align*}
&\|Qu\|_{\mathcal{X}(a)} \lesssim \|Vu\|_{\mathcal{X}^* (a)} \lesssim a^{\gamma} \|V\|_{\mathcal{M}(a)} \|u\|_{\mathcal{X}(a)} \\
& Qu(t) = \int_{0}^{t} e^{-i(t-s)H} V(s)u(s) ds . 
\end{align*}
Hence if we take $a$ sufficiently small, we have
\begin{align*}
&\Phi : \mathcal{X}(a) \rightarrow \mathcal{X}(a), u \mapsto e^{-itH} u_0 -iQu \\
& \|\Phi (u-v)\|_{\mathcal{X}(a)} \le \frac{1}{2} \|u-v\|_{\mathcal{X}(a)} .
\end{align*}
Then $\Phi$ is a contraction mapping on $\mathcal{X}(a)$ and our integral equation for $s=0$ has a unique solution $u(t) = (1+ iQ)^{-1} u_0 \in \mathcal{X}(a)$. By taking $u_0 = e^{isH} u_0$ and $V = V(t+s)$, we have a unique solution $u(t)$ to the integral equation with initial data $u(s) = u_0$ on $[s-a, s+a]$. Then $U(t, s)u_0 := u(t)$ is a linear operator on $L^2 _x$ by construction. By glueing these $U(t, s)$, we have a unique solution on $[-T, T]$ and a family of linear operators $\{U(t, s)\}$ for $t, s \in [-T, T]$. Hence it suffices to show that $\|U(t, s)u_0\|_2 = \|u_0\|_2$ holds. Instead of the regularizing technique in \cite{Y}, we use the method in \cite{O}. We assume $s=0$ for the sake of simplicity. Then
\begin{align*}
\|U(t, 0)u_0\|^2 _2 &= \left\|e^{-itH} u_0 -i \int_{0}^{t} e^{-i(t-s)H} V(s)u(s) ds \right\|^2_2 \\
& = \|u_0\|^2 _2 + \left\| \int_{0}^{t} e^{-i(t-s)H} V(s)u(s) ds \right\|^2 _2 -2 \im \int_{0}^{t} \langle e^{-isH} u_0 , V(s)u(s) \rangle ds .
\end{align*}
Note that the third term is well-defined since $e^{-isH} u_0, u(s) \in C_{T} L^2 \cap L^{\theta} _{loc, T} L^{q} _x$ ensures that 
\begin{align*}
 \int_{0}^{t} |\langle e^{-isH} u_0 , V(s)u(s) \rangle| ds & \lesssim _{t} \|V_1\|_{L^{\alpha} _{[0, t]} L^p _x} \|e^{-itH} u_0\|_{L^{\theta} _T L^q _x} \|u\|_{L^{\theta} _{[0, t]} L^q _x} + \int_{0}^{t} \|V_2 (s)\|_{\infty} \|u_0\|_2 \|u\|_{L^{\infty} _{T} L^2 _x} ds \\
& < {\infty}
\end{align*}
by the relation: $1/p + 1/q + 1/q =1, 1/\alpha + 1/\theta + 1/\theta <1$. We transform the second term as
\begin{align*}
\left\| \int_{0}^{t} e^{-i(t-s)H} V(s)u(s) ds \right\|^2 _2 &= \re \int_{0}^{t} \int_{0}^{t} \langle V(s)u(s) , e^{-i(s-s')H} V(s')u(s') \rangle ds' ds \\
& = 2 \re \int_{0}^{t} \int_{0}^{s} \langle V(s)u(s) , e^{-i(s-s')H} V(s')u(s') \rangle ds' ds \\
& = -2 \im \int_{0}^{t} \langle V(s)u(s) , u(s) + i\int_{0}^{s} e^{-i(s-s')H} V(s')u(s') ds \rangle ds \\
& = 2 \im \int_{0}^{t} \langle e^{-isH} u_0 , V(s)u(s) \rangle ds .
\end{align*}
Then we have $\|U(t, 0)u_0\|_2 = \|u_0\|_2$.
\end{proof}
\begin{proof}[Outline of the proof of Theorem \ref{1281958}]
The proof is essentially the same as in \cite{Ho} Theorem 2.18 (local existence part is just like \cite{FS}). So we only give the outline of the proof. By repeating the proof of Lemma 4.14 in \cite{Ho}, we obtain
\[\| \rho_{\gamma(t)}\|_{L^{q/2}_t L^{r/2}_x} \lesssim \|\gamma_0\|_{\mathfrak{S}^{\beta}} + \left\|\int_{\R} e^{isH} |R(s)| e^{-isH} ds \right\|_{\mathfrak{S}^{\beta}}\]
for
\begin{align*}
\gamma(t) = e^{-itH} \gamma_{0} e^{itH} -i \int_{0}^{t} e^{-i(t-s)H} R(s)e^{i(t-s)H} ds
\end{align*}
Here $\gamma_0 \in \mathfrak{S}^{\beta}$ and $R(t)$ is an operator-valued function.
Let $R >0$ be such that $\|\gamma_0\|_{\mathfrak{S}^{\beta}} < R$, $T =T(R)$ be chosen later. Set
\[ X := \left\{ (\gamma , \rho) \in C([0, T]; \mathfrak{S}^{\beta}) \times L^{q/2} ([0, T]; L^{r/2} _x) \mid \|\gamma\|_{C([0, T]; \mathfrak{S}^{\beta})} + \|\rho\|_{L^{q/2} ([0, T]; L^{r/2} _x)} \le CR \right\}.\]
Here $C>0$ is later chosen independent of $R$.
We define 
\begin{align*}
&\Phi (\gamma , \rho) =(\Phi_1 (\gamma , \rho), \rho [\Phi_1 (\gamma, \rho)]), \\
& \Phi_1 (\gamma , \rho)(t) = e^{-itH} \gamma_0 e^{itH} -i \int_{0}^{t} e^{-i(t-s)H} [w*\rho , \gamma] e^{i(t-s)H} ds .
\end{align*}
For $(\gamma , \rho) \in X$, by the orthonormal Strichartz estimates and the unitary invariance of the Schatten norms, we have
\begin{align*}
\| \Phi_1 (\gamma , \rho)\|_{C([0, T]; \mathfrak{S}^{\beta})} &\lesssim \|\gamma_0\|_{\mathfrak{S}^{\beta}} + 2\int_{0}^{t} \|w* \rho\|_{\infty} \|\gamma\|_{\mathfrak{S}^{\beta}} ds \\
& \lesssim R + 2T^{1/(q/2)'} \|w\|_{(r/2)'} \|\rho\|_{L^{q/2} ([0, T]; L^{r/2} _x)} \|\gamma\|_{C([0, T]; \mathfrak{S}^{\beta}) } \\
& \lesssim R + 2\|w\|_{(r/2)'} C^2 T^{1/(q/2)'} R^2 .
\end{align*}
By the above estimates, 
\begin{align*}
\|\rho [\Phi_1 (\gamma , \rho)]\|_{L^{q/2} ([0, T]; L^{r/2} _x)} &\lesssim R + 2 \int_{0}^{T} \|w* \rho\|_{\infty} \|\gamma\|_{\mathfrak{S}^{\beta}} ds \\
& \le R + 2\|w\|_{(r/2)'} C^2 T^{1/(q/2)'} R^2 .
\end{align*}
If we take $C$ large enough and $T$ small enough, we have $\Phi : X \rightarrow X$. By the similar computation, $\Phi$ is a contraction mapping on $X$. Hence there exists a unique solution to (H) on $[0, T_{max} )$. To prove the global existence, it suffices to show that $\gamma (t) = U(t) \gamma_0 U(t)^* (=: \eta (t))$ holds on $t \in [0, T_{max} - \epsilon ]$ since this implies $\|\gamma (t)\|_{\mathfrak{S}^{\beta}} = \|\gamma_0\|_{\mathfrak{S}^{\beta}}$ . Here the unitary operator $U(t)$ is defined by
\begin{itemize}

\item $U(t)$ is strongly continuous

\item $U(0) = I$

\item $U(t)u_0 = e^{-itH} u_0 -i \int_{0}^{t} e^{-i(t-s)H} w* \rho_{\gamma}(s)U(s)u_0 ds$ on $[0, T_{max} - \epsilon ]$. 
\end{itemize}
In this step we have used Corollary \ref{1282056}. The difficult point in the proof of $\gamma (t) = U(t) \gamma_0 U(t)^* $ is that $U(t)$ is not differentiable and hence we need a little technical computation. However they are exactly the same as in the proof of Theorem 2.18 in \cite{Ho} (just substituting $\mathfrak{S}^{\beta}$ for $\mathfrak{S}^{\frac{2q}{q+1}}$). So we omit the details.
\end{proof}

\subsection{\textbf{Refined Strichartz estimate}}
In this subsection we prove the refined Strichartz estimates. They were proved in \cite{FS} for $H=-\Delta$ and in \cite{BLN} for $H=(-\Delta)^{\sigma}$ and $H=(1-\Delta)^{1/2}$. In \cite{Ho}, the refined Strichartz estimates for $H=(D+A)^2 +V$ was proved by using the $L^p$-boundedness of spectral multipliers and the Littlewood-Paley theory. See \cite{BLN} for more details about this topic including their applications. Now we consider the refined Strichartz estimates for $H=(-\Delta)^m +V$. To deduce the $L^p$-boundedness of spectral multipliers we assume Assumption \ref{129026}.
Under Assumption \ref{129026}, \cite{EG} proved that wave operators: $W_{\pm} := \slim_{t \rightarrow \pm \infty} e^{itH} e^{-itH_0}$ are bounded operators on $L^p (\R^d)$ if $1<p<\infty$. Concerning the $L^p$-boundedness of wave operators for higher order operators, see \cite{EG}, \cite{GY}, \cite{MWY} and references therein. As a corollary we obtain the following result.
\begin{corollary} \label{1291118}
Let $f \in C^{\infty} ((0, \infty))$ satisfy $| \partial^{k} _{s} f(s)| \lesssim s^{-k}$ and $H$ be as in Assumption \ref{129026}. Then for any $p \in (1, \infty)$,
\[\|f(\lambda H)\|_{\mathcal{B} (L^p)} \lesssim 1\]
holds uniformly in $\lambda >0$. Furthermore, if $H \ge 0$, we obtain
\begin{align}
\|f(\lambda H^{\frac{1}{2m}})\|_{\mathcal{B} (L^p)} \lesssim 1 \label{1291126}
\end{align}
uniformly in $\lambda >0$.
\end{corollary}
\begin{proof}
Since the latter proof is similar, so we only consider the former estimate. By the intertwining property, we obtain
\[f(\lambda H) P_{\ac} (H) = W_{\pm} f(\lambda (-\Delta)^m)W^* _{\pm} .\] 
Since $\supp f \subset [0, \infty)$ and $\sigma (H) \cap [0, \infty ) = \sigma_{\ac} (H)$, $f(\lambda H)P_{\ac} (H) =f(\lambda H)$ holds. The symbol of $f(\lambda (-\Delta)^m)$ is $f(\lambda |\xi|^{2m})$ and it satisfies $|\partial^{\alpha} _{\xi} f(\lambda |\xi|^{2m})| \lesssim |\xi|^{-\alpha}$ by our assumption. Then by the H\"ormander-Mikhlin theorem, $\|f(\lambda (-\Delta)^m)\|_{\mathcal{B} (L^p)} \lesssim 1$ holds uniformly in $\lambda>0$. Hence by the $L^p$-boundedness of $W_{\pm}$ and $W^* _{\pm}$, we obtain
\begin{align*}
\|f(\lambda H)\|_{\mathcal{B} (L^p)} = \|W_{\pm} f(\lambda (-\Delta)^m)W^* _{\pm} \|_{\mathcal{B} (L^p)} \lesssim 1.
\end{align*}
\end{proof}
Next we prove the Littlewood-Paley theorem for $H$. If $H=-\Delta$, this result is well-known. In \cite{M1}, $H=-\Delta +V$ was considered for large class of $V \ge0$. The author considered the magnetic Schr\"odinger operator in \cite{Ho}.
\begin{proposition} \label{1291227}
Assume $H$ be as in Assumption \ref{129026} and $H\ge0$. Let $1<p<\infty$ and $\phi_{j}$ be the homogeneous Littlewood-Paley decomposition. Then
\begin{align*}
& \|u\|_{p} \approx \|\|\{\phi_j (H) u\}\|_{\ell^2}\|_{p} \\
& \|u\|_{p} \approx \|\|\{\phi_j ({H}^{\frac{1}{2m}}) u\}\|_{\ell^2}\|_{p} 
\end{align*}
holds.
\end{proposition}
\begin{proof}
We only consider the latter estimate since the former is similarly proved. We follow the argument in \cite{M1}. Set
\[Su(x) = \|\{\phi_j ({H}^{\frac{1}{2m}}) u\}\|_{\ell^2}, \quad S_{\pm} u=\left(\sum_{\pm j \ge0} |\phi_j ({H}^{\frac{1}{2m}}) u|^2\right)^{1/2}.\]
Since zero is not an eigenvalue,
\[\langle u, v \rangle = \sum_{j, k \in \Z} \langle \phi_j ({H}^{\frac{1}{2m}}) u, \phi_k ({H}^{\frac{1}{2m}}) v \rangle\]
holds. By the almost orthogonality of $\phi_j ({H}^{\frac{1}{2m}})$ and H\"older's inequality, we have
\begin{align*}
|\langle u, v \rangle| \le \sum_{j, k \in \Z} \int |\phi_j ({H}^{\frac{1}{2m}}) u \phi_k ({H}^{\frac{1}{2m}}) v| dx &= \int \sum_{j, k \in \Z, |j-k| \le 2} |\phi_j ({H}^{\frac{1}{2m}}) u \phi_k ({H}^{\frac{1}{2m}}) v| dx \\
& \le 5 \int Su(x) Sv(x) \le 5\|Su\|_p \|Sv\|_{p'}.
\end{align*}
If we prove $\|Su\|_p \lesssim \|u\|_p$ for all $p \in (1, \infty)$, we obtain $\|u\|_p \lesssim \|Su\|_p$ by the above inequality and the duality argument. First we consider $S_+$. Let $r_j (t)$ be the Rademacher functions on $[0, 1]$, i.e. $r_0 (t) =1$ for $t \in [0, 1/2]$, $r_0 (t)=-1$ for $t \in (1/2, 1)$, $r_j (t) =r_0 (2^j t)$ for $j \ge 1$ and we extend them periodically. Set
\[ m_t (s) = \sum_{j \ge 0} r_j (t) \phi (\frac{s}{2^j}).\]
By the support property of $\phi_j$, we have $|\partial^k _s m_t (s)| \lesssim s^{-k}$. Therefore $m_t$ satisfies the condition on $f$ in Corollary \ref{1291118} since $\supp \phi$ is away from zero. Hence
\[\|m_t ({H}^{\frac{1}{2m}} )u\|_p \lesssim \|u\|_p\]
holds for $p \in (1, \infty)$. Now we use Khintchin's inequality:
\[ \{a_m\} \in \ell^2, F(t):= \sum_{m \in \Z_{\ge0}} a_m r_m (t) \Rightarrow \|F\|_{L^p _t} \approx \|\{a_m\}\|_{\ell^2}\]
for all $p \in (1, \infty)$. Then we have
\begin{align*}
\|S_{+} u\|_p \approx \|\|m_t ({H}^{\frac{1}{2m}})u\|_{L^p _t}\|_p = \|m_t ({H}^{\frac{1}{2m}})u\|_{L^p ([0, 1]: L^p _x)} \lesssim \|u\|_p
\end{align*}
for all $p \in (1, \infty)$. For $S_{-}$, Set $m_t (s) = \sum_{j \ge 0} r_j (t) \phi (2^j s)$. Then $m_t \in C^{\infty} _0$ and the above argument can be also applied to this case. Hence we are done.\end{proof}
\begin{remark} \label{12131522}
In this section, decay assumptions on $V$ are due to the $L^p$-boundedness of the spectral multipliers: $f(H)$ for $f \in C^{\infty} ((0, \infty))$ satisfy $| \partial^{k} _{s} f(s)| \lesssim s^{-k}$. If $f \in C^{\infty} _0$ and $V \in C^{\infty}$, this estimate would be proved by the Helffer-Sj\"ostrand formula: 
\[f(H)=\frac{i}{2\pi} \int_{\C} \frac{\partial \tilde{f}}{\partial \bar{z}} (H-z)^{-1} dz \wedge d\bar{z}\]
and the microlocal parametrix construction of resolvents under much weaker decay assumptions. However this method does not seem to apply to $f \in C^{\infty} ((0, \infty))$ satisfying $| \partial^{k} _{s} f(s)| \lesssim s^{-k}$.
\end{remark}
The next lemma implies a relation between the Besov space associated to $H^{\frac{1}{2m}}$ and the usual one.
\begin{lemma} \label{1291301}
Let $H$ be as in Proposition \ref{1291227}. Assume there exists $s_0$ and $s_1$ such that $s_0 <s<s_1$ and $H^{\frac{s_j}{2m}} H^{-\frac{s_j}{2m}} _0$ $[\resp \quad \langle H \rangle^{\frac{s_j}{2m}} \langle H_0 \rangle^{-\frac{s_j}{2m}}]$ is bounded on $L^p$ for some $1<p<\infty$. Then we have
\[\|u\|_{\dot{B}^s _{p, q} (H)} \lesssim \|u\|_{\dot{B}^s _{p, q}}\]
\[ [\resp \quad \|u\|_{{B}^s _{p, q} (H)} \lesssim \|u\|_{{B}^s _{p, q}}]\]
for all $1\le q\le \infty$. Here 
\[\|u\|_{\dot{B}^s _{p, q} (H)}:= \|\{\|\phi_j ({H}^{\frac{1}{2m}})u\|_p\}\|_{\ell^q}.\]
and $\|\cdot\|_{B^s _{p, q} (H)}$ is defined similarly by using the inhomogeneous Littlewood-Paley decomposition .
\end{lemma}
\begin{proof}
We only consider the former estimate. By the intertwining property and the $L^p$-boundedness of wave operators,  
\begin{align*}
\|\phi_j ({H}^{\frac{1}{2m}})u\|_p = \|W_{\pm} \phi_j ({H_0}^{\frac{1}{2m}}) W^* _{\pm} u\|_p \lesssim  \|\phi_j ({H_0}^{\frac{1}{2m}}) W^* _{\pm} u\|_p 
\end{align*}
holds. By taking weighted $\ell^{q}$ norm, we obtain
\begin{align}
\|u\|_{\dot{B}^s _{p, q} (H)} \lesssim \|W^* _{\pm} u\|_{\dot{B}^s _{p, q}}. \label{1291443}
\end{align}
By the assumption, 
\[\||D|^{s_j}  W^*_{\pm} u\|_p = \|W^* _{\pm} H^{s_j/2m} u\|_p \le \|H^{s_j/2m} u\|_p \lesssim \||D|^{s_j}  u\|_p \]
holds. This implies the $\dot{H}^{s_j, p}$-boundedness of $W^* _{\pm}$. By the real interpolation: 
\[(\dot{H}^{s_0, p}, \dot{H}^{s_1, p})_{\theta, q} =\dot{B}^{s} _{p, q}, \quad s=(1-\theta)s_0 + \theta s_1\]
$W^* _{\pm}$ is bounded on $\dot{B}^{s} _{p, q}$. Combining with (\ref{1291443}), we obtain the desired estimate.
\end{proof}
The first result on the refined Strichartz estimates concerns the $d/2m$ admissible pairs.
\begin{thm} \label{1291900}
Assume $H$ be as in Assumption \ref{129026} and $H\ge0$. Let $r, q, \beta$ be as in Theorem \ref{1251816}. If $H^{s_j} H^{-s_j} _0$ are bounded on $L^2 (\R^d)$ for some $\pm s_j>0$, we have
\[\|e^{-itH} u\|_{L^q _t L^r _x} \lesssim \|u\|_{\dot{B}^0 _{2,2\beta}}.\]
\end{thm}
\begin{proof}
By using Proposition \ref{1291227} and Lemma \ref{1291301}, the proof is an easy modification of Corollary 2.14 in \cite{Ho}. So we omit the details here. We remark that under our assumption, we have $\sigma(H) = \sigma_{\ac} (H)$.
\end{proof}
\begin{example} \label{1291915}
Assume $H$ be as in Assumption \ref{129026} and $\langle Hu, u \rangle \gtrsim \langle H_0 u, u\rangle$. Then assumptions in Theorem \ref{1291900} hold.
\end{example}
\begin{proof}
By the Hardy inequality, we also have
\[\langle Hu, u \rangle \lesssim \||D|^m u\|^2 _2 + \|\langle x \rangle^{-m} u\|^2 _2 \lesssim \||D|^m u\|^2 _2 = \langle H_0 u, u\rangle.\]
Hence $H^{\pm 1/2} H^{\mp 1/2} _0$ are bounded on $L^2 (\R^d)$.
\end{proof}
Finally we prove the result for the $d/2$ admissible pair.
\begin{proof}[Proof of Theorem \ref{1291939}]
By Proposition \ref{1291227} and the orthonormal Strichartz estimates, we have
\begin{align*}
\|e^{-itH} u\|^2 _{L^q _t L^r _x} &\lesssim \left\|\sum_{j \in \Z} |e^{-itH} \phi_{j} (H^{\frac{1}{2m}}) u|^2\right\|_{L^{q/2} _t L^{r/2} _x} \\
& \lesssim \left\|\sum_{j \in 2\Z} |e^{-itH} \phi_{j} (H^{\frac{1}{2m}}) u|^2\right\|_{L^{q/2} _t L^{r/2} _x} + \left\|\sum_{j \in 2\Z +1} |e^{-itH} \phi_{j} (H^{\frac{1}{2m}}) u|^2\right\|_{L^{q/2} _t L^{r/2} _x} \\
& \lesssim \|\{\|\phi_j (H^{\frac{1}{2m}}) u\|_{\dot{H}^s}\}\|^2 _{2\beta}.
\end{align*}
By the intertwining property, 
\begin{align*}
\|\phi_j (H^{\frac{1}{2m}}) u\|_{\dot{H}^s} \lesssim \|H^{s/2m} \phi_j (H^{\frac{1}{2m}}) u\|_{2} &= \|W_{\pm} |D|^s \phi_j (|D|) W^* _{\pm} u\|_2 \\
& \le \| |D|^s \phi_j (|D|) W^* _{\pm} u\|_2
\end{align*}
holds. Then, by the Bernstein inequality, we obtain 
\[\|\phi_j (H^{\frac{1}{2m}}) u\|_{\dot{H}^s} \lesssim 2^{js} \|\phi_j (|D|) W^* _{\pm} u\|_2.\]
Hence
\[\|e^{-itH} u\|^2 _{L^q _t L^r _x} \lesssim \|W^* _{\pm} u\|^2 _{\dot{B}^s _{2, 2\beta}}\]
holds. Since $\frac{s}{2m} =  \frac{1-m}{2m} \cdot d(1/2-1/r) \in (-1/2, 0)$ and $\langle Hu, u \rangle \approx \langle H_0 u, u\rangle$, we have $s_0, s_1 \in \R$ such that $s_0 < s <s_1$ and $H^{\frac{s_j}{2m}} H^{-\frac{s_j}{2m}} _0$ are bounded on $L^2 (\R^d)$. Then, by the proof of Lemma \ref{1291301}, we obtain
\[\|e^{-itH} u\|^2 _{L^q _t L^r _x} \lesssim \|W^* _{\pm} u\|^2 _{\dot{B}^s _{2, 2\beta}} \lesssim \|u\|^2 _{\dot{B}^s _{2, 2\beta}}. \]
\end{proof}
\section{Wave, Klein-Gordon and Dirac equations}
In this section, we consider the orthonormal Strichartz estimates for wave, the Klein-Gordon and the Dirac equations. First, we consider wave and the Klein-Gordon equations. For some restricted pairs $(p, q)$, we can prove the estimate easily  under less assumptions. However, it is difficult to prove the estimates for all the pairs since in some cases we need to handle derivatives. In such cases we use the method of microlocal analysis to deduce the estimates. By using these results, we prove the local existence of a solution to the semi-relativistic Hartree equation for infinitely many particles. Next we consider the Dirac equation. Concerning the orthonormal Strichartz estimates for the Dirac equation, there seems to be no result. So first, we deduce the estimates for the free operator based on that of the Klein-Gordon equation and the vector-valued duality principle. Then we consider the massive Dirac equation with potentials. In this way, we also use the construction of the parametrix for some fractional powers of elliptic operators based on the microlocal analysis. For wave and the Klein-Gordon equations, we also consider the refined Strichartz estimates.    
\subsection{\textbf{Wave and Klein-Gordon equations}}
The orthonormal Strichartz estimates for wave and the Klein-Gordon equations were first considered in \cite{FS} for some restricted pairs and generalized to more pairs by \cite{BLN}. In \cite{BLN}, the optimality of their estimates was also proved. However there seems to be no result for both wave and the Klein-Gordon equations with potentials. As in the case of the Schr\"odinger equation, first, we deduce the abstract perturbation theorem based on the smooth perturbation theory and then we apply it to some cases including magnetic operators. 
\begin{thm} \label{1292322}
Let $\mathcal{H}:=L^2 (X)$ for some $\sigma$-finite measure space $X$. Assume self-adjoint operators $H, H_0$ and densely defined closed operators $v_1, v_2$ on $\mathcal{H}$ satisfy $H=H_0 +V=H_0 + v^*_1 v_2$ in the form sense. Furthermore we assume $v_1$ is $H_0$-smooth, $v_2$ is $H$-smooth, $H, H_0 \ge 0$, $H^{1/4} H^{-1/4} _0$ is bounded on $\mathcal{H}$ and $0 \notin \sigma_{p} (H) \cup \sigma_{p} (H_0)$. If 
\[\left\| \sum_{n=0}^ \infty{\nu_n|e^{-it\sqrt{H_0}} f_n|^2} \right\|_{L^{p}_t L^{q}_x} \lesssim \| \nu_n\|_{\ell^\alpha}\]
holds for all $\{f_n\}$: orthonormal systems in $D(H^{1/4} _0)$ and some $p, q \in [1, \infty]$, $\alpha \in (1, \infty)$, then
\[\left\| \sum_{n=0}^ \infty{\nu_n|e^{-it\sqrt{H}} f_n|^2} \right\|_{L^{p}_t L^{q}_x} \lesssim \| \nu_n\|_{\ell^\alpha}\]
holds for the same $p, q, \alpha$.
\end{thm}
Throughout this subsection, we use the following lemma, which was proved in \cite{D}.
\begin{lemma} \label{1212101}
Let $\nu \in \R$ and $H$ be a self-adjoint operator on any Hilbert space $\mathcal{H}$. We assume $H+\nu \ge0$ and $H+\nu$ is injective. If $A$ is $H$-smooth $[\resp H-supersmooth ]$, then $A(H+\nu)^{-1/4}$ is $\sqrt{H+\nu}$-smooth $[\resp \sqrt{H+\nu}-supersmooth]$.
\end{lemma}
\begin{proof}[Proof of Theorem \ref{1292322}]
By the Duhamel formula, we have
\[e^{-it\sqrt{H}} g = \cos t\sqrt{H_0} g -i \frac{\sin t\sqrt{H_0}}{\sqrt{H_0}} \sqrt{H} g - \int_{0}^{t} \frac{\sin(t-s)\sqrt{H_0}}{\sqrt{H_0}} (Ve^{-is\sqrt{H}} g) ds\]
for all $g \in D(\sqrt{H})$. Then, for all $f: \R \times X \rightarrow \C$, we obtain
\begin{align*}
fe^{-it\sqrt{H}} H^{-1/4} _0 = f \cos t\sqrt{H_0} H^{-1/4} _0  &-i f \sin t\sqrt{H_0} H^{-1/4} _0 (H^{-1/4} _0 H^{1/4})(H^{1/4} H^{-1/4} _0) \\
&- f \int_{0}^{t} \frac{\sin(t-s)\sqrt{H_0}}{\sqrt{H_0}} (Ve^{-is\sqrt{H}} H^{-1/4} _0) ds
\end{align*}
as an operator. For the first and second term, by our assumption and Lemma \ref{1113719}, 
\begin{align*}
& \|f \cos t\sqrt{H_0} H^{-1/4} _0\|_{\mathfrak{S}^{2\alpha'}_{x \rightarrow t, x}} \lesssim \|f\|_{L^{2p'} _t L^{2q'} _x} \\
& \|f \sin t\sqrt{H_0} H^{-1/4} _0 (H^{-1/4} _0 H^{1/4})(H^{1/4} H^{-1/4} _0)\|_{\mathfrak{S}^{2\alpha'}_{x \rightarrow t, x}} \lesssim \|f\|_{L^{2p'} _t L^{2q'} _x} 
\end{align*}
holds. For the last term, it suffices to show 
\[\left\| f \int_{0}^{\infty} \frac{\sin(t-s)\sqrt{H_0}}{\sqrt{H_0}} (Ve^{-is\sqrt{H}} H^{-1/4} _0) ds \right\|_{\mathfrak{S}^{2\alpha'}_{x \rightarrow t, x}} \lesssim \|f\|_{L^{2p'} _t L^{2q'} _x} \]
by Lemma \ref{2311152210}. By our assumption, we obtain
\begin{align*}
& \left\| f \int_{0}^{\infty} \frac{\sin(t-s)\sqrt{H_0}}{\sqrt{H_0}} (Ve^{-is\sqrt{H}} H^{-1/4} _0) ds \right\|_{\mathfrak{S}^{2\alpha'}_{x \rightarrow t, x}} \\
&\lesssim \left\|f(t)\sin t \sqrt{H_0} H^{-1/4} _0 \int_{0}^{\infty} \cos s\sqrt{H_0} H^{-1/4} _0 Ve^{-is\sqrt{H}} H^{-1/4} _0 ds \right\|_{\mathfrak{S}^{2\alpha'}_{x \rightarrow t, x}} \\
& + \left\|f(t)\cos t \sqrt{H_0} H^{-1/4} _0 \int_{0}^{\infty} \sin s\sqrt{H_0} H^{-1/4} _0 Ve^{-is\sqrt{H}} H^{-1/4} _0 ds \right\|_{\mathfrak{S}^{2\alpha'}_{x \rightarrow t, x}} \\
& \lesssim \|f\|_{L^{2p'} _t L^{2q'} _x} \left\|\int_{0}^{\infty} \cos s\sqrt{H_0} H^{-1/4} _0 Ve^{-is\sqrt{H}} H^{-1/4} _0 ds \right\|_{\mathcal{B}(L^2)} \\
& + \|f\|_{L^{2p'} _t L^{2q'} _x} \left\|\int_{0}^{\infty} \sin s\sqrt{H_0} H^{-1/4} _0 Ve^{-is\sqrt{H}} H^{-1/4} _0 ds \right\|_{\mathcal{B}(L^2)} .
\end{align*}
For the first term, by Lemma \ref{1212101}, 
\begin{align*}
&\left| \left\langle \int_{0}^{\infty} \cos s\sqrt{H_0} H^{-1/4} _0 Ve^{-is\sqrt{H}} H^{-1/4} _0 uds, v \right\rangle \right| \\
&= \left| \int_{0}^{\infty} \left\langle v_2 e^{-is\sqrt{H}} H^{-1/4} _0 u, v_1 \cos s\sqrt{H_0} H^{-1/4} _0 v  \right\rangle ds \right| \\
& \lesssim \|v_2 e^{-is\sqrt{H}} H^{-1/4} _0 u\|_{L^2 _t L^2 _x} \|v_1 \cos s\sqrt{H_0} H^{-1/4} _0 v \|_{L^2 _t L^2 _x} \\
& \lesssim \|H^{1/4} H^{-1/4} _0 u\|_2 \|v\|_2 \lesssim \|u\|_2 \|v\|_2 
\end{align*}
holds. By the duality argument we obtain
\[\left\|\int_{0}^{\infty} \cos s\sqrt{H_0} H^{-1/4} _0 Ve^{-is\sqrt{H}} H^{-1/4} _0 ds \right\|_{\mathcal{B}(L^2)} \lesssim 1.\]
By repeating similar calculations, we also obtain
\[\left\|\int_{0}^{\infty} \sin s\sqrt{H_0} H^{-1/4} _0 Ve^{-is\sqrt{H}} H^{-1/4} _0 ds \right\|_{\mathcal{B}(L^2)} \lesssim 1.\]
Hence we have the desired estimates.
\end{proof}
Next we state the result concerning the Klein-Gordon equation.
\begin{thm} \label{1212216}
Let $\mathcal{H}:=L^2 (X)$ for some $\sigma$-finite measure space $X$. Assume self-adjoint operators $H, H_0$ and densely defined closed operators $v_1, v_2$ on $\mathcal{H}$ satisfy $H=H_0 +V=H_0 + v^*_1 v_2$ in the form sense. Furthermore we assume $v_1$ is $H_0$-smooth, $v_2$ is $H$-smooth, $H_0 +1 \ge0$, $H +1 \ge0$, $0 \notin \sigma_{p} (H+1) \cup \sigma_{p} (H_0 +1)$ and  $(H+1)^{1/4} (H _0 +1)^{-1/4}$ is bounded on $\mathcal{H}$. If 
\[\left\| \sum_{n=0}^ \infty{\nu_n|e^{-it\sqrt{H_0 +1}} f_n|^2} \right\|_{L^{p}_t L^{q}_x} \lesssim \| \nu_n\|_{\ell^\alpha}\]
holds for all $\{f_n\}$: orthonormal systems in $D((H_0 +1)^{1/4} )$ and some $p, q \in [1, \infty]$, $\alpha \in (1, \infty)$, then
\[\left\| \sum_{n=0}^ \infty{\nu_n|e^{-it\sqrt{H+1}} f_n|^2} \right\|_{L^{p}_t L^{q}_x} \lesssim \| \nu_n\|_{\ell^\alpha}\]
holds for the same $p, q, \alpha$.
\end{thm}
\begin{proof}
The proof is just substituting $H+1$ for $H$, $H_0 +1$ for $H_0$ in the proof of Theorem \ref{1292322}. So we omit the details.
\end{proof}
Now we apply Theorem \ref{1292322} to the Schr\"odinger operator. The first result (Corollary \ref{12131724}) concerns the inverse square type potentials including $V(x) = -a|x|^{-2}$ for $a < \frac{(d-2)^2}{4}$.
\begin{proof}[Proof of Corollary \ref{12131724}]
Since $\langle Hf,f\rangle  \lesssim \| \nabla f \|^2_2$ holds, $H^{1/4} H^{-1/4} _0$ is bounded on $L^2 (\R^d)$. Furthermore, $\frac{1}{|x|}$ is $H_0$-smooth by Proposition 5.3 in \cite{BM}. $H$-smoothness of $|x|V$ follows from Theorem 2.5 in \cite{BM}. The assumption $0 \notin \sigma_{p} (H) \cup \sigma_{p} (H_0)$ also follows from the above $H_0$-smoothness or $H$-smoothness result. Then, if $q=\frac{d+1}{d-1}$, (\ref{12131753}) follows from the corresponding free estimate (Theorem 2 in \cite{BLN}) and Theorem \ref{1292322}. Since (\ref{12131753}) for $q=1$ is trivial, we obtain the desired estimates by the interpolation.
\end{proof}
Next we consider the magnetic Schr\"odinger operator.
\begin{corollary}\label{12131747}
Assume $d \ge3$, $A:\R^d \rightarrow \R^d, V:\R^d \rightarrow \R$ satisfy $|A(x)| + |\langle x \rangle V(x)| \lesssim \langle x \rangle^{-1-\epsilon}$, $\langle x \rangle^{1+\epsilon'}A(x) \in \dot{W}^{1/2, 2d}$, $A \in C^0(\R^d)$ for some $0<\epsilon'<\epsilon$. Let $H = -\Delta +A(x)\cdot D + D\cdot A(x) + V(x)$ or $H = (D +A(x))^2 +V(x)$. Here $D=-i\nabla$. Furthermore we assume zero is neither an eigenvalue nor a resonance of $H$ and $H\ge0$. Then (\ref{12131753}) holds for the same $p, q, \alpha, s$.
\end{corollary}
\begin{proof}
By the Hardy inequality, we have
\begin{align*}
&|\langle Vu, u\rangle| \lesssim \left\|\frac{u}{|x|}\right\|^2 _2 \lesssim \|Du\|^2 _2 = \langle H_0 u, u\rangle \\
& |\langle ADu, u\rangle| +|\langle DAu, u\rangle| \lesssim \|Du\|^2 _2 +\left\|\frac{u}{|x|}\right\|^2 _2 \lesssim \langle H_0 u, u\rangle.
\end{align*}
Hence $H^{1/4} H^{-1/4} _0$ is bounded on $L^2 (\R^d)$. As in the proof of Theorem 2.10 in \cite{Ho}, we decompose 
\[H= - \Delta + \sum_{j=1}^{3} Y_j^*Z_j .\] Here 
\begin{gather*} 
\ Y_3 =|V|^{1/2}\sgn V, \quad  Z_3 = |V|^{1/2}, \quad Y_1^* = Aw^{-1}D|D|^{-1/2}, \\
\ Z_1 =|D|^{1/2}w, \quad  Y_2 =Z_1, \quad Z_2 =Y_1 \quad and \quad  w= \langle x \rangle^{- \tau} \quad for \quad \tau \in(1/2, 1/2 +\epsilon').
\end{gather*}
Since $P_{\ac} (H) =I$ holds under our assumptions, $Y_j$ are $H_0$-supersmooth and $Z_j$ are $H$-supersmooth (See the proof of Theorem 2.10 in \cite{Ho}). The assumption $0 \notin \sigma_{p} (H) \cup \sigma_{p} (H_0)$ follows from the Kato smoothing estimates and the assumption that zero is a regular point. Then we obtain (\ref{12131753}) by the same argument as in the proof of Corollary \ref{12131724}.
\end{proof}
Concerning the critical inverse square potential, we obtain the following result.
\begin{corollary}\label{12162148}
Let $d \ge 3$ and $H=-\Delta -\frac{(d-2)^2}{4} |x|^{-2}$. Then we have
\begin{align*}
\left\| \sum_{n=0}^ \infty{\nu_n|e^{-it\sqrt{H}} P^{\perp} _{\rad} f_n|^2} \right\|_{L^{p}_t L^{q}_x} \lesssim \| \nu_n\|_{\ell^\alpha}
\end{align*}
for the same $p, q, \alpha, s, \{f_n\}$ as in Corollary \ref{12131724}.
\end{corollary}
\begin{proof}
Since $P^{\perp} _{\rad}$ commutes with $f(H_0)$ and $f(H)$ for all $f \in L^2 _{loc} (\R)$, by repeating the proof of Theorem \ref{1292322} with the following Duhamel formula:
\begin{align*}
e^{-it\sqrt{H}} P^{\perp} _{\rad} g = &\cos t\sqrt{H_0} P^{\perp} _{\rad} g -i \frac{\sin t\sqrt{H_0}}{\sqrt{H_0}} \sqrt{H} P^{\perp} _{\rad} g \\
&- \int_{0}^{t} \frac{\sin(t-s)\sqrt{H_0}}{\sqrt{H_0}} (VP^{\perp} _{\rad} e^{-is\sqrt{H}} g) ds
\end{align*}
for all $g \in D(\sqrt{H})$, it suffices to show that $H^{1/4} H^{-1/4} _0$ is bounded on $L^2 (\R^d)$, $\frac{1}{|x|}$ is $-\Delta$-smooth and $\frac{1}{|x|} P^{\perp} _{\rad}$ is $H$-smooth. The $L^2$-boundedness of $H^{1/4} H^{-1/4} _0$ follows from the Hardy inequality as in the above proof. The $-\Delta$-smoothness of $\frac{1}{|x|}$ follows from Theorem 1 in \cite{KatoYajima}. Finally $H$-smoothness of $\frac{1}{|x|} P^{\perp} _{\rad}$ follows from Proposition 3.4 in \cite{M2}.
\end{proof}
Next we give some applications of Theorem \ref{1212216}. The first application (Corollary \ref{12131859}) includes the Schr\"odinger operator with inverse square potentials.
\begin{proof}[Proof of Corollary \ref{12131859}]
Since $\langle Vu, u\rangle \lesssim \|u\|^2 _{H^1}$ is assumed, we obtain $\langle (H+1)u, u\rangle \lesssim \langle (H_0 +1)u, u\rangle$. Hence $(H+1)^{1/4} (H_0 +1)^{-1/4}$ is bounded on $L^2 (\R^d)$. Then (\ref{12131926}) can be proved by the same way as in Corollary \ref{12131724} using Theorem \ref{1212216} and the free estimates (Theorem 3 in \cite{BLN}).
\end{proof}
\begin{corollary}\label{12131921}
Assume $d \ge3$, $A:\R^d \rightarrow \R^d, V:\R^d \rightarrow \R$ satisfy $|A(x)| + |\langle x \rangle V(x)| \lesssim \langle x \rangle^{-1-\epsilon}$, $\langle x \rangle^{1+\epsilon'}A(x) \in \dot{W}^{1/2, 2d}$, $A \in C^0(\R^d)$ for some $0<\epsilon'<\epsilon$. Let $H = -\Delta +A(x)\cdot D + D\cdot A(x) + V(x)$ or $H = (D +A(x))^2 +V(x)$. Here $D=-i\nabla$. Furthermore we assume zero is neither an eigenvalue nor a resonance of $H$, $H+1\ge0$ and $\sigma _{p} (H) \cap (-\infty, 0) =\emptyset$. Then we have (\ref{12131926}) for the same $p, q, \alpha, s$.
\end{corollary}
\begin{proof}
The proof is just the same as in Corollary \ref{12131747}. So we omit the details.
\end{proof}
By a proof similar to Corollary \ref{12162148}, we obtain the result for the critical inverse square potential.
\begin{corollary}\label{12162254}
Let $d \ge 3$ and $H=-\Delta -\frac{(d-2)^2}{4} |x|^{-2}$. Then we have
\begin{align*}
\left\| \sum_{n=0}^ \infty{\nu_n|e^{-it\sqrt{H+1}} P^{\perp} _{\rad} f_n|^2} \right\|_{L^{p}_t L^{q}_x} \lesssim \| \nu_n\|_{\ell^\alpha}
\end{align*}
for the same $p, q, \alpha, s, \{f_n\}$ as in Corollary \ref{12131859}.
\end{corollary}
As an application of the above estimates, we prove the refined Strichartz estimates for wave and the Klein-Gordon equations. If $A=V=0$, this result was proved in \cite{BLN} based on the orthonormal Strichartz estimates. First we give a result where $A=0$.
\begin{corollary}\label{12141329}
Assume $H$ is as in Corollary \ref{12131724} or as in Corollary \ref{12131747} and satisfies $A=0$. If $V\ge 0$, we obtain 
\begin{align}
\|e^{-it\sqrt{H}} u\|_{L^{2p} _t L^{2q} _x} \lesssim \|u\|_{\dot{B}^s _{2, 2\alpha}} \label{12141716}
\end{align}
for the same $p, q, \alpha, s$ as in Corollary \ref{12131724}.
\end{corollary}
\begin{proof}
Since $V\ge 0$ holds, we have $\|u\|_{p} \approx \|\|\{\phi_j (\sqrt{H}) u\}\|_{\ell^2}\|_{p}$ for all $p \in (1, \infty)$ by Proposition 2.9 in \cite{M1}. Furthermore, by $\langle Hu, u \rangle \approx \langle H_0 u, u \rangle$, the $\dot{B}^s _{2, 2\alpha}$-boundedness of the wave operator $W^{*} _{\pm}$ can be proved in the same way as in the proof of Lemma \ref{1291301}. Hence by repeating the proof of Theorem \ref{1291939}, we obtain the desired estimates.
\end{proof}
By the similar proof as above, we obtain the following estimate.
\begin{corollary} \label{12141736}
Assume $H$ is as in Corollary \ref{12131859} or as in Corollary \ref{12131921} and satisfies $A=0$. If $V\ge 0$, we obtain 
\begin{align}
\|e^{-it\sqrt{H+1}} u\|_{L^{2p} _t L^{2q} _x} \lesssim \|u\|_{B^s _{2, 2\alpha}} \label{121417100}
\end{align}
for the same $p, q, \alpha, s$ as in Corollary \ref{12131859}. 
\end{corollary}
Next we give a result including the magnetic Schr\"odinger operator.
\begin{corollary} \label{12141711}
Assume $H$ is as in Corollary \ref{12131747} and satisfies $\langle Hu, u \rangle \gtrsim \langle H_0 u, u \rangle$ and $\|V_{-}\|_{K_d} < \frac{\pi^{d/2}}{\Gamma (d/2 -1)}$. Then (\ref{12141716}) holds for the same $p, q, \alpha, s$. Here $\|\cdot\|_{K_d}$ denotes the Kato norm. See Definition 4.9 in \cite{Ho} for its definition.
\end{corollary}
\begin{proof}
Under our assumptions, the Littlewood-Paley theorem for $H$ was proved in Proposition 4.12 in \cite{Ho}. By the proof of Corollary \ref{12131747} and our assumptions, we have $\langle Hu, u \rangle \approx \langle H_0 u, u \rangle$. Then the proof is just the same as before and we omit the details.
\end{proof}
\begin{corollary}\label{12141800}
Assume $H$ is as in Corollary \ref{12131921} and satisfies $\langle (H+1) u, u \rangle \gtrsim \|u\|^2 _{H^1}$ and $\|V_{-}\|_{K_d} < \frac{\pi^{d/2}}{\Gamma (d/2 -1)}$. Then (\ref{121417100}) holds for the same $p, q, \alpha, s$.
\end{corollary}
\begin{proof}
The proof is similar so we omit the details.
\end{proof}
Next we extend the orthonormal Strichartz estimates for more general $(p, q)$. In order to prove the estimate for $(p, q)$ close to the endpoint, we need to handle derivatives which did not appear in the previous setting. Especially we need to estimate the fractional powers of $H$ or $H+1$. To overcome this difficulty, we use a method from the microlocal analysis and construct a parametrix though it only works for the Klein-Gordon equation. As our first step, we prove an abstract perturbation theorem suitable for our setting as before.
\begin{thm}\label{12141937}
Let $\mathcal{H}:= L^2 (X)$ for a $\sigma$-finite measure space $X$. Assume self-adjoint operators $H, H_0$ and a densely defined closed operator $V$ satisfy
\[H=H_0 +V, \quad D(H) \cup D(H_0) \subset D(V).\]
Furthermore we assume there exist $\theta \in \R$, densely defined injective closed operators $w$ and $W$ which satisfy the following:
\begin{align*}
&H+1, H_0 +1\ge 0, \quad (H+1)^{\theta} (H_0 +1)^{-\theta} \in \mathcal{B} (\mathcal{H}), \\
&(H_0 +1)^{\theta -1/2} (H+1)^{1/2} (H_0 +1)^{-\theta} \in \mathcal{B} (\mathcal{H}), \\
&w^{-1} (H+1)^{-\theta+1/4} V (H_0 +1)^{\theta -1/4} W^{-1} \in \mathcal{B} (\mathcal{H}).
\end{align*}
Suppose $w$ is $H$-smooth and $W$ is $H_0$-smooth.
If we have
\[\left\| \sum_{n=0}^ \infty{\nu_n|e^{-it\sqrt{H_0 +1}} f_n|^2} \right\|_{L^{p}_t L^{q}_x} \lesssim \| \nu_n\|_{\ell^\alpha}\]
for some $p, q \in [1, \infty]$, $\alpha \in (1, \infty)$ and all $\{f_n\}$: orthonormal systems in $D((H_0 +1)^{\theta})$, then
\[\left\| \sum_{n=0}^ \infty{\nu_n|e^{-it\sqrt{H +1}} f_n|^2} \right\|_{L^{p}_t L^{q}_x} \lesssim \| \nu_n\|_{\ell^\alpha}\]
holds for the same $p, q, \alpha, \theta, \{f_n\}$.
\end{thm}
\begin{proof}
By the Duhamel formula, we obtain 
\begin{align*}
fe^{-it\sqrt{H+1}} (H_0 +1)^{-\theta} = &f\cos t \sqrt{H_0 +1} (H_0 +1)^{-\theta} \\
&-i f \sin t\sqrt{H_0 +1} (H_0 +1)^{-\theta} (H_0 +1)^{\theta -1/2} \sqrt{H+1} (H_0 +1)^{-\theta} \\
& -f\int_{0}^{t} \frac{\sin (t-s)\sqrt{H_0 +1}}{\sqrt{H_0 +1}} Ve^{-is\sqrt{H+1}} (H_0 +1)^{-\theta} ds.
\end{align*}
By the assumptions, we estimate the first and second term as
\begin{align*}
&\|f\cos t \sqrt{H_0 +1} (H_0 +1)^{-\theta}\|_{\mathfrak{S}^{2\alpha'}_{x \rightarrow t, x}} \lesssim \|f\|_{L^{2p'} _t L^{2q'} _x}, \\
& \|f \sin t\sqrt{H_0 +1} (H_0 +1)^{-\theta} (H_0 +1)^{\theta -1/2} \sqrt{H+1} (H_0 +1)^{-\theta}\|_{\mathfrak{S}^{2\alpha'}_{x \rightarrow t, x}} \\
&\lesssim \|f\|_{L^{2p'} _t L^{2q'} _x} \|(H_0 +1)^{\theta -1/2} \sqrt{H+1} (H_0 +1)^{-\theta}\|_{\mathcal{B} (\mathcal{H})} \lesssim \|f\|_{L^{2p'} _t L^{2q'} _x}.
\end{align*}
For the third term, by Lemma \ref{2311152210}, we obtain
\begin{align*}
&\left\|f\int_{0}^{t} \frac{\sin (t-s)\sqrt{H_0 +1}}{\sqrt{H_0 +1}} Ve^{-is\sqrt{H+1}} (H_0 +1)^{-\theta} ds \right\|_{\mathfrak{S}^{2\alpha'}_{x \rightarrow t, x}} \\
&\lesssim \left\|f\int_{0}^{\infty} \frac{\sin (t-s)\sqrt{H_0 +1}}{\sqrt{H_0 +1}} Ve^{-is\sqrt{H+1}} (H_0 +1)^{-\theta} ds \right\|_{\mathfrak{S}^{2\alpha'}_{x \rightarrow t, x}} \\
& \lesssim \|f\|_{L^{2p'} _t L^{2q'} _x} \left\|\int_{0}^{\infty} (H_0 +1)^{\theta -1/2} \cos s\sqrt{H_0 +1} Ve^{-is\sqrt{H+1}} (H_0 +1)^{-\theta} ds \right\|_{\mathcal{B} (\mathcal{H})} \\
&+ \|f\|_{L^{2p'} _t L^{2q'} _x} \left\|\int_{0}^{\infty} (H_0 +1)^{\theta -1/2} \sin s\sqrt{H_0 +1} Ve^{-is\sqrt{H+1}} (H_0 +1)^{-\theta} ds \right\|_{\mathcal{B} (\mathcal{H})}.
\end{align*}
Concerning the second factor of the first term in the last inequality, we have
\begin{align*}
&\left\|\int_{0}^{\infty} (H_0 +1)^{\theta -1/2} \cos s\sqrt{H_0 +1} Ve^{-is\sqrt{H+1}} (H_0 +1)^{-\theta} ds \right\|_{\mathcal{B} (\mathcal{H})} \\
& \lesssim \left\|\int_{0}^{\infty} (H_0 +1)^{ -1/4} \cos s\sqrt{H_0 +1} W W^{-1} (H_0 +1)^{\theta -1/4} V\right. \\ 
&\left. \quad \quad \quad  \cdot (H+1)^{-\theta +1/4} w^{-1} w e^{-is\sqrt{H+1}} (H+1)^{\theta -1/4} (H_0 +1)^{-\theta} ds \right\|_{\mathcal{B} (\mathcal{H})} \\
& \lesssim \|W \cos t\sqrt{H_0 +1} (H_0 +1)^{ -1/4}\|_{\mathcal{B} (\mathcal{H} \rightarrow L^2 _t \mathcal{H})} \cdot \|w^{-1} (H+1)^{-\theta+1/4} V (H_0 +1)^{\theta -1/4} W^{-1} \|_{\mathcal{B} (\mathcal{H})} \\
& \quad \quad \cdot \|we^{-it\sqrt{H+1}} (H+1)^{-1/4} \|_{\mathcal{B} (\mathcal{H} \rightarrow L^2 _t \mathcal{H})} \cdot \|(H+1)^{\theta} (H_0 +1)^{-\theta}\|_{\mathcal{B} (\mathcal{H})} \\
& \lesssim 1
\end{align*}
by our assumptions. By similar calculations, we also obtain
\[\left\|\int_{0}^{\infty} (H_0 +1)^{\theta -1/2} \sin s\sqrt{H_0 +1} Ve^{-is\sqrt{H+1}} (H_0 +1)^{-\theta} ds \right\|_{\mathcal{B} (\mathcal{H})} \lesssim 1\]
Hence we have 
\[\|fe^{-it\sqrt{H+1}} (H_0 +1)^{-\theta}\|_{\mathfrak{S}^{2\alpha'}_{x \rightarrow t, x}} \lesssim \|f\|_{L^{2p'} _t L^{2q'} _x}\]
and by Lemma \ref{1113719}, we obtain the desired estimates.
\end{proof}
By the same way, we can prove the following theorem corresponding to wave equations.
\begin{thm}\label{12152027}
Let $\mathcal{H}:= L^2 (X)$ for a $\sigma$-finite measure space $X$. Assume self-adjoint operators $H, H_0$ and a densely defined closed operator $V$ satisfy
\[H=H_0 +V, \quad D(H) \cup D(H_0) \subset D(V).\]
Furthermore we assume there exist $\theta \in \R$, densely defined injective closed operators $w$ and $W$ which satisfy the following:
\begin{align*}
&H, H_0 \ge 0, \quad H^{\theta} {H_0 }^{-\theta} \in \mathcal{B} (\mathcal{H}), \\
&{H_0 }^{\theta -1/2} H^{1/2} {H_0 }^{-\theta} \in \mathcal{B} (\mathcal{H}), \\
&w^{-1} H^{-\theta+1/4} V {H_0 }^{\theta -1/4} W^{-1} \in \mathcal{B} (\mathcal{H}).
\end{align*}
Suppose $w$ is $H$-smooth and $W$ is $H_0$-smooth. If we have
\[\left\| \sum_{n=0}^ \infty{\nu_n|e^{-it\sqrt{H_0 }} f_n|^2} \right\|_{L^{p}_t L^{q}_x} \lesssim \| \nu_n\|_{\ell^\alpha}\]
for some $p, q \in [1, \infty]$, $\alpha \in (1, \infty)$ and all $\{f_n\}$: orthonormal systems in $D({H_0 }^{\theta})$, then
\[\left\| \sum_{n=0}^ \infty{\nu_n|e^{-it\sqrt{H }} f_n|^2} \right\|_{L^{p}_t L^{q}_x} \lesssim \| \nu_n\|_{\ell^\alpha}\]
holds for the same $p, q, \alpha, \theta, \{f_n\}$.
\end{thm}
We apply the above theorem to the magnetic Schr\"odinger operator. We need to estimate $w^{-1} (H+1)^{-\theta+1/4} V (H_0 +1)^{\theta -1/4} W^{-1}$ or $w^{-1} H^{-\theta+1/4} V {H_0 }^{\theta -1/4} W^{-1}$ but this seems difficult. In this paper we use the construction of a parametrix and hence we assume $A$ and $V$ are smooth. However, even though we assume the above conditions, we have no result for wave equations since $H^{\theta}$ may behave as a singular integral operator like $|D|^{\alpha}$. We only prove the result for the Klein-Gordon equation here.
Before proving Corollary \ref{1216058}, we collect some definitions for pseudodifferential operators. We use the symbol class $S(m, g)$, which was introduced by H\"ormander. See \cite{Hor} for its definitions and properties. For $a \in S(m, g)$ and $u \in \mathcal{S}$, we define
\[a^{W} (x, D_x) u(x) = (2\pi)^{-d} \int\int e^{i(x-y)\xi} a\left(\frac{x+y}{2}, \xi\right) u(y) dyd\xi .\]
Then $a^{W} (x, D_x) u \in \mathcal{S}$ and $a^{W} (x, D_x)$ is continuous on $\mathcal{S}$. We extend $a^{W} (x, D_x) : \mathcal{S}' \rightarrow \mathcal{S}'$ by continuity. For the $L^2$-boundedness of $a^{W} (x, D_x)$, we use the Calder\`on-Vaillancourt theorem (Theorem 18.6.3 in \cite{Hor}) several times.
\begin{proof}[Proof of Corollary \ref{1216058}]
We define $w=W= \langle x \rangle^{-1-\epsilon'} \langle D \rangle^{1/2}$ for sufficiently small $\epsilon' >0$. Then $w$ is $-\Delta$-smooth and $W$ is $H$-smooth by Theorem 1.1 in \cite{EGS}. Since $\langle (H+1)u, u \rangle \lesssim \|u\|^2 _{H^1}$ holds, $(H_0 +1)^{-1/2} (H+1)^{1/2}$ is bounded on $L^2 (\R^d)$ and by the complex interpolation, we obtain
\begin{align*}
&(H_0 +1)^{\theta -1/2} (H+1)^{1/2} (H_0 +1)^{-\theta} \in \mathcal{B} (L^2 (\R^d)), \\
&(H+1)^{\theta} (H_0 +1)^{-\theta} \in \mathcal{B} (L^2 (\R^d))
\end{align*}
for all $\theta \in [0, 1/2]$. Now we may assume $1/2 \le s \le 1$ since other cases are already proved in Corollary \ref{12131921} under less assumptions. By Theorem \ref{12141937} and the orthonormal Strichartz estimates for the free Hamiltonian (Theorem 6, 7 in \cite{BLN}), it suffices to show
\begin{align}
&\|\langle D \rangle^{-1/2} \langle x \rangle^{1+\epsilon'} (H+1)^{-\theta +1/4} \tilde{V} \langle D \rangle^{2(\theta -1/4)} \langle D \rangle^{-1/2} \langle x \rangle^{1+\epsilon'}\|_{\mathcal{B} (L^2 (\R^d))} \lesssim 1, \label{12191840} \\
&\tilde{V} = DA+AD+|A|^2 +V \notag
\end{align}
for all $\theta \in [1/4, 1/2]$. By our assumptions, we have
\begin{align}
(H+1)^{-\theta + 1/4} = a^{W} _{-\theta +1/4, n} (x, D_x) + \int_{\Gamma} z^{-s} (H+1-z)^{-1} r^{W} _n (x, D_x, z) dz \label{12191757}
\end{align}
for all $n \in \N$.
Here $a _{-\theta +1/4, n} \in S(m^{-\theta +1/4}, g)$ for $g=\frac{dx^2}{\langle x \rangle^2} + \frac{d\xi^2}{\langle \xi \rangle^2}$ and $m(x, \xi) = \langle \xi \rangle^2$. $\Gamma = \{\lambda e^{i\eta} \mid \lambda >0\} \cup \{\lambda e^{-i\eta} \mid \lambda >0\}$ for some $\eta \in (0, \frac{\pi}{2})$. $r_n (x, \xi, z)$ is a smooth function with respect to $(x, \xi, z)$ and also analytic with respect to $z$. $r_n (x, \xi, z)$ satisfies the following bound:
\[\langle x \rangle^{|\alpha|} \langle \xi \rangle^{|\beta|} |\partial^{\alpha} _x \partial^{\beta} _{\xi} r_n (x, \xi, z)| \lesssim_{n, k, \eta} h(x, \xi)^{n} m(x, \xi) (\sigma^{W} (H)(x, \xi) +1+\lambda)^{-1}\]
for all $(x, \xi) \in \R^{2d}$, $\alpha, \beta \in \N^d _{0}$ satisfying $|\alpha|+|\beta| \le k$, $k \in \N$, $z=\lambda + i \lambda \tan \eta$. Here $h(x, \xi)=\langle x \rangle^{-1} \langle \xi \rangle^{-1}$ holds. Equality (\ref{12191757}) is proved in Appendix B in \cite{Ybook} by using a construction of a parametrix for resolvents. Actually, in subsection 3.3, we prove a generalization of (\ref{12191757}) to some matrix-valued pseudodifferential operators and the proof of (\ref{12191757}) is easier. Hence we omit the details here.
By (\ref{12191757}), we obtain
\begin{align}
&\langle D \rangle^{-1/2} \langle x \rangle^{1+\epsilon'} (H+1)^{-\theta +1/4} \tilde{V} \langle D \rangle^{2(\theta -1/4)} \langle D \rangle^{-1/2} \langle x \rangle^{1+\epsilon'} \notag \\
&= \langle D \rangle^{-1/2} \langle x \rangle^{1+\epsilon'} (a^{W} _{-\theta +1/4, n} (x, D_x))^* \tilde{V} \langle D \rangle^{2(\theta -1/4)} \langle D \rangle^{-1/2} \langle x \rangle^{1+\epsilon'} \notag \\
& \quad \quad  + \langle D \rangle^{-1/2} \langle x \rangle^{1+\epsilon'} \left(\int_{\Gamma} z^{-s} (H+1-z)^{-1} r^{W} _n (x, D_x, z) dz \right)^* \tilde{V} \langle D \rangle^{2(\theta -1/4)} \langle D \rangle^{-1/2} \langle x \rangle^{1+\epsilon'} \label{12191852}
\end{align}
For the first term in (\ref{12191852}), by $a _{-\theta +1/4, n} \in S(m^{-\theta +1/4}, g)$, the Calder\`on-Vaillancourt theorem and symbolic calculus, we obtain
\[\|\langle D \rangle^{-1/2} \langle x \rangle^{1+\epsilon'} (a^{W} _{-\theta +1/4, n} (x, D_x))^* \tilde{V} \langle D \rangle^{2(\theta -1/4)} \langle D \rangle^{-1/2} \langle x \rangle^{1+\epsilon'}\|_{\mathcal{B} (L^2 (\R^d))} \lesssim 1.\]
For the adjoint of the second term, we transform it into
\begin{align}
&\langle x \rangle^{1+\epsilon'} \langle D \rangle^{-1/2} \langle D \rangle^{2(\theta -1/4)} \tilde{V} \int_{\Gamma} z^{-s} (H+1-z)^{-1} r^{W} _n (x, D_x, z) dz \langle x \rangle^{1+\epsilon'} \langle D \rangle^{-1/2} \notag \\
& = \langle x \rangle^{1+\epsilon'} \langle D \rangle^{-1/2} \langle D \rangle^{2(\theta -1/4)} \tilde{V} \langle D \rangle^{-1} \cdot \langle D \rangle (H+1)^{-1/2} \notag \\
& \quad \quad \quad \cdot \int_{\Gamma} z^{-s} (H+1)^{1/2} (H+1-z)^{-1} (r^{W} _n (x, D_x, z)\langle x \rangle^{1+\epsilon'} \langle D \rangle^{-1/2}) dz. \label{12192103}
\end{align}
The first factor is bounded by the Calder\`on-Vaillancourt theorem since we assume $\theta \in [1/4, 1/2]$, $|\partial^{\alpha} _x A(x)| \lesssim \langle x \rangle^{-2-\epsilon -|\alpha|}$ and $|\partial^{\alpha} _x V(x)| \lesssim \langle x \rangle^{-2-\epsilon -|\alpha|}$. For the adjoint of the second factor, we use the same parametrix and obtain
\begin{align*}
(H+1)^{-1/2} \langle D \rangle = \tilde{a}^{W} _n (x, D_x) \langle D \rangle + \int_{\Gamma} z^{-s} (H+1-z)^{-1} \tilde{r}^{W} _n (x, D_x, z) \langle D \rangle dz
\end{align*}
for some $\tilde{a} _n \in S(m^{-1/2}, g)$ and $\tilde{r} _n$ satisfying the same estimates as $r_n$. Then $\tilde{a}^{W} _n (x, D_x) \langle D \rangle$ is bounded by the Calder\`on-Vaillancourt theorem and the latter term is estimated as
\begin{align*}
&\left\| \int_{\Gamma} z^{-s} (H+1-z)^{-1} \tilde{r}^{W} _n (x, D_x, z) \langle D \rangle dz \right\|_{\mathcal{B} (L^2 (\R^d))} \\
&\lesssim \int_{0}^{\infty} \lambda^{-s} (1+ \lambda)^{-1} d \lambda < \infty
\end{align*}
for sufficiently large $n$
since $H \ge \gamma$ implies $\|(H+1-z)^{-1}\|_{\mathcal{B} (L^2 (\R^d))} \lesssim 1$ uniformly in $z \in \Gamma$ and the estimates on $\tilde{r} _n$ implies $\|\tilde{r}^{W} _n (x, D_x, z) \langle D \rangle\|_{\mathcal{B} (L^2 (\R^d))} \lesssim (1+ \lambda)^{-1}$. Hence we obtain 
\[\|(H+1)^{-1/2} \langle D \rangle\|_{\mathcal{B} (L^2 (\R^d))} \lesssim 1.\]
Finally, for the third factor in (\ref{12192103}), we estimate it as
\begin{align*}
&\left\|\int_{\Gamma} z^{-s} (H+1)^{1/2} (H+1-z)^{-1} (r^{W} _n (x, D_x, z)\langle x \rangle^{1+\epsilon'} \langle D \rangle^{-1/2}) dz \right\|_{\mathcal{B} (L^2 (\R^d))} \\
& \lesssim \int_{0}^{\infty} \lambda^{-s} (1+ \lambda)^{-1} d \lambda < \infty
\end{align*}
for sufficiently large $n$ since we have $\|(H+1)^{1/2} (H+1-z)^{-1}\|_{\mathcal{B} (L^2 (\R^d))} \lesssim 1$ uniformly in $z \in \Gamma$ by the same reason as above. Then we obtain
\begin{align*}
&\left\|\langle x \rangle^{1+\epsilon'} \langle D \rangle^{-1/2} \langle D \rangle^{2(\theta -1/4)} \tilde{V} \int_{\Gamma} z^{-s} (H+1-z)^{-1} r^{W} _n (x, D_x, z) dz \langle x \rangle^{1+\epsilon'} \langle D \rangle^{-1/2} \right\|_{\mathcal{B} (L^2 (\R^d))} \\
& \quad \quad \quad \quad \quad \quad \quad \quad \quad \quad \quad \quad \quad \quad \quad \quad \quad \quad \quad \quad \quad \quad \quad \quad \quad \quad \quad \quad \quad \quad \quad \quad \quad \quad \quad \quad  \lesssim 1
\end{align*}
and hence (\ref{12191840}) holds.
\end{proof}
\begin{remark}
The assumption $|\partial^{\alpha} _x A(x)| \lesssim \langle x \rangle^{-2-\epsilon -|\alpha|}$ is needed to estimate the first term in (\ref{12191852}). In order to weaken this assumption to $|\partial^{\alpha} _x A(x)| \lesssim \langle x \rangle^{-1-\epsilon -|\alpha|}$, we need to use the $H_0$ and $H$-smoothness of $\langle x \rangle^{-1/2 -\epsilon} |D|^{1/2}$ instead of $\langle x \rangle^{-1 -\epsilon} \langle D \rangle^{1/2}$. However, if we use that, it seems difficult to use the microlocal analysis since there appears singular integral operators. Furthermore, in the estimate of the second factor in (\ref{12192103}), it is not enough to apply the Fefferman-Phong inequality since it only implies weaker estimates:
\[H \ge A \langle D \rangle^2 -B\]
for some $A, B >0$.
\end{remark}
As a corollary, we obtain the following refined Strichartz estimates. The proof is the same as that of Corollary \ref{12141800}. Hence we omit the details.
\begin{corollary}\label{12162131}
Assume $H$ be as in Corollary \ref{1216058} and $\|V_{-}\|_{K_d} < \frac{\pi^{d/2}}{\Gamma (d/2 -1)}$. Then
\[\|e^{-it\sqrt{H+1}} u\|_{L^{q} _t L^{r} _x} \lesssim \|u\|_{B^s _{2, 2\beta}}\]
holds for the same $q, r, \beta, s$ as in Corollary \ref{1216058}.
\end{corollary}
\subsection{\textbf{Semi-relativistic Hartree equation for infinitely many particles}}
In this subsection, we prove the local existence of a solution to the semi-relativistic Hartree equation with electromagnetic potentials for infinitely many particles. The ordinary semi-relativistic Hartree equations are given by 
\begin{align}
\left\{
\begin{array}{l}
i\partial_t u(t)= \sqrt{H+1} u(t) + (w*|u|^2)u \\
u(0)=u_0
\end{array}
\right.
\tag{H1}\label{12271728}
\end{align}
or
\begin{align}
\left\{
\begin{array}{l}
i\partial_t u_j (t)= \sqrt{H+1} u(t) + (w* \displaystyle \sum_{k=1}^{N} |u_k|^2)u_j \\
u_j (0)=u_{0, j}
\end{array}
\right.
\tag{H2}\label{12271731}
\end{align}
for $j =1, \dots ,N$. If $H=-\Delta$, (\ref{12271728}) describes boson stars and (\ref{12271731}) describes a white dwarf which has $N$ electrons. See \cite{L} and \cite{FL2} for (\ref{12271728}) and \cite{FL1} for (\ref{12271731}) for more details about these equations. We consider the infinitely many particle version of (\ref{12271731}), which is given by
\begin{align}
\left\{
\begin{array}{l}
i\partial_t \gamma=[\sqrt{H+1} +w*\rho_{\gamma},\gamma] \\
\gamma(0)=\gamma_0
\end{array}
\right.
\tag{SRH}\label{12271742}
\end{align}
for an operator-valued function $\gamma$. If $H=-\Delta$, the local existence of a solution to (\ref{12271742}) was proved in \cite{BLN}. We follow the argument in \cite{BLN} but some modifications are needed since $e^{-it \sqrt{H+1}}$ does not commute with Fourier multipliers $f(D)$.
First we prove the inhomogeneous orthonormal Strichartz estimates, which was first proved in \cite{FLLS}.
\begin{lemma} \label{12271847}
Assume $H$, $q, r, \beta, s$ are as in Corollary \ref{12131859}, Corollary \ref{12131921} or Corollary \ref{1216058} and satisfies $\langle (H+1)u, u \rangle \approx \langle (H_0 +1) u, u \rangle$. Then the following results hold.
\begin{enumerate}
\item For all $V \in L^{(q/2)'} _t L^{(r/2)'} _x$, 
\[\left\| \int_{\R} e^{it \sqrt{H+1}} (H+1)^{-s/2} V(t) (H+1)^{-s/2} e^{-it \sqrt{H+1}} dt \right\|_{\mathfrak{S}^{\beta'}} \lesssim \|V\|_{L^{(q/2)'} _t L^{(r/2)'} _x}\]

\item For $\gamma_R (t) = \int_{0}^{t} e^{-i(t-t') \sqrt{H+1}} R(t') e^{i(t-t') \sqrt{H+1}} dt'$,
\[\|\rho ((H+1)^{-s/2} \gamma_R (t) (H+1)^{-s/2})\|_{L^{q/2} _t L^{r/2} _x} \lesssim \left\| \int_{\R} e^{it \sqrt{H+1}} |R(t)| e^{-it \sqrt{H+1}} dt \right\|_{\mathfrak{S}^{\beta}}\]
\end{enumerate}
\end{lemma}
We note that $\langle (H+1)u, u \rangle \approx \langle (H_0 +1) u, u \rangle$ holds if $H$ is as in Corollary \ref{1216058} and $\langle (H+1)u, u \rangle \lesssim \langle (H_0 +1) u, u \rangle$ holds if $H$ is as in Corollary \ref{12131921} without further assumptions.
\begin{proof}
By the duality principle, the orthonormal Strichartz estimates and $\langle (H+1)u, u \rangle \approx \langle (H_0 +1) u, u \rangle$, we obtain
\[\|\rho ((H+1)^{-s/2} \gamma (t) (H+1)^{-s/2})\|_{L^{q/2} _t L^{r/2} _x} \lesssim \|\gamma_0\|_{\mathfrak{S}^{\beta}}\]
for $\gamma (t) = e^{-it \sqrt{H+1}} \gamma_0 e^{it \sqrt{H+1}}$. Hence we obtain
\begin{align*}
&\left| \tr \left( \gamma_0 \int_{\R} e^{it \sqrt{H+1}} (H+1)^{-s/2} V(t) (H+1)^{-s/2} e^{-it \sqrt{H+1}} dt \right)\right| \\
& = \left| \int_{\R} \int_{\R^d} \rho ((H+1)^{-s/2} \gamma (t) (H+1)^{-s/2}) V(t, x) dtdx \right| \lesssim  \|\gamma_0\|_{\mathfrak{S}^{\beta}} \|V\|_{L^{(q/2)'} _t L^{(r/2)'} _x}
\end{align*}
and by the duality of Schatten spaces, we finish the proof of $1$. $2$ can be proved in the similar way so we omit the details.
\end{proof}
Before stating our result, we define the Sobolev-type Schatten spaces. For $\gamma \in \mathcal{B} (L^2)$, we define $\|\gamma\|_{\mathfrak{S}^{\beta, s}} := \|\langle D \rangle^{s} \gamma \langle D \rangle^{s}\|_{\mathfrak{S}^{\beta}}$ and $\|\gamma\|_{\tilde{\mathfrak{S}}^{\beta, s}} :=\|(H+1)^{s/2} \gamma (H+1)^{s/2}\|_{\mathfrak{S}^{\beta}}$. Under the assumptions in Lemma \ref{12271847}, $\|\gamma\|_{\mathfrak{S}^{\beta, s}} \approx \|\gamma\|_{\tilde{\mathfrak{S}}^{\beta, s}}$ holds.
\begin{thm} \label{1228226}
Assume $H$, $q, r, \beta, s$ are as in Lemma \ref{12271847} and $w \in B^{s+\delta} _{(r/2)', \infty}$ for some $\delta >0$. Then for all $\gamma_0 \in \mathfrak{S}^{\beta, s}$, there exists a unique local solution to (\ref{12271742}) satisfying $\gamma \in C([-T, T]; \mathfrak{S}^{\beta, s})$ and $\rho_{\gamma} \in L^{q/2} _T L^{r/2} _x$ for some $T >0$.
\end{thm}
\begin{proof}
Let $R >0$ be such that $\|\gamma_0\|_{\tilde{\mathfrak{S}}^{\beta, s}} < R$, $T =T(R)$ be chosen later. Set
\[ X := \left\{ (\gamma , \rho) \in C([0, T]; \tilde{\mathfrak{S}}^{\beta, s}) \times L^{q/2} ([0, T]; L^{r/2} _x) \mid \|\gamma\|_{C([0, T]; \tilde{\mathfrak{S}}^{\beta, s})} + \|\rho\|_{L^{q/2} ([0, T]; L^{r/2} _x)} \le CR \right\}.\]
Here $C>0$ is later chosen independent of $R$.
We define 
\begin{align*}
&\Phi (\gamma , \rho) =(\Phi_1 (\gamma , \rho), \rho [\Phi_1 (\gamma, \rho)]), \\
& \Phi_1 (\gamma , \rho)(t) = e^{-it \sqrt{H+1}} \gamma_0 e^{it \sqrt{H+1}} -i \int_{0}^{t} e^{-i(t-s)\sqrt{H+1}} [w*\rho , \gamma] e^{i(t-s)\sqrt{H+1}} ds .
\end{align*}
Then by Lemma \ref{12271847} and $\|\gamma\|_{\mathfrak{S}^{\beta, s}} \approx \|\gamma\|_{\tilde{\mathfrak{S}}^{\beta, s}}$, we obtain
\begin{align*}
\| \Phi_1 (\gamma , \rho)\|_{C_T \tilde{\mathfrak{S}}^{\beta, s}} &\le \|\gamma_0\|_{\tilde{\mathfrak{S}}^{\beta, s}} + \int_{0}^{t} \|[w*\rho , \gamma]\|_{\tilde{\mathfrak{S}}^{\beta, s}} ds \\
& \lesssim R + \int_{0}^{t} \|[w*\rho , \gamma]\|_{\mathfrak{S}^{\beta, s}} ds \\
& \le R + \int_{0}^{t} \left(\|\langle D \rangle^s w*\rho \langle D \rangle^{-s}\|_{\mathcal{B} (L^2)} + \|\langle D \rangle^{-s} w*\rho \langle D \rangle^{s}\|_{\mathcal{B} (L^2)} \right) \|\gamma\|_{\mathfrak{S}^{\beta, s}} ds.
\end{align*}
Now we use the following inequality $\|fg\|_{H^t} \lesssim \|f\|_{B^{|t|+\delta} _{\infty, \infty}} \|g\|_{H^t}$ for all $t \in \R$ (see \cite{T}, \cite{BLN}) and obtain
\[\|\langle D \rangle^{\pm s} w*\rho \langle D \rangle^{\mp s} u\|_{2} \lesssim \|w*\rho\|_{B^{s+\delta} _{\infty, \infty}} \|u\|_{2} \lesssim\|w\|_{B^{s+\delta} _{(r/2)', \infty}} \|\rho\|_{r/2} \|u\|_2,\]
where in the last inequality we have used Theorem 2.1 in \cite{KS}. Hence
\begin{align*}
\| \Phi_1 (\gamma , \rho)\|_{C_T \tilde{\mathfrak{S}}^{\beta, s}} \lesssim R + \|w\|_{B^{s+\delta} _{(r/2)', \infty}} \|\gamma\|_{C_T \tilde{\mathfrak{S}}^{\beta, s}} \|\rho\|_{L^{q/2} ([0, T]; L^{r/2} _x)} T^{1/(\frac{q}{2})'}
\end{align*}
holds. Similarly, by using Lemma \ref{12271847}, we obtain
\[\|\rho [\Phi_1 (\gamma, \rho)]\|_{L^{q/2} ([0, T]; L^{r/2} _x)} \lesssim R + \|w\|_{B^{s+\delta} _{(r/2)', \infty}} \|\gamma\|_{C_T \tilde{\mathfrak{S}}^{\beta, s}} \|\rho\|_{L^{q/2} ([0, T]; L^{r/2} _x)} T^{1/(\frac{q}{2})'}.\]
Thus $\Phi :X \rightarrow X$ if we take $C$ large and $T$ small enough. Similarly we can prove $\Phi$ is a contraction on $X$ and has a unique fixed point. Since $\|\gamma\|_{\mathfrak{S}^{\beta, s}} \approx \|\gamma\|_{\tilde{\mathfrak{S}}^{\beta, s}}$ holds, we have a unique solution $\gamma$ satisfying $\gamma \in C_T \mathfrak{S}^{\beta, s}$ and $\rho_{\gamma} \in L^{q/2} ([0, T]; L^{r/2} _x)$.
\end{proof}
\begin{remark}
As in the case of higher order equations, it is possible to prove that there exists a family of unitary operators $\{U(t)\}$ on $L^2 (\R^d)$ satisfying 
\begin{enumerate}
\item $U(t)$ is a strongly continuous bounded operator on $H^s$.

\item $\gamma (t)= U(t) \gamma_0 U^* (t)$ holds on $[0, T]$.
\end{enumerate}
However $\{U(t)\}$ are not unitary on $H^s$ generally and it seems difficult to show the global existence of a solution. 
\end{remark}
\subsection{\textbf{Dirac equation}}
In this subsection, we prove the orthonormal Strichartz estimates for the Dirac equation with potentials. Concerning the ordinary Strichartz estimates for the free massive Dirac operator, see \cite{MNO}, \cite{DF} and references therein. For the free massless Dirac operator, see \cite{DF} and references therein. The Strichartz estimates for the Dirac operator with potentials are proved in \cite{DF} based on the perturbation method by \cite{RS} and the Kato smoothing estimates. Since there seems to be no result concerning the orthonormal Strichartz estimates for the Dirac equation, we first deduce the estimates for the free massive and massless Dirac operators. Then we prove the estimates for the perturbed massive Dirac operators based on the Kato smoothing estimates and the microlocal analysis. During  the rest of this subsection, we always assume $d=3$ for the sake of simplicity. First we define the massless and massive Dirac operators. Set
\begin{align*}
&\alpha_k = 
\begin{pmatrix}
0 & \sigma_k \\
\sigma_k & 0
\end{pmatrix}, \quad 
\sigma_1 =
\begin{pmatrix}
0 & 1\\
1 & 0
\end{pmatrix}, \quad
\sigma_2 =
\begin{pmatrix}
0 & -i\\
i & 0
\end{pmatrix}, \quad
\sigma_3 =
\begin{pmatrix}
1 & 0\\
0 & -1
\end{pmatrix} \\
& \beta =
\begin{pmatrix}
I_2 & 0\\
0 & -I_2
\end{pmatrix}.
\end{align*} 
Here $I_2$ denotes the $2 \times 2$ identity matrix. Then the massless Dirac operators is defined by
\begin{align*}
\mathcal{D} = \sum_{k=1}^3 \alpha_k D_k
\end{align*}
and the massive Dirac operator is defined by $\mathcal{D} + \beta$. Note that $\mathcal{D}$ and $\mathcal{D} + \beta$ are self-adjoint operators on $\mathcal{H} =L^2 (\R^3; \C^4)$. In order to establish the vector-valued duality principle, which is an analogue of the ordinary duality principle by Frank-Sabin \cite{FS}, we define density functions as follows.
\begin{defn}\label{1219027}
Let $\gamma \in \mathfrak{S}^{\alpha} (\mathcal{H})$ for some $\alpha \in [1, \infty]$. If there exists $\rho_{\gamma} \in L^1 _{loc} (\R^3 ;\C^4)$ satisfying $\tr (f\gamma) = \int_{\R^3} \rho_{\gamma} (x)f(x)dx$ for all simple functions $f: \R^3 \rightarrow \C^4$, we call $\rho_{\gamma}$ the density function of $\gamma$. Here, for $f(x)=(f_1 (x), f_2 (x), f_3 (x), f_4 (x))$, we define $\rho_{\gamma} (x)f(x)=\sum_{j=1}^4 ({\rho_{\gamma}})_j (x)f_j (x)$. We also regard $f$ as a bounded operator on $\mathcal{H}$ by considering it as a matrix-valued multiplication operator:
\[f= \begin{pmatrix}
f_1              & &   &\text{\huge{0}} \\
     &f_2                            \\
     &        &f_3                   \\    
 {\text{\huge{0}}}  & &  &f_4   \\  
\end{pmatrix}\] 
\end{defn}
Note that density functions are uniquely determined by $\gamma$ if they exist by the fundamental lemma of the calculus of variations. They are also denoted by $\rho (\gamma)$. Next we prove the vector-valued duality principle.
\begin{lemma}\label{1219111}
Assume $s \in \R$. Suppose there exists a symmetric operator $V$ on $\mathcal{H}$ satisfying
$D(\mathcal{D} + \beta ) \subset D(V)$ and $\mathcal{D} + \beta +V$ is self-adjoint on $\mathcal{H}$.
Then the following are equivalent:
\begin{enumerate}
\item $\|\rho (e^{-it(\mathcal{D} + \beta +V)} \{(\mathcal{D} + \beta)^{2}\}^s \gamma \{(\mathcal{D} + \beta)^{2}\}^s e^{it(\mathcal{D} + \beta +V)})\|_{L^p _t L^q _x} \lesssim \|\gamma\|_{\mathfrak{S}^{\alpha}}$

\item $\|fe^{-it(\mathcal{D} + \beta +V)} \{(\mathcal{D} + \beta)^{2}\}^s\|_{\mathfrak{S}^{2\alpha'} _{x \rightarrow t, x}} \lesssim \|f\|_{L^{2p'} _t L^{2q'} _x}$
\end{enumerate}
for $p, q \in [1, \infty]$ and $\alpha \in (1, \infty)$.
\end{lemma}
In the above lemma, since $(\mathcal{D} + \beta)^2 = (1-\Delta)I_4 >0$ holds, $\{(\mathcal{D} + \beta)^{2}\}^s$ are well-defined for all $s \in \R$. 
\begin{proof}
$1 \Rightarrow 2$ is proved by the similar way as in the proof of Lemma 3.1 in \cite{Ha}. We only need to set $A(t) =e^{-it(\mathcal{D} + \beta +V)} \{(\mathcal{D} + \beta)^{2}\}^s$ and substitute $\diag \{|f_1|^2, |f_2|^2, |f_3|^2, |f_4|^2\}$ for $|f|^2$. However note that all computations are done on $\mathcal{H}$ not on $L^2 (\R^d)$. $2 \Rightarrow 1$ is proved as in the proof of Lemma 4.6 in \cite{Ho}. We just substitute $\diag \{|f_1|^{1/2}, |f_2|^{1/2}, |f_3|^{1/2}, |f_4|^{1/2}\}$ for $f^{1/2}$. Hence we omit the details of the proof.
\end{proof}
Since we obtain
\[\tr (f | u \rangle \langle u |) = \int_{\R^3} \sum_{j} f_j (x) |u_j (x)|^2 dx\]
for all $u \in \mathcal{H}$ by an easy computation, $\rho (| u \rangle \langle u |) =(|u_1|^2, |u_2|^2, |u_3|^2, |u_4|^2 )$ holds. If we take $\gamma = \displaystyle \sum_{n=0}^{\infty} \nu_n |f_n \rangle \langle f_n |$ for some orthonormal system $\{f_n\}$ in $\mathcal{H}$, the inequality 1 in Lemma \ref{1219111} becomes 
\begin{align*}
\left\| \sum_{n=0}^{\infty} \nu_n |e^{-it(\mathcal{D} + \beta +V)} \{(\mathcal{D} + \beta)^{2}\}^s f_n|^2 \right\|_{L^p _t L^q _x} \lesssim \|\{\nu_n\}\|_{\ell^{\alpha}}
\end{align*}
for $|e^{-it(\mathcal{D} + \beta +V)} \{(\mathcal{D} + \beta)^{2}\}^s f_n|^2 = \displaystyle \sum_{j=1}^4 |(e^{-it(\mathcal{D} + \beta +V)} \{(\mathcal{D} + \beta)^{2}\}^s f_n)_j|^2$. This implies
\begin{align*}
\left\| \sum_{n=0}^{\infty} \nu_n |e^{-it(\mathcal{D} + \beta +V)} g_n|^2 \right\|_{L^p _t L^q _x} \lesssim \|\{\nu_n\}\|_{\ell^{\alpha}}
\end{align*}
for $g_n = \{(\mathcal{D} + \beta)^{2}\}^s f_n$, which is an orthonormal system in $H^{-2s} (\R^3 ; \C^4)$. By using the above argument and the orthonormal Strichartz estimates for the Klein-Gordon equation, we prove the orthonormal Strichartz estimates for the free massive Dirac equation.
\begin{thm} \label{12211346}
Let $q, r, s, \beta$ be as in Corollary \ref{1216058} with $d=3$. Then we have
\[\left\| \sum_{n=0}^{\infty} \nu_n |e^{-it(\mathcal{D} + \beta)} g_n|^2 \right\|_{L^{q/2} _t L^{r/2} _x} \lesssim \|\{\nu_n\}\|_{\ell^{\beta}}\]
for all orthonormal systems $\{g_n\}$ in $H^s (\R^3 ; \C^4)$.
\end{thm}
\begin{proof}
As in \cite{MNO}, we have the following explicit formula:
\[e^{-it(\mathcal{D} + \beta)} =I_4 \cos t \sqrt {1-\Delta} -i(\mathcal{D} + \beta) \frac{\sin t\sqrt{1-\Delta}}{\sqrt{1-\Delta}}\]
as an operator on $\mathcal{H}$. We decompose it into some parts which are either of the following:
\begin{align*}
& g=(g_1, g_2, g_3, g_4) \mapsto (0, \dots, \underbrace{\cos t \sqrt{1-\Delta} g_j}_{j}, \cdots, 0), \\
& g=(g_1, g_2, g_3, g_4) \mapsto (0, \dots, \underbrace{\psi (D) \sin t \sqrt{1-\Delta} g_j}_{k}, \cdots, 0).
\end{align*}
Here $j, k \in \{1, 2, 3, 4\}, \psi \in C^{\infty} (\R)$ and it satisfies $|\partial^{\alpha} _{\xi} \psi (\xi)| \lesssim \langle \xi \rangle^{-|\alpha|}$. Since the former term is easier to treat, we only focus on the latter term. In order to use Lemma \ref{1219111}, we calculate as follows:
\begin{align*}
&f \cdot [ g=(g_1, g_2, g_3, g_4) \mapsto (0, \dots, \underbrace{\psi (D) \sin t \sqrt{1-\Delta} g_j}_{k}, \cdots, 0)] \cdot \{(\mathcal{D} + \beta)^2\}^{-s/2} \\
&= [g \mapsto (0, \dots, \underbrace{f_k \sin t \sqrt{1-\Delta} \{(\mathcal{D} + \beta)^2\}^{-s/2} \psi (D)g_j}_k, \dots, 0) ] \\
& =[v \mapsto (0, \dots, \underbrace{v}_k, \dots, 0)] \circ [u \mapsto f_k \sin t \sqrt{1-\Delta} (1-\Delta)^{-s/2} u] \circ [g_j \mapsto \psi (D) g_j] \circ [g \mapsto g_j] \\
&=T_1 \circ T_2 \circ T_3 \circ T_4.
\end{align*}
Here $T_1 : L^2 _t L^2 (\R^3; \C) \rightarrow L^2 _t L^2 (\R^3; \C^4)$, $T_3 : L^2 (\R^3; \C) \rightarrow  L^2 (\R^3; \C)$ and $T_4 : L^2 (\R^3; \C^4) \rightarrow L^2 (\R^3; \C)$ are bounded operators. Furthermore, concerning $T_2$, we have the following bound
\[\|[u \mapsto f_k \sin t \sqrt{1-\Delta} (1-\Delta)^{-s/2} u]\|_{\mathfrak{S}^{2\beta'} (L^2 (\R^3; \C) \rightarrow L^2 _t L^2 (\R^3; \C))} \lesssim \|f_k\|_{L^{2(q/2)'} _t L^{2(r/2)'} _x}\]
by the orthonormal Strichartz estimates for the Klein-Gordon equation. Hence we obtain
\[\|T_1 \circ T_2 \circ T_3 \circ T_4 \|_{\mathfrak{S}^{2\beta'} (\mathcal{H} \rightarrow L^2 _t \mathcal{H})} \lesssim \|f_k\|_{L^{2(q/2)'} _t L^{2(r/2)'} _x}\]
By taking a summation for $j$ and $k$, we have
\[\|fe^{-it(\mathcal{D} + \beta)} \{(\mathcal{D} + \beta)^2\}^{-s/2}\|_{\mathfrak{S}^{2\beta'} (\mathcal{H} \rightarrow L^2 _t \mathcal{H})} \lesssim \|f\|_{L^{2(q/2)'} _t L^{2(r/2)'} _x} .\]
Then by Lemma \ref{1219111}, we obtain the desired estimates.
\end{proof}
By repeating a similar computation, we obtain the following results concerning the massless Dirac equation. We omit their proofs here.
\begin{lemma} \label{12211616}
Assume $s \in \R$. Suppose there exists a symmetric operator $V$ on $\mathcal{H}$ satisfying
$D(\mathcal{D}) \subset D(V)$ and $\mathcal{D} +V$ is self-adjoint on $\mathcal{H}$.
Then the following are equivalent:
\begin{enumerate}
\item $\|\rho (e^{-it(\mathcal{D} + V)} \{\mathcal{D}^{2}\}^s \gamma \{\mathcal{D}^{2}\}^s e^{it(\mathcal{D} +V)})\|_{L^p _t L^q _x} \lesssim \|\gamma\|_{\mathfrak{S}^{\alpha}}$

\item $\|fe^{-it(\mathcal{D} +V)} \{\mathcal{D}^{2}\}^s\|_{\mathfrak{S}^{2\alpha'} _{x \rightarrow t, x}} \lesssim \|f\|_{L^{2p'} _t L^{2q'} _x}$
\end{enumerate}
for $p, q \in [1, \infty]$ and $\alpha \in (1, \infty)$.
\end{lemma}
\begin{thm} \label{12211623}
Let $q, r, s, \beta$ satisfy either of the following:
\begin{enumerate}
\item $(q, r)$ is a $\frac{3-1}{2}$ admissible pair satisfying $2 \le r < 6$. $\{g_n\}$ is an orthonormal system in $\dot{H}^s$ with $s= \frac{3+1}{2} (1/2 -1/r)$. $\beta = \frac{2r}{r+2}$.

\item $(q, r)$ is a $\frac{3-1}{2}$ admissible pair satisfying $r \in [6, \infty)$. $\{g_n\}$ is an orthonormal system in $\dot{H}^s$ with $s= \frac{3+1}{2} (1/2 -1/r)$. $1 \le \beta <q/2$.
\end{enumerate}
Then we have
\[\left\| \sum_{n=0}^{\infty} \nu_n |e^{-it\mathcal{D}} g_n|^2 \right\|_{L^{q/2} _t L^{r/2} _x} \lesssim \|\{\nu_n\}\|_{\ell^{\beta}}\]
for all orthonormal systems $\{g_n\}$ in $\dot{H}^s (\R^3 ; \C^4)$.
\end{thm}
Finally, as an application of the above duality principle and the orthonormal Strichartz estimates for the free Dirac operator, we prove the estimates for the massive Dirac operator with potentials. As in the case of the Klein-Gordon equation, we employ the method from the microlocal analysis. Thus it seems difficult to treat the massless Dirac equation in the same way.
\begin{proof}[Proof of Theorem \ref{12212051}]
Since we assume $|\partial^{\beta} _x V(x)| \lesssim \epsilon \langle x \rangle^{-2-|\beta|}$ for all $|\beta| \le 1$, we have $(\mathcal{D} + \beta +V)^2 \approx (\mathcal{D} + \beta)^2$. The self-adjointness of $\mathcal{D} + \beta +V$ and its spectral property are proved in subsection 2.3 in \cite{DF}. By the Duhamel formula, we have
\begin{align*}
&fe^{-it(\mathcal{D} + \beta +V)} \{(\mathcal{D} + \beta)^2\}^{-s/2} = fe^{-it(\mathcal{D} + \beta)} \{(\mathcal{D} + \beta)^2\}^{-s/2} \\
& \quad \quad \quad \quad \quad \quad \quad \quad \quad \quad \quad \quad \quad  -if e^{-it(\mathcal{D} + \beta)} \int_{0}^{t} e^{is(\mathcal{D} + \beta)} Ve^{-is(\mathcal{D} + \beta +V)} \{(\mathcal{D} + \beta)^2\}^{-s/2} ds
\end{align*}
for all $f: \R \times \R^3 \rightarrow \C^4$. For the first term, by Theorem \ref{12211346}, we obtain
\begin{align}
\|fe^{-it(\mathcal{D} + \beta)} \{(\mathcal{D} + \beta)^2\}^{-s/2}\|_{\mathfrak{S}^{2\beta'} (\mathcal{H} \rightarrow L^2 _t \mathcal{H})} \lesssim \|f\|_{L^{2(q/2)'} _t L^{2(r/2)'} _x}. \label{12212304}
\end{align}
For the latter term, by Lemma \ref{2311152210} and (\ref{12212304}), it suffices to show
\begin{align}
\left\| \int_{0}^{\infty} e^{is(\mathcal{D} + \beta)} \{(\mathcal{D} + \beta)^2\}^{s/2} Ve^{-is(\mathcal{D} + \beta +V)} \{(\mathcal{D} + \beta)^2\}^{-s/2} ds\right\|_{\mathcal{B} (\mathcal{H})} \lesssim 1. \label{12212311}
\end{align}
Since $\langle x \rangle^{-1}$ is $\mathcal{D} + \beta$-smooth by Proposition 2.10 in \cite{DF}, it suffices to show
\begin{align*}
\|\langle x \rangle \{(\mathcal{D} + \beta)^2\}^{s/2} Ve^{-is(\mathcal{D} + \beta +V)} \{(\mathcal{D} + \beta)^2\}^{-s/2}\|_{\mathcal{B} (\mathcal{H} \rightarrow L^2 _t \mathcal{H})} \lesssim 1
\end{align*}
to obtain (\ref{12212311}). Furthermore, by $(\mathcal{D} + \beta +V)^2 \approx (\mathcal{D} + \beta)^2$ and the $\mathcal{D} + \beta +V$-smoothness of $\langle x \rangle^{-1}$ (Proposition 2.10 in \cite{DF}), it is reduced to
\begin{align}
\|\langle x \rangle \{(\mathcal{D} + \beta)^2\}^{s/2} V \{(\mathcal{D} + \beta +V)^2\}^{-s/2} \langle x \rangle\|_{\mathcal{B} (\mathcal{H})} \lesssim 1. \label{12212324}
\end{align}   
If $s=0$, this is obvious by our assumptions. Hence we assume $0<s \le1$. Then, for $H=(\mathcal{D} + \beta +V)^2$, we have
\begin{align*}
H^{-s/2} = \frac{1}{2\pi i} \int_{\Gamma} z^{-s/2} (z-H)^{-1} dz
\end{align*}
by the Cauchy theorem. Here $\Gamma = \{re^{i\theta} \mid r>0\} \cup \{re^{-i\theta} \mid r>0\}$ for some $\theta \in (0, \pi/2)$ and the integral is absolutely convergent since $H \gtrsim \gamma$ for some $\gamma >0$. Now we set $H=(H^{ij})_{i, j \in \{1, 2, 3, 4\}}$ and define $p^{ii} _k (z)$ as
\[\sigma^{W} (H^{ii} -z) \#^{W} (p^{ii} _0 (z) + p^{ii} _1 (z) + \dots +p^{ii} _n (z)) =1+r^{ii} _n (z)\]
for $i \in \{1, 2, 3, 4\}, k \in \N_0$. Here $\#^{W}$ denotes the composition in the Weyl calculus and $p^{ii} _k (z) \in S(h^k (m+\lambda)^{-1}, g)$ holds uniformly in $\lambda >0$ for $m(x, \xi) = \langle \xi \rangle^{2}$, $h(x, \xi) = \langle x \rangle^{-1} \langle \xi \rangle^{-1}$, $g= \frac{dx^2}{\langle x \rangle^2} + \frac{d\xi^2}{\langle \xi \rangle^2}$ and $z = \lambda +i\lambda \tan \theta$. Furthermore 
\[\langle x \rangle^{|\alpha|} \langle \xi \rangle^{|\beta|} |\partial^{\alpha} _x \partial^{\beta} _{\xi} r^{ii} _n (x, \xi, z)| \lesssim_{n, k, \eta} h(x, \xi)^{n} m(x, \xi) (\sigma^{W} (H)(x, \xi) +1+\lambda)^{-1}\]
holds for all $(x, \xi) \in \R^{2d}$, $\alpha, \beta \in \N^d _{0}$ satisfying $|\alpha|+|\beta| \le k$, $k \in \N$, $z=\lambda + i \lambda \tan \eta$. This construction is based on the asymptotic expansion
\[a \#^{W} b \sim \sum_{\alpha, \beta} \frac{(-1)^{|\alpha|}}{(2i)^{|\alpha + \beta|} \alpha ! \beta !} (\partial^{\alpha} _x \partial^{\beta} _{\xi} a(x,\xi)) (\partial^{\beta} _x \partial^{\alpha} _{\xi} b(x, \xi))\]
(e.g. see \cite{Ma}) and an induction argument. However this is well-known (see e.g. Proposition 2.12 in \cite{M1}) and we omit the details here. Then we define
\[b^{jj} _n (x, \xi) := \frac{i}{2\pi} \int_{\Gamma} z^{-s/2} (p^{jj} _n (x, \xi, \bar{z}))^* dz .\]
Here $(p^{jj} _n (x, \xi, z))^*$ is a Weyl symbol of $((p^{jj} _n)^{W} (x, D_x, z))^*$ and hence $(p^{jj} _n (x, \xi, z))^* \in S(h^n (m+\lambda)^{-1}, g)$ holds. Then we obtain $b^{jj} _n \in S(m^{-s/2} h^n, g)$ simce
\begin{align*}
\langle x \rangle^{|\alpha|} \langle \xi \rangle^{|\beta|} |\partial^{\alpha} _x \partial^{\beta} _{\xi} b^{jj} _n (x, \xi)| &\lesssim \int_{0}^{\infty} \lambda^{-s/2} h^n (m+ \lambda)^{-1} d\lambda \\
& \lesssim h^n m^{-1} \int_{0}^{m} \lambda^{-s/2} d \lambda +h^n \int_{m}^{\infty} \lambda^{-1-s/2} d\lambda \\
& \lesssim m^{-s/2} h^n
\end{align*}
holds. Now we define
\[a^{jj} _{-s/2} \sim \sum_{n=0}^{\infty} b^{jj} _n \in S(m^{-s/2}, g).\]
By the above argument, we have
\begin{align*}
&(H-z) \diag \{ (p^{jj} _0)^{W} (x, D_x, z) + \dots + (p^{jj} _n)^{W} (x, D_x, z) \} = I+R_n (z) +S_n (z), \\
&R_n (z) = \diag \{(r^{jj} _n)^{W} (x, D_x, z)\}, S_n (z) \in OPS(\langle \xi \rangle (m+\lambda)^{-1} \langle x \rangle^{-2}, g)
\end{align*}
Here $\diag \{a^{jj}\}$ denotes the diagonal matrix whose $(j, j)$ component is $a^{jj}$. Hence we obtain
\[(H-z)^{-1} = \diag \{(p^{jj} _0)^{W} (x, D_x, z) + \dots + (p^{jj} _n)^{W} (x, D_x, z) \} -(R_n (z) +S_n (z))(H-z)^{-1}\]
and by taking its adjoint,
\begin{align}
(H-z)^{-1} = &\diag \{((p^{jj} _0)^{W} (x, D_x, \bar{z}))^* + \dots + ((p^{jj} _n)^{W} (x, D_x, \bar{z}))^* \} \notag \\
& \quad \quad \quad \quad \quad \quad \quad \quad \quad -(H-z)^{-1} (R^* _n (\bar{z}) +S^* _n (\bar{z})) \label{1222129}
\end{align}
holds. By substituting (\ref{1222129}) into Cauchy's integral formula, we obtain
\begin{align*}
H^{-s/2} = & \diag \{(a^{jj} _{-s/2})^{W} (x, D_x)\} + \diag \{(c^{jj} _n)^{W} (x, D_x)\} \\
&  \quad \quad - \frac{1}{2 \pi i} \int_{\Gamma} z^{-s/2} (z-H)^{-1} (R^* _n (\bar{z}) +S^* _n (\bar{z})) dz
\end{align*}
for some $c^{jj} _n \in S(m^{-s/2} h^{n+1}, g)$. For the first and second terms, by the Calder\`on-Vaillancourt theorem, we obtain 
\[\|\langle x \rangle \{(\mathcal{D} + \beta)^2\}^{s/2} V (\diag \{(a^{jj} _{-s/2})^{W} (x, D_x)\} + \diag \{(c^{jj} _n)^{W} (x, D_x)\}) \langle x \rangle\|_{\mathcal{B} (\mathcal{H})} \lesssim 1.\]
For the third term, we compute as
\begin{align*}
&\int_{\Gamma} \langle x \rangle \langle D \rangle^{s} I_4 V z^{-s/2} (z-H)^{-1} (R^* _n (\bar{z}) +S^* _n (\bar{z})) \langle x \rangle dz \\
&= \int_{\Gamma} z^{-s/2} (\langle x \rangle \langle D \rangle^{s} I_4 V  \langle D \rangle^{-s}) \cdot (\langle D \rangle^{s} I_4 \langle H \rangle^{-s/2}) \cdot (\langle H \rangle^{s/2} (z-H)^{-1} \langle H \rangle^{1/2}) \\
& \quad \quad \quad \quad \quad  \cdot (\langle H \rangle^{-1/2} \langle D \rangle I_4) \cdot (\langle D \rangle^{-1} (R^* _n (\bar{z}) +S^* _n (\bar{z})) \langle x \rangle) dz.
\end{align*}
Here, in the integrand, the first factor is bounded on $\mathcal{H}$ by the Calder\`on-Vaillancourt theorem. The second and fourth factors are bounded on $\mathcal{H}$ since we have $(\mathcal{D} + \beta +V)^2 \approx (\mathcal{D} + \beta)^2$. The third factor is bounded on $\mathcal{H}$ uniformly in $z \in \Gamma$ since $H \gtrsim \gamma$ holds. Concerning the fifth factor, by the Calder\`on-Vaillancourt theorem and
\[S_n (z) \in OPS(\langle \xi \rangle (m+\lambda)^{-1} \langle x \rangle^{-2}, g), \quad R_n (z) \in OPS(m(m+\lambda)^{-1} h^{n+1}, g),\]
we have
\[\|\langle D \rangle^{-1} (R^* _n (\bar{z}) +S^* _n (\bar{z})) \langle x \rangle\|_{\mathcal{B} (\mathcal{H})} \lesssim (1+ \lambda)^{-1} \]
for sufficiently large $n \in \N$.
Thus we obtain
\[\left\|\int_{\Gamma} \langle x \rangle \langle D \rangle^{s} I_4 V z^{-s/2} (z-H)^{-1} (R^* _n (\bar{z}) +S^* _n (\bar{z})) \langle x \rangle dz \right\|_{\mathcal{B} (\mathcal{H})} \lesssim 1\]
and hence (\ref{12212324}) is proved.
\end{proof}
\appendix
\section{Strichartz estimate for wave equation}
In this section, we give an abstract perturbation theorem concerning the Strichartz estimates for wave equations and apply it to some equations. The Strichartz estimates for wave equations with potentials are proved in many papers. For example, in \cite{BPST}, inverse square type potentials are treated. Furthermore, in \cite{DF}, \cite{D} and \cite{D2}, wave equations with magnetic potentials are treated. See references therein for more informations. Our method is essentially the same as that in \cite{D} but it can be also applied to higher order or fractional wave equations. Furthermore we also treat the endpoint case, which is not included in the above references. In the following theorem, settings are as in \cite{BM}.
\begin{thm}[non endpoint case] \label{12231549}
Suppose $\mathcal{H}$ is a Hilbert space and self-adjoint operators $H, H_0$ and densely defined closed operators $v_1, v_2$ satisfy
\[D(H) \cup D(H_0) \subset D(v_1) \cap D(v_2) \quad \text{and} \quad \langle Hu,v \rangle=\langle H_0 u,v\rangle+\langle v_2 u, v_1 v\rangle\]
for all $u, v \in D(H) \cap D(H_0)$. Furthermore we assume $H$ and $H_0$ are injective, $v_1$ is $H_0$-smooth, $v_2 P_{\ac} (H)$ is $H$-smooth, $H \ge0$, $H_0 \ge 0$ and $\|H^{1/4} u\|_{\mathcal{H}} \lesssim \|H^{1/4} _0 u \|_{\mathcal{H}}$. Let $\mathcal{A}$ be a Banach space and assume $(\mathcal{H}, \mathcal{A})$ is a Banach couple (see \cite{BM} for its definition). If
\begin{align}
\|H^{\alpha} _0 e^{-it \sqrt{H_0}} u\|_{L^p _t \mathcal{A}} \lesssim \|H^{1/4} _0 u\|_{\mathcal{H}} \label{12231613}
\end{align}
holds for some $\alpha \in \R$ and $p>2$, we obtain
\begin{align}
\|H^{\alpha} _0 e^{-it \sqrt{H}} P_{\ac} (H)u\|_{L^p _t \mathcal{A}} \lesssim \|H^{1/4} _0 u\|_{\mathcal{H}} + \|H^{1/4} _0 P_{\ac} (H)u\|_{\mathcal{H}} \label{12232318}
\end{align}
for the same $p, \alpha$.
\end{thm}
\begin{proof}
By the Duhamel formula we have
\begin{align*}
e^{-it \sqrt{H}} P_{\ac} (H) u = & \cos t \sqrt{H_0} P_{\ac} (H) u -i \frac{\sin t \sqrt{H_0}}{\sqrt{H_0}} \sqrt{H} P_{\ac} (H) u \\
& \quad \quad  - \int_{0}^{t} \frac{\sin (t-s) \sqrt{H_0}}{\sqrt{H_0}} V e^{-is \sqrt{H}} P_{\ac} (H) u ds
\end{align*}
and hence by using (\ref{12231613}), 
\begin{align*}
\|H^{\alpha} _0 e^{-it \sqrt{H}} P_{\ac} (H)u\|_{L^p _t \mathcal{A}} &\lesssim \|H^{1/4} _0 P_{\ac} (H)u\|_{\mathcal{H}} + \|H^{1/4} _0 ({\sqrt{H_0}}^{-1} \sqrt{H}) P_{\ac} (H) u\|_{\mathcal{H}}    \\
& \quad \quad + \left\|H^{\alpha -1/2} _0 \sin t \sqrt{H_0} \int_{0}^{t} \cos s \sqrt{H_0} V e^{-is \sqrt{H}} P_{\ac} (H) u ds \right\|_{L^p _t \mathcal{A}} \\
& \quad \quad +  \left\|H^{\alpha -1/2} _0 \cos t \sqrt{H_0} \int_{0}^{t} \sin s \sqrt{H_0} V e^{-is \sqrt{H}} P_{\ac} (H) u ds \right\|_{L^p _t \mathcal{A}} 
\end{align*}
holds. For the second term, we have
\begin{align*}
\|H^{1/4} _0 ({\sqrt{H_0}}^{-1} \sqrt{H}) P_{\ac} (H) u\|_{\mathcal{H}} \lesssim \|H^{1/4} _0 u\|_{\mathcal{H}}
\end{align*}
since we assume $\|H^{1/4} u\|_{\mathcal{H}} \lesssim \|H^{1/4} _0 u \|_{\mathcal{H}}$. In order to estimate the third term we compute as
\begin{align*}
&\left| \left\langle \int_{0}^{\infty} H^{-1/4} _0 \cos s \sqrt{H_0} V e^{-is \sqrt{H}} P_{\ac} (H) u ds, g \right\rangle \right|  \\
& = \left| \int_{0}^{\infty} \left\langle v_2 e^{-is \sqrt{H}} P_{\ac} (H) u, v_1 H^{-1/4} _0 \cos s \sqrt{H_0} \right\rangle \right| \\
& \lesssim \|v_2 e^{-is \sqrt{H}} P_{\ac} (H) u\|_{L^2 _t \mathcal{H}} \|g\|_{\mathcal{H}}.
\end{align*}
Here, in the last line, we have used the $H_0$-smoothness of $v_1$ and Lemma \ref{1212101}. Then by the duality and (\ref{12231613}), we obtain
\begin{align*}
&\left\|H^{\alpha -1/2} _0 \sin t \sqrt{H_0} \int_{0}^{\infty} \cos s \sqrt{H_0} V e^{-is \sqrt{H}} P_{\ac} (H) u ds \right\|_{L^p _t \mathcal{A}} \\
& \lesssim \left\| \int_{0}^{\infty} H^{-1/4} _0 \cos s \sqrt{H_0} V e^{-is \sqrt{H}} P_{\ac} (H) u ds \right\|_{\mathcal{H}} \\
& \lesssim  \|v_2 e^{-is \sqrt{H}} P_{\ac} (H) u\|_{L^2 _t \mathcal{H}}.
\end{align*}
Since $p>2$, by the Christ-Kiselev lemma (\cite{CK}), we have
\begin{align*}
&\left\|H^{\alpha -1/2} _0 \sin t \sqrt{H_0} \int_{0}^{t} \cos s \sqrt{H_0} V e^{-is \sqrt{H}} P_{\ac} (H) u ds \right\|_{L^p _t \mathcal{A}} \\
& \lesssim \|v_2 e^{-is \sqrt{H}} P_{\ac} (H) u\|_{L^2 _t \mathcal{H}} \lesssim \|H^{1/4} u\|_{\mathcal{H}} \lesssim \|H^{1/4} _0 u \|_{\mathcal{H}}.
\end{align*}
In the last line, we have used the $H$-smoothness of $v_2 P_{\ac} (H)$. Similarly, we obtain
\[\left\|H^{\alpha -1/2} _0 \cos t \sqrt{H_0} \int_{0}^{t} \sin s \sqrt{H_0} V e^{-is \sqrt{H}} P_{\ac} (H) u ds \right\|_{L^p _t \mathcal{A}} \lesssim \|H^{1/4} _0 u \|_{\mathcal{H}} \]
and hence we have (\ref{12232318}).
\end{proof}
In addition to the conditions in Theorem \ref{12231549}, if we have $\|H^{1/4} _0 u\|_{\mathcal{H}} \lesssim \|H^{1/4} u \|_{\mathcal{H}}$,
\begin{align*}
\|H^{1/4} _0 P_{\ac} (H)u\|_{\mathcal{H}} \lesssim \|H^{1/4} P_{\ac} (H)u\|_{\mathcal{H}} \lesssim \|H^{1/4} _0 u\|_{\mathcal{H}} 
\end{align*}
holds. Hence we obtain the following corollary.
\begin{corollary} \label{12232320}
Under the conditions in Theorem \ref{12231549}, we assume $\sigma (H)= \sigma_{\ac} (H) = [0, \infty)$ or $\|H^{1/4} _0 u\|_{\mathcal{H}} \lesssim \|H^{1/4} u \|_{\mathcal{H}}$. Then
\[\|H^{\alpha} _0 e^{-it \sqrt{H}} P_{\ac} (H)u\|_{L^p _t \mathcal{A}} \lesssim \|H^{1/4} _0 u\|_{\mathcal{H}}\]
holds.
\end{corollary}
As our first application of the above corollary, we consider the magnetic Schr\"odinger operator. The following example was proved by D'Ancona \cite{D} but our proof seems a little simpler.
\begin{example} \label{12232330}
Assume $d \ge3$, $A:\R^d \rightarrow \R^d, V:\R^d \rightarrow \R$ satisfy $|A(x)| + |\langle x \rangle V(x)| \lesssim \langle x \rangle^{-1-\epsilon}$, $\langle x \rangle^{1+\epsilon'}A(x) \in \dot{W}^{1/2, 2d}$, $A \in C^0(\R^d)$ for some $0<\epsilon'<\epsilon$. Let $H = -\Delta +A(x)\cdot D + D\cdot A(x) + V(x)$ or $H = (D +A(x))^2 +V(x)$ and $H \ge 0$. Here $D=-i\nabla$. If zero is neither an eigenvalue nor a resonance of $H$, we have
\begin{align}
\||D|^{\frac{1}{q} - \frac{1}{p}} e^{-it \sqrt{H}} u\|_{L^p _t L^q _x} \lesssim \||D|^{1/2}  u\|_{L^2 _x} \label{1226014}
\end{align}
for all non endpoint $\frac{d-1}{2}$ admissible pair $(p, q)$.
\end{example}
\begin{proof}
By the assumptions, $P_{\ac} (H) =I$ holds. We have $\|H^{1/4} u\|_{\mathcal{H}} \lesssim \|H^{1/4} _0 u \|_{\mathcal{H}}$ by the Hardy inequality as before and $H_0$-smoothness of $v_1$ and $H$-smoothness of $v_2$ follow from the proof of Corollary \ref{12131747}.
\end{proof}
The next example concerns the inverse square type potentials. The proof is similar to that of Corollary \ref{12131724} and we omit the details here.
\begin{example} \label{1226007}
Let $H$ be as in Corollary \ref{12131724} and $(p, q)$ be as in the above example. Then (\ref{1226014}) holds.
\end{example}
The next example concerning the critical inverse square potential is similarly proved as Corollary \ref{12162148} by using Theorem \ref{12231549}.
\begin{example} \label{1226018}
Assume $H= -\Delta - \frac{(d-2)^2}{4} |x|^{-2}$ and $(p, q)$ is as in the above examples. Then
\[\||D|^{\frac{1}{q} - \frac{1}{p}} P^{\perp} _{\rad} e^{-it \sqrt{H}} u\|_{L^p _t L^q _x} \lesssim \||D|^{1/2}  u\|_{L^2 _x} \]
holds.
\end{example}
Next we prove the Strichartz estimates for higher order or fractional wave equations. The following two examples are proved by using Theorem \ref{12231549} and the estimates for the free operators as in Theorem \ref{1251816} or \ref{1251821}.
\begin{example} \label{1226029}
Assume $H$ is as in Theorem \ref{1251816} with $m \ge 2$ and $H \ge 0$. Then we have
\[\||D|^{\frac{m-2}{p} + \frac{m}{2}} e^{-it \sqrt{H}} u\|_{L^p _t L^{q, 2} _x} \lesssim \||D|^{m/2}  u\|_{L^2 _x} \]
for all non endpoint $\frac{d}{2}$ admissible pairs $(p, q)$.
\end{example}
\begin{example} \label{1226035}
Assume $H$ is as in Theorem \ref{1251821} with $\sigma > 1$ and $\langle Hu, u \rangle \lesssim \langle H_0 u, u \rangle$ holds. Then 
\[\||D|^{\frac{\sigma -2}{p} + \frac{\sigma}{2}} e^{-it \sqrt{H}} u\|_{L^p _t L^{q, 2} _x} \lesssim \||D|^{\sigma /2}  u\|_{L^2 _x} \]
holds for all non endpoint $\frac{d}{2}$ admissible pairs $(p, q)$.
\end{example}
The next example concerns the spherically averaged Strichartz estimates, which were proved for the fractional Schr\"odinger equations with potentials in \cite{MY1}. Let $L^2 _{\omega} = L^2 (S^{d-1}, d\omega)$, $\mathcal{L}^p _r = L^p (\R_{>0}; r^{d-1} dr)$, $\|f\|_{\mathcal{L}^p _r L^2 _{\omega}} := \|\|f(r \omega)\|_{L^2 _{\omega}} \|_{\mathcal{L}^p _r}$ and we define $\|f\|_{B(\mathcal{L}^p _r L^2 _{\omega})} := \| \{ \|\phi_j (D) f\|_{\mathcal{L}^p _r L^2 _{\omega}} \} \|_{\ell^2}$ for some homogeneous Littlewood-Paley decomposition $\{\phi_j \}$.
\begin{example} \label{1226110}
Let $H= (-\Delta)^{\sigma} +V$ be as in Theorem \ref{1251821} satisfying $\langle Hu, u \rangle \lesssim \langle H_0 u, u \rangle$, $\sigma \in (1, 2)$ and $d \ge 2$. Suppose $(p, q)$ satisfies $2<p, q \le \infty$, $\frac{1}{p} \le (d- \frac{1}{2})(\frac{1}{2} - \frac{1}{q})$, $\frac{1}{p} \ne (d- \frac{1}{2})(\frac{1}{2} - \frac{1}{q})$ if $d=2$. Set $s=-d(\frac{1}{2} - \frac{1}{q})+ \frac{\sigma}{p}$. Then we have
\[\||D|^{s + \frac{\sigma}{2}} e^{-it \sqrt{H}} u\|_{L^p _t B(\mathcal{L}^q _r L^2 _{\omega})} \lesssim \||D|^{\sigma /2}  u\|_{L^2 _x} .\]
\end{example}
\begin{proof}
The corresponding estimates for the free operator are proved in Appendix A in \cite{MY1}. The $H_0$-smoothness of $v_1$ and the $H$-smoothness of $v_2$ are proved in Theorem 1.3 in \cite{MY1}. Hence by using Theorem \ref{12231549}, we have the desired estimates.
\end{proof}
\begin{remark} \label{1226127}
If $V$ and $u$ are radially symmetric functions, we have $\|\phi _j (D) u\|_{\mathcal{L}^q _r L^2 _{\omega}} \approx \|\phi_j (D) u\|_{q}$ and hence $\|u\|_{B(\mathcal{L}^q _r L^2 _{\omega})} \gtrsim \|u\|_{q}$ holds by the Minkowski inequality and the Littlewood-Paley theorem. Therefore, by Example \ref{1226110}, we obtain
\[\||D|^{s + \frac{\sigma}{2}} e^{-it \sqrt{H}} u\|_{L^p _t L^q _x} \lesssim \||D|^{\sigma /2}  u\|_{L^2 _x}.\]
Then, by the Sobolev inequality $\||D|^{s + \frac{\sigma}{2}} u\|_q \gtrsim \||D|^{\frac{\sigma -2}{p} + \frac{\sigma}{2}} u\|_{\tilde{q}}$ for a $\frac{d}{2}$ admissible pair $(p, \tilde{q})$, we obtain
\[\||D|^{\frac{\sigma -2}{p} + \frac{\sigma}{2}} e^{-it \sqrt{H}} u\|_{L^p _t L^{\tilde{q}} _x} \lesssim \||D|^{\sigma /2}  u\|_{L^2 _x}.\]
Thus Example \ref{1226110} is also a refinement of the $L^p _t L^q _x$ type Strichartz estimates.  
\end{remark}
The next example is the inhomogeneous elliptic operators. Concerning the Strichartz estimates for the Schr\"odinger type equations, they are treated in \cite{MY1}.
\begin{example} \label{1226241}
Assume $H=H_0 +V$, $H_0 = \left( \displaystyle \sum_{j=1}^{\sigma} a_j (-\Delta)^{j} \right)^2$, $\sigma \in \N$, $a_j \ge 0$, $a_{\sigma} >0$, $\sigma < d/4$ and $H$ satisfies the assumptions in Theorem \ref{1251821} in which $(-\Delta)^{\sigma}$ is replaced by $H_0$ and $|x|^{2\sigma} V \in L^{\infty}$ is replaced by $|x|^{4\sigma} V \in L^{\infty}$. Furthermore we assume $\langle Hu, u \rangle \lesssim \langle H_0 u, u \rangle$. Then we have 
\[\||D|^{\frac{2\sigma -2}{p} +\sigma} e^{-it \sqrt{H}} u\|_{L^p _t L^{q, 2} _x} \lesssim \|H^{1/4} _0  u\|_{L^2 _x} \]
for all non endpoint $\frac{d}{2}$ admissible pairs $(p, q)$.
\end{example}
\begin{proof}
The Strichartz estimates for the free operator are proved in Appendix A in \cite{MY1}. We decompose $v_1 = |x|^{2\sigma} V$ and $v_2 = |x|^{-2\sigma}$. Then $H_0$-smoothness of $v_1$ and $H$-smoothness of $v_2$ are proved in Theorem 5.2 in \cite{MY1}. Hence by repeating the proof of Theorem \ref{12231549} with easy modifications, we obtain the desired estimates. 
\end{proof}
Next we focus on the endpoint Strichartz estimates.
\begin{thm} \label{1226344}
Suppose $\mathcal{H}$ is a Hilbert space and self-adjoint operators $H, H_0$ and a densely defined closed operator $V$ satisfy
\begin{align*}
D(H) \cup D(H_0) \subset D(V), \quad H=H_0 +V, \quad H, H_0 \ge 0, \quad \|H^{1/4} u\|_{\mathcal{H}} \lesssim \|H^{1/4} _0 u \|_{\mathcal{H}} .
\end{align*}
Let $\mathcal{A}$ be a Banach space and assume $(\mathcal{H}, \mathcal{A})$ is a Banach couple (see \cite{BM} for its definition). We also assume the following homogeneous and inhomogeneous endpoint Strichartz estimates for $H_0$ and the following smoothing estimate for some $\alpha \in \R$:
\begin{align}
&\|H^{\alpha} _0 e^{-it \sqrt{H_0}} u\|_{L^2 _t \mathcal{A}} \lesssim \|H^{1/4} _0 u\|_{\mathcal{H}} \label{1226409} \\
&\left\| H^{\alpha} _0 \int_{0}^{t} \frac{\sin (t-s) \sqrt{H_0}}{\sqrt{H_0}} F(s) ds \right\|_{L^2 _t \mathcal{A}} \lesssim \|H^{-\alpha} _0 F\|_{L^2 _t \mathcal{A}'} \label{1226411} \\
& \|H^{-\alpha} _0 V e^{-it \sqrt{H}} P_{\ac} (H) u\|_{L^2 _t \mathcal{A}'} \lesssim \|H^{1/4} _0 u\|_{\mathcal{H}} . \label{1226413}
\end{align}
Then we obtain
\begin{align}
\|H^{\alpha} _0 e^{-it \sqrt{H}} P_{\ac} (H)u\|_{L^2 _t \mathcal{A}} \lesssim \|H^{1/4} _0 u\|_{\mathcal{H}} + \|H^{1/4} _0 P_{\ac} (H)u\|_{\mathcal{H}} . \label{1226416}
\end{align}
\end{thm}
We note that the same statement as Corollary \ref{12232320} holds for the above theorem.
\begin{proof}
By the Duhamel formula in Theorem \ref{12231549}, we have
\begin{align*}
\|H^{\alpha} _0 e^{-it \sqrt{H}} P_{\ac} (H)u\|_{L^2 _t \mathcal{A}} &\lesssim \|H^{1/4} _0 P_{\ac} (H)u\|_{\mathcal{H}} +\|H^{1/4} _0 u\|_{\mathcal{H}} \\
& \quad \quad + \left\| H^{\alpha} _0 \int_{0}^{t} \frac{\sin (t-s) \sqrt{H_0}}{\sqrt{H_0}} V e^{-is \sqrt{H}} P_{\ac} (H) u ds \right\|_{L^2 _t \mathcal{A}} \\
& \lesssim \|H^{1/4} _0 P_{\ac} (H)u\|_{\mathcal{H}} +\|H^{1/4} _0 u\|_{\mathcal{H}} +  \|H^{-\alpha} _0 V e^{-it \sqrt{H}} P_{\ac} (H) u\|_{L^2 _t \mathcal{A}'} \\
& \lesssim \|H^{1/4} _0 P_{\ac} (H)u\|_{\mathcal{H}} +\|H^{1/4} _0 u\|_{\mathcal{H}} .
\end{align*}
Here, in the first inequality, (\ref{1226409}) is used, in the second inequality, (\ref{1226411}) is used and in the last inequality, (\ref{1226413}) is used.
\end{proof}
The first application is concerned with very short range potentials.
\begin{example} \label{12261413}
Assume $H=-\Delta +V$ satisfies $H \ge 0$, $|V(x)| \lesssim \langle x \rangle^{-2-\epsilon}$, $\langle x \rangle^{1/2 + \epsilon'}V \in \dot{W}^{\frac{1}{d-1}, \frac{2d(d-1)}{3d-1}} \cap L^{\frac{2d(d-1)}{3d-2}}$ for some $\epsilon' > \epsilon$ and zero is a regular point of $H$. Then we have
\[\||D|^{-\frac{1}{d-1}} e^{-it \sqrt{H}} u\|_{L^2 _t L^{\frac{2(d-1)}{d-3}}} \lesssim \|u\|_{\dot{H}^{1/2}}\]
if $d \ge 4$.
\end{example}
\begin{proof}
Since $\langle x \rangle^{-1/2 -\tilde{\epsilon}} |D|^{1/2}$ is $H$-smooth and $\langle Hu, u \rangle \lesssim \langle H_0 u, u \rangle$ holds, it suffices to show the boundedness of $|D|^{\frac{1}{d-1}} V |D|^{-1/2} \langle x \rangle^{1/2 +\tilde{\epsilon}} : L^2 \rightarrow L^{(\frac{2(d-1)}{d-3})'}$ for sufficiently small $\tilde{\epsilon}$. This follows from the following inequality.
\begin{align*}
&\||D|^{\frac{1}{d-1}} V |D|^{-1/2} \langle x \rangle^{1/2 +\tilde{\epsilon}}\|_{2 \rightarrow (\frac{2(d-1)}{d-3})'} \\
&= \||D|^{\frac{1}{d-1}} V \langle x \rangle^{1/2 + \epsilon'} |D|^{-1/2} (|D|^{1/2} \langle x \rangle^{-1/2 - \epsilon'} |D|^{-1/2} \langle x \rangle^{1/2 +\tilde{\epsilon}})\|_{2 \rightarrow (\frac{2(d-1)}{d-3})'} \\
& \lesssim \||D|^{\frac{1}{d-1}} V \langle x \rangle^{1/2 + \epsilon'} |D|^{-1/2}\|_{2 \rightarrow (\frac{2(d-1)}{d-3})'} \\
& \lesssim \||D|^{\frac{1}{d-1}} (\langle x \rangle^{1/2 + \epsilon'} V)\|_{\frac{2d(d-1)}{3d-1}} \||D|^{-1/2}\|_{2 \rightarrow \frac{2d}{d-1}} + \|\langle x \rangle^{1/2 + \epsilon'} V\|_{\frac{2d(d-1)}{3d-2}} \||D|^{\frac{1}{d-1} -1/2}\|_{2 \rightarrow \frac{2d(d-1)}{d^2 -2d+2}} \\
& \lesssim 1.
\end{align*}
Here, in the third line, we have used the $L^2$-boundedness of $|D|^{1/2} \langle x \rangle^{-1/2 - \epsilon'} |D|^{-1/2} \langle x \rangle^{1/2 +\tilde{\epsilon}}$ (see \cite{EGS}), in the fourth line, we have used the fractional Leibniz rule and in the last line, we have used the Sobolev embedding and our assumptions.
\end{proof}
The next example is concerned with higher order operators. 
\begin{example} \label{12261703}
Assume $H$ is as in Theorem \ref{1251816} with $m \ge 2$ and $H \ge 0$. Then we have
\[\||D|^{m-1} e^{-it \sqrt{H}} u\|_{L^2 _t L^{\frac{2d}{d-2}, 2}} \lesssim \||D|^{m/2} u\|_2 .\]
\end{example}
\begin{proof}
Since we can decompose $V =v_1 v_2$, $|v_j (x)| \lesssim \langle x \rangle^{-m-\epsilon}$, we obtain
\begin{align*}
\||D|^{-(m-1)} Ve^{-it \sqrt{H}} u\|_{L^2 _t L^{\frac{2d}{d+2}, 2}} &\lesssim \||D|^{-(m-1)} v_1\|_{2 \rightarrow \frac{2d}{d+2}, 2} \cdot \|v_2 e^{-it \sqrt{H}} u\|_{L^2 _t L^2 _x} \\
& \lesssim \|H^{1/4} u\|_2 \lesssim \||D|^{m/2} u\|_2 ,
\end{align*}
where we have used the Sobolev embedding $|D|^{-(m-1)} : L^{\frac{2d}{2m+d}, 2} \rightarrow L^{\frac{2d}{d+2}, 2}$ and the H\"older inequality $v_1 : L^2 \rightarrow L^{\frac{2d}{2m+d}, 2}$. Thus by Theorem \ref{1226344}, we have the desired estimates.
\end{proof}
For the last example, we consider the fractional operators with the Hardy type potentials. The proof is similar to Example \ref{12261703} and we omit the details.
\begin{example} \label{12261733}
Assume $H$ is as in Theorem \ref{1251821} with $\sigma >1$ and $\langle Hu, u \rangle \lesssim \langle H_0 u, u \rangle$. Then we have
\[\||D|^{\sigma-1} e^{-it \sqrt{H}} u\|_{L^2 _t L^{\frac{2d}{d-2}, 2}} \lesssim \||D|^{\sigma /2} u\|_2 .\]
\end{example}
\section{Strichartz estimate for Klein-Gordon equation}
In this section, as Appendix A, we give an abstract perturbation theorem concerning the Strichartz estimates for the Klein-Gordon equation. See \cite{BC}, \cite{DF}, \cite{D2}, \cite{LSS} for references closely related to this section. We also give an example of the Strichartz estimates for the higher order Klein-Gordon equation. The following theorem is proved by the same way as in the case of wave equations.
\begin{thm}[non endpoint case] \label{12261838}
Suppose $\mathcal{H}$ is a Hilbert space and self-adjoint operators $H, H_0$ and densely defined closed operators $v_1, v_2$ satisfy
\[D(H) \cup D(H_0) \subset D(v_1) \cap D(v_2) \quad \text{and} \quad \langle Hu,v \rangle=\langle H_0 u,v\rangle+\langle v_2 u, v_1 v\rangle\]
for all $u, v \in D(H) \cap D(H_0)$. Furthermore we assume $H+1$ and $H_0 +1$ are injective, $v_1$ is $H_0$-smooth, $v_2 P_{\ac} (H)$ is $H$-smooth, $H+1 \ge0$, $H_0 +1 \ge 0$ and $\|(H+1)^{1/4} u\|_{\mathcal{H}} \lesssim \|(H_0 +1)^{1/4} u \|_{\mathcal{H}}$. Let $\mathcal{A}$ be a Banach space and assume $(\mathcal{H}, \mathcal{A})$ is a Banach couple (see \cite{BM} for its definition). If
\begin{align}
\|(H_0 +1)^{\alpha} e^{-it \sqrt{H_0 +1}} u\|_{L^p _t \mathcal{A}} \lesssim \|(H_0 +1)^{1/4} u\|_{\mathcal{H}} \label{12261843}
\end{align}
holds for some $\alpha \in \R$ and $p>2$, we obtain
\begin{align}
\|(H_0 +1)^{\alpha} e^{-it \sqrt{H+1}} P_{\ac} (H)u\|_{L^p _t \mathcal{A}} \lesssim \|(H_0 +1)^{1/4} u\|_{\mathcal{H}} + \|(H_0 +1)^{1/4} P_{\ac} (H)u\|_{\mathcal{H}} \label{12261844}
\end{align}
for the same $p, \alpha$.
\end{thm}
We give some applications to the Schr\"odinger operator. The proofs are similar to those of corresponding estimates for wave equations in Appendix A so we omit the details.
\begin{example} \label{12261849}
Let $H$ be as in Corollary \ref{12232330} but satisfy $H+1 \ge 0$ instead of $H \ge 0$. Then we have
\begin{align}
\|\langle D \rangle^{\frac{1}{q} - \frac{1}{p}} e^{-it \sqrt{H+1}} u\|_{L^p _t L^q _x} \lesssim \|u\|_{H^{1/2}} \label{12262054}
\end{align}
for all non endpoint $\frac{d-1}{2}$ or $\frac{d}{2}$ admissible pairs $(p, q)$.
\end{example}
\begin{example} \label{12262106}
Let $H$ be as in Corollary \ref{12131724} but satisfy $\langle Hu, u \rangle \lesssim \|u\|^2 _{H^1}$ instead of $\langle Hu, u \rangle \lesssim \|\nabla u\|^2 _2$. Then (\ref{12262054}) holds for the same $p, q$.
\end{example}
\begin{example} \label{12262113}
Let $H=-\Delta - \frac{(d-2)^2}{4} |x|^{-2}$. Then 
\[\|\langle D \rangle^{\frac{1}{q} - \frac{1}{p}} e^{-it \sqrt{H+1}} P^{\perp} _{\rad}u\|_{L^p _t L^q _x} \lesssim \|u\|_{H^{1/2}}\]
holds for all non endpoint $\frac{d-1}{2}$ or $\frac{d}{2}$ admissible pairs $(p, q)$.
\end{example}
Next we consider the higher order Klein-Gordon equations. Set $P_m (\xi) := (|\xi|^{2m} +1)^{1/2}$. Then $P_m (\xi) =1+ \frac{1}{2} |\xi|^{2m} + O(|\xi|^{4m})$ as $\xi \rightarrow 0$ and $P_m (\xi) = |\xi|^m + O(|\xi|^{-m})$ as $\xi \rightarrow \infty$. Hence the Strichartz estimates for $P_m (D)$ differs depending on high or low energy. 
\begin{lemma}[Low energy estimate] \label{12262149}
For all $m \in \N$, $m \ge 2$ and $\frac{d}{2}$ admissible pairs $(p, q)$, we have
\[\left\||D|^{\frac{2(m-1)}{p}} e^{-itP_m (D)} \chi_{\{|D| \lesssim 1\}} u\right\|_{L^p _t L^{q, 2} _x} \lesssim \|u\|_2\]
\end{lemma}
\begin{proof}
We follow the argument in \cite{MY1}. By the Littlewood-Paley theorem, the Minkowski inequality and the real interpolation, we have
\begin{align}
\|\chi_{\{|D| \lesssim 1\}} u\|_{L^{q, 2}} \lesssim \left(\sum_{j \le 1} \|\phi_j (D) \chi_{\{|D| \lesssim 1\}} u \|^2 _{L^{q, 2}} \right)^{1/2} ,\label{1227309}
\end{align}
where $\phi_j$ denotes a homogeneous Littlewood-Paley decomposition. Since we have
\[|\nabla P_m (\xi)| \approx |\xi|^{2m-1}, \quad | \hess P_m (\xi)| \approx |\xi|^{d(2m-2)}\]
on $\{|\xi| \lesssim 1\}$, by Theorem 3.1 in \cite{HHZ}, we obtain
\[\left| \int_{\R^d} e^{ix\xi -itP_m (\xi)} |\xi|^{d(m-1)} \chi_{\{|\xi| \lesssim 1\}} d\xi \right| \lesssim |t|^{-\frac{d}{2}} .\]
Hence we have the dispersive estimate $\||D|^{d(m-1)} e^{-itP_m (D)} \chi_{\{|D| \lesssim 1\}}\|_{1 \rightarrow \infty} \lesssim |t|^{-\frac{d}{2}}$. Since $\||2^{-j} D|^{d(1-m)} \phi_j (D)\|_{\infty \rightarrow \infty} \lesssim 1$ holds, we obtain $\|e^{-itP_m (D)} \phi_j (D)\|_{1 \rightarrow \infty} \lesssim 2^{-jd(m-1)} |t|^{-\frac{d}{2}}$ for $j \le 1$. Thus
\[\|\phi_j (D) e^{-i2^{-2j(m-1)} (t-s)P_m (D)} \phi_j (D)\|_{1 \rightarrow \infty} \lesssim |t-s|^{-\frac{d}{2}}\]
holds and by the Keel-Tao theorem (\cite{KT}), we obtain
\[\|e^{-i2^{-2j(m-1)} tP_m (D)} \phi_j (D)u\|_{L^p _t L^{q, 2} _x} \lesssim \|u\|_2.\]
Then, by changing variables, (\ref{1227309}) and the Minkowski inequality, we obtain the desired estimates.  
\end{proof}
\begin{lemma}[High energy estimate] \label{1227320}
For all $m \in \N$, $m \ge 2$ and $\frac{d}{2}$ admissible pairs $(p, q)$, we have
\[\left\||D|^{\frac{m-2}{p}} e^{-itP_m (D)} \chi_{\{|D| \gtrsim 1\}} u\right\|_{L^p _t L^{q, 2} _x} \lesssim \|u\|_2\]
\end{lemma}
\begin{proof}
Since we have
\[|\nabla P_m (\xi)| \approx |\xi|^{m-1}, \quad | \hess P_m (\xi)| \approx |\xi|^{d(m-2)}\]
on $\{|\xi| \gtrsim 1\}$, by the similar argument as above, we obtain $\|e^{-itP_m (D)} \phi_j (D)\|_{1 \rightarrow \infty} \lesssim 2^{-jd(\frac{m}{2} -1)} |t|^{-\frac{d}{2}}$ for $j \ge 1$. Then the rest of the argument is similar so we omit the details here.
\end{proof}
Now we give an application of Theorem \ref{12261838} to Lemma \ref{1227320}. The proof is similar to the case of higher order wave equations so we omit it here.
\begin{example} \label{1227335}
Assume $H=(-\Delta)^m +V$, $m \ge 2$ is as in Theorem \ref{1251816} or \ref{1251821} and satisfies $H+1 \ge 0$ and $\sigma (H) = \sigma_{\ac} (H)$. Then
\[\|\langle D \rangle^{\frac{m-2}{p} + \frac{m}{2}} \chi_{\{|D| \gtrsim 1\}} e^{-it \sqrt{H+1}} u\|_{L^p _t L^{q, 2} _x} \lesssim \|\langle D \rangle^{m/2} u\|_2\]
holds for all non endpoint $\frac{d}{2}$ admissible pairs $(p, q)$.
\end{example}
Finally we consider the endpoint Strichartz estimates. The following theorem is proved similarly as the corresponding one for wave equations (see Appendix A). 
\begin{thm} \label{1227353}
Suppose $\mathcal{H}$ is a Hilbert space and self-adjoint operators $H, H_0$ and a densely defined closed operator $V$ satisfy
\begin{align*}
&D(H) \cup D(H_0) \subset D(V), \quad H=H_0 +V, \\
& H+1, H_0 +1 \ge 0, \quad \|(H+1)^{1/4} u\|_{\mathcal{H}} \lesssim \|(H_0 +1)^{1/4} u \|_{\mathcal{H}} .
\end{align*}
Let $\mathcal{A}$ be a Banach space and assume $(\mathcal{H}, \mathcal{A})$ is a Banach couple (see \cite{BM} for its definition). We also assume the following homogeneous and inhomogeneous endpoint Strichartz estimates for $H_0$ and the following smoothing estimate for some $\alpha \in \R$:
\begin{align*}
&\|(H_0 +1)^{\alpha} e^{-it \sqrt{H_0 +1}} u\|_{L^2 _t \mathcal{A}} \lesssim \|(H_0 +1)^{1/4} u\|_{\mathcal{H}} \\
&\left\| (H_0 +1)^{\alpha} \int_{0}^{t} \frac{\sin (t-s) \sqrt{H_0 +1}}{\sqrt{H_0 +1}} F(s) ds \right\|_{L^2 _t \mathcal{A}} \lesssim \|(H_0 +1)^{-\alpha} F\|_{L^2 _t \mathcal{A}'} \\
& \|(H_0 +1)^{-\alpha} V e^{-it \sqrt{H+1}} P_{\ac} (H) u\|_{L^2 _t \mathcal{A}'} \lesssim \|(H_0 +1)^{1/4} u\|_{\mathcal{H}} . 
\end{align*}
Then we obtain
\begin{align*}
\|(H_0 +1)^{\alpha} e^{-it \sqrt{H+1}} P_{\ac} (H)u\|_{L^2 _t \mathcal{A}} \lesssim \|(H_0 +1)^{1/4} u\|_{\mathcal{H}} + \|(H_0 +1)^{1/4} P_{\ac} (H)u\|_{\mathcal{H}} . 
\end{align*}
\end{thm}
We give some examples concerning very short range potentials.
\begin{example} \label{1227416}
Assume $H=-\Delta +V$ satisfy $|\partial^{\alpha} _x V(x)| \lesssim \langle x \rangle^{-2-\epsilon}$ for all $\alpha \in \N^d _0$, $H+1 \ge 0$ and zero is neither an eigenvalue nor a resonance of $H$. Then we have
\[\|\langle D \rangle^{-\frac{1}{d}} e^{-it \sqrt{H+1}} P_{\ac} (H) u\|_{L^2 _t L^{\frac{2d}{d-2}} _x} \lesssim \|u\|_{H^{1/2}} + \|P_{\ac} (H)u\|_{H^{1/2}}.\]
\end{example}
\begin{proof}
By the H\"older inequality, $\langle x \rangle^{-1-\epsilon'} :L^2 \rightarrow L^{\frac{2d}{d+2}}$ is a bounded operator and $A:= \langle D \rangle^{\frac{1}{d}} V \langle D \rangle^{-1/2} \langle x \rangle^{2 +\epsilon} : L^{\frac{2d}{d+2}} \rightarrow L^{\frac{2d}{d+2}}$ is also bounded since the symbol of $A$ is in $S^0 _{1, 0}$. Then, by Theorem \ref{1227353} and the Kato smoothing estimate, we obtain the desired estimates.  
\end{proof}
\begin{example} \label{1227439}
Assume $H=-\Delta +V$ satisfy $|\partial^{\alpha} _x V(x)| \lesssim \langle x \rangle^{-(1+\frac{d}{d-1})-\epsilon}$ for all $\alpha \in \N^d _0$, $H+1 \ge 0$ and zero is neither an eigenvalue nor a resonance of $H$. Then we have
\[\|\langle D \rangle^{-\frac{1}{d-1}} e^{-it \sqrt{H+1}} P_{\ac} (H) u\|_{L^2 _t L^{\frac{2(d-1)}{d-3}} _x} \lesssim \|u\|_{H^{1/2}} + \|P_{\ac} (H)u\|_{H^{1/2}}.\]
\end{example}
\begin{proof}
The boundedness of $\langle D \rangle^{\frac{1}{d-1}} V \langle D \rangle^{-1/2} \langle x \rangle^{1 +\epsilon'} :L^2 \rightarrow L^{(\frac{2(d-1)}{d-3})'}$ follows from the boundedness of $\langle x \rangle^{-(\frac{d}{d-1})-\epsilon'} :L^2 \rightarrow L^{(\frac{2(d-1)}{d-3})'}$ and $\langle D \rangle^{\frac{1}{d-1}} V \langle D \rangle^{-1/2} \langle x \rangle^{1 +\frac{d}{d-1} + 2\epsilon'} :L^{(\frac{2(d-1)}{d-3})'} \rightarrow L^{(\frac{2(d-1)}{d-3})'}$ by the same reason as above. Then we obtain the desired estimates by repeating the above proof.
\end{proof}

\section*{Acknowledgement}
The author thanks his supervisor Kenichi Ito for valuable comments and discussions. He is partially supported by FoPM, WINGS Program, the University of Tokyo.

\end{document}